\documentclass[notitlepage,11pt,reqno]{amsart}
\usepackage[foot]{amsaddr}
\usepackage[numbers,sort,compress]{natbib}
\usepackage{amssymb,nicefrac,bm,upgreek,mathtools,verbatim,pdf14}
\usepackage[final]{hyperref}
\usepackage[mathscr]{eucal}
\usepackage{dsfont}
\usepackage{graphicx}
\usepackage[margin=10pt,font=small,labelfont=bf]{caption}
\usepackage[normalem]{ulem}
\usepackage[margin=1in]{geometry}
\allowdisplaybreaks

\usepackage{color}
\definecolor{dmagenta}{rgb}{.4,.1,.5}
\definecolor{dblue}{rgb}{.0,.0,.5}
\definecolor{dred}{rgb}{.6,.0,.0}
\definecolor{dgreen}{rgb}{.0,.5,.0}

\newtheorem{lemma}{Lemma}[section]
\newtheorem{theorem}{Theorem}[section]
\newtheorem{proposition}{Proposition}[section]
\newtheorem{corollary}{Corollary}[section]

\theoremstyle{definition}
\newtheorem{definition}{Definition}[section]

\theoremstyle{remark}
\newtheorem{remark}{Remark}[section]
\numberwithin{equation}{section}

\hypersetup{
  colorlinks=true,
  citecolor=dblue,
  linkcolor=dblue,
  frenchlinks=false,
  pdfborder={0 0 0},
  naturalnames=false,
  hypertexnames=false,
  breaklinks}
\usepackage[capitalize,nameinlink]{cleveref}

\crefname{section}{Section}{Sections}
\crefname{subsection}{Section}{Sections}
\crefname{condition}{Condition}{Conditions}
\crefname{hypothesis}{Hypothesis}{Conditions}
\crefname{assumption}{Assumption}{Assumptions}
\crefname{lemma}{Lemma}{Lemmas}

\Crefname{figure}{Figure}{Figures}

\crefformat{equation}{\textup{#2(#1)#3}}
\crefrangeformat{equation}{\textup{#3(#1)#4--#5(#2)#6}}
\crefmultiformat{equation}{\textup{#2(#1)#3}}{ and \textup{#2(#1)#3}}
{, \textup{#2(#1)#3}}{, and \textup{#2(#1)#3}}
\crefrangemultiformat{equation}{\textup{#3(#1)#4--#5(#2)#6}}%
{ and \textup{#3(#1)#4--#5(#2)#6}}{, \textup{#3(#1)#4--#5(#2)#6}}%
{, and \textup{#3(#1)#4--#5(#2)#6}}

\Crefformat{equation}{#2Equation~\textup{(#1)}#3}
\Crefrangeformat{equation}{Equations~\textup{#3(#1)#4--#5(#2)#6}}
\Crefmultiformat{equation}{Equations~\textup{#2(#1)#3}}{ and \textup{#2(#1)#3}}
{, \textup{#2(#1)#3}}{, and \textup{#2(#1)#3}}
\Crefrangemultiformat{equation}{Equations~\textup{#3(#1)#4--#5(#2)#6}}%
{ and \textup{#3(#1)#4--#5(#2)#6}}{, \textup{#3(#1)#4--#5(#2)#6}}%
{, and \textup{#3(#1)#4--#5(#2)#6}}

\crefdefaultlabelformat{#2\textup{#1}#3}
\newcommand{\cE}{{\mathcal{E}}}  
\newcommand{\eom}{{\mathscr{G}}} 
\newcommand{\cG}{{\mathcal{G}}}  
\newcommand{\sH}{{\mathscr{H}}}  
\newcommand{\cI}{{\mathcal{I}}}  
\newcommand{\cJ}{{\mathcal{J}}}  
\newcommand{\cK}{{\mathcal{K}}}  
\newcommand{\Lg}{\mathcal{L}}    
\newcommand{\cL}{{\mathscr{L}}}  

\newcommand{\cP}{{\mathcal{P}}}  
\newcommand{\cS}{{\mathcal{S}}}  
\newcommand{\sS}{{\mathscr{S}}}  
\newcommand{\Lyap}{{\mathcal{V}}}  
\newcommand{\sX}{{\mathscr{X}}}  
\newcommand{\cX}{{\mathcal{X}}}  
\newcommand{\sZ}{{\mathscr{Z}}} 
\newcommand{\cZn}{{\mathcal{Z}^n}} 
\newcommand{\hcZn}{{\Hat{\mathcal{Z}}^n}} 
\newcommand{\bcZn}{{\Breve{\mathcal{Z}}^n}} 
\newcommand{\fZ}{{\mathfrak{Z}}} 

\newcommand{\RR}{\mathds{R}}
\newcommand{\NN}{\mathds{N}}
\newcommand{\ZZ}{\mathds{Z}}
\newcommand{\RI}{\mathds{R}^{I}}
\newcommand{\Rd}{\mathds{R}^{d}}
\DeclareMathOperator{\Exp}{\mathbb{E}}
\DeclareMathOperator{\Prob}{\mathbb{P}}
\newcommand{\D}{\mathrm{d}}
\newcommand{\E}{\mathrm{e}}

\newcommand{\Act}{{\mathbb{U}}}
\newcommand{\Uadm}{\mathfrak{U}}
\newcommand{\Usm}{\mathfrak{U}_{\mathrm{SM}}}
\newcommand{\Ussm}{\mathfrak{U}_{\mathrm{SSM}}}

\newcommand{\sF}{\mathfrak{F}}  
\newcommand{\Ind}{\mathds{1}}   
\newcommand{\Cc}{\mathcal{C}}   

\newcommand{\abs}[1]{\lvert#1\rvert}
\newcommand{\norm}[1]{\lVert#1\rVert}
\newcommand{\babs}[1]{\bigl\lvert#1\bigr\rvert}
\newcommand{\Babs}[1]{\Bigl\lvert#1\Bigr\rvert}
\newcommand{\babss}[1]{\biggl\lvert#1\biggr\rvert}
\newcommand{\bnorm}[1]{\bigl\lVert#1\bigr\rVert}

\newcommand{\transp}{^{\mathsf{T}}}
\newcommand{\df}{\coloneqq}

\DeclareMathOperator*{\diag}{diag}

\newcommand{\order}{{\mathscr{O}}}
\newcommand{\sorder}{{\mathfrak{o}}}
\newcommand{\grad}{\nabla}

\usepackage{accents}
\newlength{\dhatheight}
\newcommand{\doublehat}[1]{%
    \settoheight{\dhatheight}{\ensuremath{\Hat{#1}}}%
    \addtolength{\dhatheight}{-0.30ex}%
    \Hat{\vphantom{\rule{1pt}{\dhatheight}}%
    \smash{\Hat{#1}}}}
\begin{document}

\title[Asymptotic Average Optimality for Large-Scale Parallel Server Networks]
{Infinite Horizon Asymptotic Average Optimality\\
for Large-Scale Parallel Server Networks}

\author[Ari Arapostathis]{Ari Arapostathis$^\dag$}
\address{$^\dag$ Department of Electrical and Computer Engineering\\
The University of Texas at Austin, EER~7.824,
Austin, TX~~78712}
\email{ari@ece.utexas.edu}
\author[Guodong Pang]{Guodong Pang$^\ddag$}
\address{$^\ddag$ The Harold and Inge Marcus Dept.
of Industrial and Manufacturing Eng.\\
College of Engineering\\
Pennsylvania State University\\
University Park, PA~~16802}
\email{gup3@psu.edu} 

\date{\today}

\begin{abstract}
We study infinite-horizon asymptotic average optimality for parallel server networks
with multiple classes of jobs and multiple server pools in the  Halfin--Whitt  regime. 
Three control formulations are considered: 1) minimizing the queueing and idleness cost,
2) minimizing the queueing cost under a 
constraints on idleness at each server pool, and
3)  fairly allocating the idle servers among different server pools.
For the third problem, we consider a class of
\emph{bounded-queue, bounded-state} (BQBS)  stable networks,
in which any moment of the state is bounded by that of the queue only
(for both the limiting diffusion and diffusion-scaled state processes). 
We show that the optimal values for the diffusion-scaled state processes converge
to the corresponding values of the ergodic control problems for the limiting diffusion.  
We present a family of state-dependent Markov balanced saturation
policies (BSPs) that stabilize the controlled diffusion-scaled state processes.
It is shown that under these policies, the diffusion-scaled
state process is exponentially ergodic, provided that
at least one class of jobs has a positive  abandonment rate.
We also establish useful moment bounds, and study
the ergodic properties of the diffusion-scaled state
processes, which play a crucial role in  proving the asymptotic optimality.
\end{abstract}



\maketitle


\section{Introduction}

Large-scale parallel server networks are used to model various service,
manufacturing and telecommunications systems; see, e.g., 
\cite{AAM,  AIMMTYT, CY01, GKM03, Hall,   KY14, MMR, PY11, Shietal, TX13, WW02a,
WW02, Williams16}.  
We consider multiclass multi-pool networks operating in the Halfin--Whitt (H--W) regime,
where the demand of each class and the numbers of servers in each pool get large
simultaneously in an appropriate manner so that the system becomes critically
loaded while the service and abandonment rates are fixed. 
We study optimal control problems of such networks under the infinite-horizon expected
average (ergodic) cost criteria, since steady-state performance measures are among
the most important metrics to understand the system dynamics. 
Specifically, we consider the following unconstrained and constrained ergodic
control problems (see \cref{S3.1,S4.2}):
(P1) minimizing the queueing and idleness cost,
(P2) minimizing the queueing cost while imposing a dynamic constraint on the
idleness of each server pool (e.g., requiring that
the long-run average idleness does not exceed a
given threshold), and
(P3) minimizing the queueing cost while requiring fairness on idleness
(e.g., the average idleness of each server pool is a fixed proportion of the
total average idleness of all server pools).
The scheduling policy determines the allocation of service capacity to each
class at each time.
We consider only work conserving scheduling policies that are
non-anticipative and preemptive.

In \cite{ABP14} and \cite{AP15}, we have studied the corresponding ergodic control
problems (P1$'$)--(P2$'$) for the limiting diffusions arising from such networks
(see \cref{S3.2}).
Problem~(P1$'$) for multiclass networks (``V'' networks) was studied in
\cite{ABP14}, where a comprehensive study of the ergodic control problem
for a broader class of diffusions as well as asymptotic optimality
results for the network model were presented.
In \cite{AP15} we have shown that problem (P1$'$) and (P2$'$) are well-posed
for multiclass multi-pool networks, and presented a full characterization
of optimality for the limiting diffusion.
We also provided important insights on the stabilizability\footnote{We say that
a control (policy) is stabilizing, if it results
in a finite value for the optimization criterion.}
of the controlled diffusion by employing a leaf elimination algorithm,
which is used to derive an explicit expression for the drift.
We addressed problem (P3) for the `N' network model,
where we studied its well-posedness, characterized the optimal solutions, and
established asymptotic optimality in \cite{AP16}. 
For this particular network topology, the fairness constraint requires that the
long-run average idleness of the two server pools satisfies a fixed ratio condition.

In this paper we establish the asymptotic optimality for
the ergodic control problems (P1)--(P3).
In other words, we show that the optimal
values for the
diffusion-scaled state processes converge to the corresponding values for the
limiting diffusion.
The main challenge lies in understanding the recurrence
properties of the diffusion-scaled state
processes for multiclass multi-pool networks in the H--W regime.
Despite the recent studies on stability of multiclass multi-pool networks under
certain scheduling policies
\cite{stolyar-yudovina-12a, stolyar-yudovina-12b, stolyar-15},
the existing results are not sufficient for our purpose.
The difficulty is particularly related to the so-called ``joint work conservation"
(JWC) condition, which requires that no servers are idling unless all the queues
are empty, and
which plays a key role in the derivation of the limiting diffusion, and
the study of  discounted control problems in \cite{Atar-05a, Atar-05b};
see  \cref{S2.2} for a detailed discussion. 
For the limiting diffusion, the JWC condition holds over the entire state space;
however, for the diffusion-scaled state process in the $n^{\rm th}$ system
($n$ is the scaling parameter), it holds only in a bounded subset of the state space.
As a consequence, a stabilizing control\footnote{To avoid confusion, `control'
always refers to a
control strategy for the limiting diffusion, while `policy' refers to
a scheduling strategy for the pre-limit model.}
for the limiting diffusion cannot be
directly translated to a scheduling policy for the $n^{\rm th}$ system which
stabilizes the diffusion-scaled state process in the H--W regime. 

Our first main contribution addresses the above mentioned critical issue
of stabilizability
of the diffusion-scaled state processes.
We have identified a family of stabilizing policies for multiclass multi-pool networks,
which we refer to as the ``Balanced Saturation Policies" (BSPs) (see \cref{D-SDPf}). 
Such a policy strives to keep the state process for each class `close' to the
corresponding steady state quantity, which is state-dependent and dynamic. 
The specific stabilizing policy for the `N' network presented
in \cite{AP16} belongs to this
family of BSPs. 
We show that if the abandonment parameter is positive for at least one class,
then the diffusion-scaled state processes are exponentially stable under
any BSP (see \cref{P5.1}).

In addition to the ergodicity properties,
moment bounds are also essential
to prove asymptotic optimality. 
An important implication of the exponential ergodicity property proved in
\cite[Theorem~4.2]{AP15} is that the controlled diffusion satisfies a very
useful moment bound, see \cref{E-BQYBS}, namely,
 that any moment (higher than first order) of the state is
controlled by the corresponding moments of the queue and idleness.
This moment bound is also shown for the $n^{\rm th}$ system (\cref{P6.1}).
In studying the moment bounds, we have identified an important class of multiclass
multi-pool networks, which we refer to as \emph{bounded-queue, bounded-state} (BQBS)
stable networks (\cref{S4}).
The limiting diffusion of this class of networks has the following important property.
Any moment (higher than first order) of the state is controlled by
the corresponding moment of the queue alone (see \cref{P6.2}). 
The class of BQBS stable networks contains many interesting examples,
including networks with a single dominant class (see \cref{fig-networks}),
and networks with certain parameter constraints, e.g.,
service rates that are only pool-dependent.
It is worth noting that the  ergodic control problem with fairness constraints
for the general multiclass multi-pool networks may not be well-posed.
However, the problems (P3)
and (P3$'$) are well-posed for the BQBS stable networks. 
In addition, the problems (P1) and (P1$'$) which only penalize the queueing
cost are also well-posed for the BQBS stable networks. 

The proof of asymptotic optimality involves the convergence of the value functions,
specifically, establishing the lower and upper bounds
(see \cref{T3.1,T3.2,T4.2}). 
In establishing the lower bound, the arguments are analogous to those for the `N' network
in \cite{AP16} and the `V' network in \cite{ABP14}.
This involves proving the tightness of the mean empirical measures
of the diffusion-scaled state process, controlled under
some eventually JWC scheduling policy (see \cref{DEJWC}), and also
showing that any limit of these empirical measures is an ergodic occupation
measure for the limiting diffusion model (see \cref{L6.1}). 

The proof of the upper bound is the most challenging. 
We utilize the following important property which arises
from a spatial truncation technique
for all three problems (P1$'$)--(P3$'$): there exists a continuous precise
stationary Markov control $\Bar{v}_{\epsilon}$ which is $\epsilon$-optimal
for the limiting diffusion control problem,  and under which
the diffusion is exponentially ergodic 
(see  \cref{L7.1}, \cref{C7.1}, and \cite[Theorem~4.2]{ABP14}).  
For the $n^{\rm th}$ system, we construct a concatenated admissible
policy in the following manner. In the JWC conservation region, we apply a
scheduling policy constructed canonically from the Markov control $\Bar{v}_{\epsilon}$
(see \cref{D6.1}), while outside the JWC region, we apply a fixed BSP.
We also show that that under this concatenated policy the
diffusion-scaled state process is exponentially ergodic,
and its mean empirical measures converge to the ergodic occupation
measure of the limiting diffusion associated with the control $\Bar{v}_{\epsilon}$. 

\vspace{-5pt}
\subsection{Literature review}

There is an extensive literature on scheduling control of multiclass multi-pool
networks in the H--W regime.
For the infinite-horizon discounted criterion, Atar \cite{Atar-05a, Atar-05b} first
studied the unconstrained scheduling control problem under a set of conditions
on the network structure, the system parameters, and the running cost function
(Assumptions~2 and 3 in \cite{Atar-05b}).
Atar et al. \cite{Atar-09} further investigated simplified models with service rates
that either only class-dependent, or pool-dependent.
Gurvich and Whitt \cite{GW09, GW09b, GW10} studied
queue-and-idleness-ratio
controls for multiclass multi-pool networks,
by proving a state-space-collapse (SSC) property
under suitable conditions on the network structure and system parameters
(Theorems~3.1 and 5.1 in \cite{GW09}). 
For finite-horizon cost criteria,  Dai and Tezcan
\cite{dai-tezcan-08, dai-tezcan-11} studied scheduling
controls of multiclass multi-pool networks,
also by proving an SSC property under certain assumptions.

There has also been a lot of activity on ergodic control of multiclass multi-pool
networks in the H--W regime,
in addition to \cite{ABP14,AP15,AP16} mentioned earlier.
For the inverted `V' model, Armony \cite{Armony-05} has shown that the
fastest-server-first policy is asymptotically optimal for minimizing the
steady-state expected queue length and waiting time,
and Armony and Ward \cite{AW-10} have shown that  a threshold policy is
asymptotically optimal for minimizing
the expected queue length and waiting time subject to a ``fairness" constraint
on the workload division. 
For multiclass multi-pool networks, Ward and Armony \cite{WA-13} have studied blind
fair routing policies, and  used simulations to validate their performance, and
compared them with non-blind policies derived from the limiting diffusion
control problem.
Biswas \cite{Biswas-15} recently studied a specific multiclass multi-pool network with
``help" where each server pool has a dedicated stream of a customer class,
and can help with other customer classes only when it has idle servers. 
For this network model, the control policies may not be work-conserving,
and the associated controlled diffusion has a uniform stability property,
which is not satisfied for general multiclass multi-pool networks.
 
This work contributes to the understanding of the stability of multiclass multi-pool 
 networks in the H-W regime. 
 Gamarnik and Stolyar \cite{gamarnik-stolyar} studied the tightness of the stationary distributions
 of the diffusion-scaled state processes under any work conserving 
scheduling policy for the `V' network, while ergodicity properties for the limiting diffusion under constant
Markov controls are established in  \cite{dieker-gao, APS17}. 
We refer the reader to  \cite{stolyar-yudovina-12b, stolyar-yudovina-12a, stolyar-15} for 
the stability analysis of a load balancing scheduling
policy, ``longest-queue freest-server" (LQFS-LB), 
and a \emph{leaf activity priority} policy, for multiclass multi-pool networks. 
For the `N' network with no abandonment,  Stolyar \cite{stolyar-14} studied the stability of
a static priority scheduling policy.
\subsection{Organization of the paper}
In the subsection which follows we summarize the notation used in the paper.
In \cref{S-MM}, we describe the model and the scheduling control problems,
and in \cref{S2.2}, we discuss the JWC condition.
In \cref{S2.3}, we state some basic properties of the diffusion-scaled
processes and the control parameterization, which leads to the diffusion limit.
In \cref{S2.4}, we review some relevant properties of the limiting
diffusion from \cite{AP15}. 
We state the control objectives of the problems (P1) and (P2) in \cref{S3.1},
and the corresponding diffusion control problems (P1$'$) and (P2$'$) in
\cref{S3.2}, and summarize the asymptotic optimality results in
\cref{S3.3}. 
In \cref{S4} we describe the BQBS stable networks and study the fairness
problems (P3) and (P3$'$). 
In \cref{S5}, we introduce the family of stabilizing BSPs, and show that
under these we have exponential stability. 
In \cref{S6}, we focus on the ergodic properties of the
$n^{\rm th}$ system, including the moment bounds, convergence of mean empirical
measures and a stability preserving property in the JWC region. 
In \cref{S-LUB}, we complete the proofs of the lower and upper bounds
of the three problems. 
We conclude in \cref{S8}. 

\subsection{Notation}
The symbol $\RR$ denotes the field of real numbers,
and $\RR_{+}$ and $\NN$ denote the sets of nonnegative
real numbers and natural numbers, respectively.
The minimum (maximum) of two real numbers $a$ and $b$,
is denoted by $a\wedge b$ ($a\vee b$).
Define $a^{+}\df a\vee 0$ and $a^{-}\df-(a\wedge 0)$. 
The integer part of a real number $a$ is denoted by $\lfloor a\rfloor$.
We also let $e\df (1,\dotsc,1)\transp$.

For a set $A\subset\Rd$, we use
$\Bar A$, $A^{c}$, and $\Ind_{A}$ to denote the closure,
the complement, and the indicator function of $A$, respectively.
A ball of radius $r>0$ in $\Rd$ around a point $x$ is denoted by $B_{r}(x)$,
or simply as $B_{r}$ if $x=0$.
The Euclidean norm on $\Rd$ is denoted by $\abs{\,\cdot\,}$,
$x\cdot y$ denotes the inner product of $x,y\in\RR^{d}$,
and $\norm{x}\df \sum_{i=1}^{d}\abs{x_{i}}$.
 
We let $\Cc^{\infty}_{c}(\RR^{d})$ denote the set of smooth real-valued functions
on $\RR^d$ with compact support.
For a Polish space $\cX$, we denote by $\cP(\cX)$ the space of
probability measures on the Borel subsets of $\cX$ under the Prokhorov topology.
For $\nu\in\cP(\cX)$ and a Borel measurable map $f\colon\cX\to\RR$,
we often use the abbreviated notation
$\nu(f)\df \int_{\cX} f\,\D{\nu}\,.$
The quadratic variation of a square integrable martingale is
denoted by $\langle\,\cdot\,,\cdot\,\rangle$.
For any path $X(\cdot)$ of a c\`adl\`ag process,
we use the notation $\Delta X(t)$ to denote the jump at time $t$.

\section{The Model}
All random variables introduced below are defined on a complete
probability space 
$(\Omega,\mathfrak{F},\Prob)$ and $\Exp$ denotes the associated expectation
operator.

\subsection{The multiclass multi-pool network model} \label{S-MM}

We consider a sequence of network systems with the associated variables,
parameters and processes indexed by $n$. 
Each of these, is a multiclass multi-pool Markovian network with $I$
classes of customers
and $J$ server pools, labeled as $1,\dots, I$ and $1,\dots, J$, respectively.
Let $\cI = \{1,\dots,I\}$ and $\cJ = \{1, \dots, J\}$. 
Customers of each class form their own queue and are served in the
first-come-first-served (FCFS) service discipline.
The buffers of all classes are assumed to have infinite capacity.
Customers can abandon/renege while waiting in queue. 
Each class of customers can be served by a subset of server pools,
and each server pool can serve a subset of customer classes.
We let $\cJ(i) \subset \cJ$, denote the subset of server pools that can serve class
$i$ customers, and $\cI(j) \subset \cI$ the subset
of customer classes that can be served by server pool $j$.
We form a bipartite graph $\cG = (\cI\cup \cJ, \cE)$ with a set
of edges defined by $\cE = \{(i,j)\in\cI\times\cJ\colon j\in\cJ(i)\}$, and
use the notation $i \sim j$, if $(i,j)\in\cE$,
and $i \nsim j$, otherwise.
We assume that the graph $\cG$ is a tree. 

For each $j \in \cJ$, let $N_j^{n}$ be the number of servers
(statistically identical) in server pool $j$. Set $N^n = (N^n_j)_{j \in \cJ}$. 
Customers of class $i \in \cI$ arrive according to a Poisson
process with rate $\lambda^{n}_i>0$,
and have class-dependent exponential abandonment rates $\gamma_i^{n} \ge 0$.
These customers are served at an exponential rate
$\mu_{ij}^{n}>0$ at server pool $j$, if $i \sim j$, and 
we set $\mu_{ij}^{n}=0$, if $i \nsim j$.
Thus, the set of edges $\cE$ can thus be written as
$\cE=\bigl\{(i, j) \in \cI\times \cJ\;\colon\, \mu_{ij}^{n} >0\bigr\}$.
We assume that the customer arrival, service, and abandonment processes of
all classes are mutually independent.
We define
\begin{equation*}
\RR^{\cG}_+ \;\df\; \bigl\{\xi=[\xi_{ij}]\in\RR^{I \times J}_+\colon
\xi_{ij}=0~~\text{for~}i\nsim j\bigr\}\,,
\end{equation*}
and analogously define $\ZZ^{\cG}_+$.

\subsubsection{The Halfin--Whitt regime} 
We study these multiclass multi-pool networks in the Halfin--Whitt regime
(or the Quality-and-Efficiency-Driven (QED) regime), where the arrival
rates of each class and the numbers of servers of each server pool grow
large as $n \to \infty$ in such a manner that the system becomes
critically loaded.
Throughout the paper, the set of parameters is assumed to satisfy the following. 

\noindent
\textsl{Parameter Scaling.}
There exist positive constants $\lambda_i$ and  $\nu_j$,
nonnegative constants $\gamma_i$ and  $\mu_{ij}$, with $\mu_{ij}>0$
for $i\sim j$ and $\mu_{ij}=0$ for $i\nsim j$,  and constants 
$\Hat{\lambda}_i$, $\Hat{\mu}_{ij}$ and $\Hat{\nu}_j$, such that
the following limits exist as $n\to\infty$.
\begin{equation}\label{HWpara}
\frac{\lambda^{n}_{i} - n \lambda_{i}}{\sqrt{n}} \;\to\;\Hat{\lambda}_{i}\,,\qquad
{\sqrt{n}}\,(\mu^{n}_{ij} - \mu_{ij}) \;\to\;\Hat{\mu}_{ij}\,,
\qquad \frac{N^{n}_{j} -  n\nu_{j}}{\sqrt{n}} \;\to\;  \Hat{\nu}_j \,, \qquad
\gamma_{i}^{n} \;\to\;\gamma_{i}\,.
\end{equation}

\noindent
\textsl{Fluid scale equilibrium.}
We assume that the linear program (LP) given by
\begin{align*}
\text{Minimize} \quad  \max_{j \in \cJ}\;\sum_{i \in \cI} \xi_{ij}\,,
\quad
\text{subject to} \quad  \sum_{j \in \cJ} \mu_{ij} \nu_j \xi_{ij}
\;=\; \lambda_i\,,~ i \in \cI\,,\quad
\text{and\ \ } [\xi_{ij}] \in\RR^{\cG}_+\,,
\end{align*}
has a unique solution
$\xi^*=[\xi^*_{ij}]\in\RR^{\cG}_+$  satisfying 
\begin{equation} \label{critfluid}
\sum_{i \in \cI} \xi^*_{ij} \;=\; 1, \quad \forall j \in \cJ \,,
\quad\text{and}\quad \xi^*_{ij}>0\quad \text{for all~} i \sim j\,.
\end{equation}
This assumption is referred to as the \emph{complete resource pooling} condition
\cite{williams-2000, Atar-05b}. 
It implies that the graph $\cG$ is a tree \cite{williams-2000, Atar-05b}.

We define  $x^*=(x_{i}^{*})_{i \in \cI}\in\RR^I_+$, and
$z^* = [z_{ij}^*]\in\RR^{\cG}_+$ by
\begin{equation} \label{staticfluid}
x_{i}^{*}\;=\;\sum_{j \in \cJ} \xi^*_{ij} \nu_j\,,
\qquad z_{ij}^* = \xi_{ij}^* \nu_j \,. 
\end{equation}
The vector  $x^*$ can be interpreted as the steady-state total number
of customers in each class, and the matrix $z^*$ as the steady-state number
of customers in each class receiving service, in the fluid scale.
Note that the steady-state queue lengths are all zero in the fluid scale.
The quantity $\xi_{ij}^*$  can be interpreted as the steady-state fraction of
service allocation of pool~$j$ to class-$i$ jobs in the fluid scale.
It is evident that \cref{critfluid,staticfluid} imply that
$e\cdot  x^*= e \cdot  \nu$, 
where $\nu\df(\nu_j)_{j \in \cJ}$.

\subsubsection{The state descriptors}
For $i\in\cI$,
let $X^{n}_i = \{X^{n}_i(t)\colon t\ge 0\}$ and $Q^{n}_i = \{Q^{n}_i(t)\colon t\ge 0\}$
be the number of class $i$ customers in the system and in the queue, respectively,
and for $j\in\cJ$, let  $Y^{n}_j = \{Y^{n}_j(t)\colon t\ge 0\}$, be the number of
idle servers in pool $j$.
We also let $Z_{ij}^{n} = \{Z_{ij}^{n}(t)\colon t\ge 0\}$ denote the
number of class $i$ customers being served in server pool $j$.
Set $X^{n} = (X_i^{n})_{i \in \cI}$, $Y^{n} = (Y_j^{n})_{j \in \cJ}$, 
$Q^{n} = (Q_i^{n})_{i \in \cI}$,
and $Z^{n} = (Z_{ij}^{n})_{i \in \cI,\, j \in \cJ}$. 
For each $t\ge 0$, we have the fundamental balance equations
\begin{equation} \label{baleq}
\begin{split}
 X^{n}_i(t)&\;=\;Q_i^{n}(t) + \sum_{j \in \cJ(i)} Z_{ij}^{n}(t)
\qquad\forall\,i\in\cI\,, \\[5pt]
 N_j^{n}&\;=\;Y_j^{n}(t) + \sum_{i \in\cI(j)} Z_{ij}^{n}(t)
 \qquad\forall\,j\in\cJ\,.
\end{split}
\end{equation}

\subsubsection{Scheduling control}

The control process is $Z^{n}$.
We  only consider work conserving scheduling policies that are non-anticipative and
preemptive. 
Work conservation requires that the processes $Q^{n}$ and $Y^{n}$ satisfy
\begin{equation*}
Q_i^{n}(t) \wedge Y^{n}_j(t)\;=\;0
\qquad \forall i \sim j\,, \quad\forall\, t \ge 0\,. 
\end{equation*}
In other words,  whenever there are customers waiting in queues,
if a server becomes free and can serve one of the customers, the server cannot
idle and must decide which customer to serve and start service immediately.
Service preemption is allowed, that is, service of a customer can be
interrupted at any time to serve some other customer of another class and
resumed at a later time. 

For $(x,z)\in\ZZ_+^I\times \ZZ^{\cG}_+$, we define
\begin{equation}\label{E-qy}
\begin{split}
q_i(x,z) &\;\df\; x_i - \sum_{j\in\cJ}
z_{ij}\,,\quad i\in\cI\,,\\[5pt]
y_j^n(z) &\;\df\;   N_j^n - \sum_{i\in\cJ}z_{ij}\,,
\quad j\in\cJ\,,
\end{split}
\end{equation}
and the \emph{action set} $\cZn(x)$ by
\begin{equation*}
\cZn(x)\;\df\; \bigl\{z \in \ZZ^{\cG}_+\;\colon
q_i(x,z) \wedge y_j^n(z) =0\,,~
q_i(x,z)\ge0\,,~y_j^n(z) \ge0 \;\;\;\forall\,(i,j)\in\cE\bigr\}\,. 
\end{equation*}
We denote $y_j(x,z)\,=\,y_j^n(x,z)$ whenever no confusion occurs.

Let $A^{n}_i$, $S^{n}_{ij}$, and $R^{n}_i$, $(i,j)\in\cE$, be
mutually independent
rate-$1$ Poisson processes, and also independent of the initial
condition $X^{n}_i(0)$.
Define the $\sigma$-fields 
\begin{align*}
\mathscr{F}^{n}_t &\;\df\; \sigma \bigl\{ X^{n}(0), \Tilde{A}^{n}_i(t),
\Tilde{S}^{n}_{ij}(t), \Tilde{R}^{n}_i(t)\;\colon\, i \in \cI, \; j \in \cJ,
\; 0 \le s \le t \bigr\} \vee \mathcal{N} \,,\\[5pt]
\mathscr{G}^{n}_t &\;\df\; \sigma \bigl\{ \delta\Tilde{A}^{n}_i(t, r),
\delta \Tilde{S}^{n}_{ij}(t, r), \delta\Tilde{R}^{n}_i(t, r)\;\colon\, i \in \cI,
\; j \in \cJ, \; r \ge 0 \bigr\} \,,
\end{align*}
where  $\mathcal{N}$ is the collection of all $\Prob$-null sets, and
\begin{align*}
\Tilde{A}^{n}_i(t) &\;\df\; A^{n}_i(\lambda_i^{n} t),
&\quad \delta\Tilde{A}^{n}_i(t,r)
&\;\df\; \Tilde{A}^{n}_i(t+r) - \Tilde{A}^{n}_i(t) \,, \\
\Tilde{S}^{n}_{ij}(t) &\;\df\; S^{n}_{ij} \left(\mu_{ij}^{n}
\int_0^t Z_{ij}^{n}(s)\,\D{s} \right),
&\quad \delta \Tilde{S}^{n}_{ij}(t, r) &\;\df\; S^{n}_{ij}
\left( \mu_{ij}^{n}\int_0^t Z_{ij}^{n}(s)\,\D{s} + \mu_{ij}^{n} r \right)
- \Tilde{S}^{n}_{ij}(t) \,,\\
\Tilde{R}^{n}_i(t) &\;\df\; R_i^{n}
\left(\gamma_i^{n} \int_0^t Q^{n}_i(s)\,\D{s} \right)\,,
&\quad \delta \Tilde{R}^{n}_i(t, r) &\;\df\; R_i^{n}
\left(\gamma_i^{n} \int_0^t Q^{n}_i(s)\, \D{s}
+ \gamma_i^{n} r \right) -  \Tilde{R}^{n}_i(t) \,. 
\end{align*}
The filtration $\bm{\mathscr{F}}^{n}\df\{\mathscr{F}^{n}_t\colon t \ge 0\}$ represents
the information available up to time $t$, and the filtration 
$\bm{\mathscr{G}}^{n}\df\{\mathscr{G}^{n}_t\colon t \ge 0\}$ contains the information
about future increments of the processes.
 
We say that  a scheduling policy $Z^n$ is \emph{admissible} if 
\begin{enumerate}
\item[(i)] $Z^{n}(t)\in \cZn(X^n(t))$ a.s.\ for all $t\ge0$;
\smallskip
\item[(ii)] $Z^{n}(t)$ is adapted to $\mathscr{F}^{n}_t$;
\smallskip
\item[(iii)] $\mathscr{F}^{n}_t$ is independent of $\mathscr{G}^{n}_t$ at each time
$t\ge 0$;
\smallskip
\item[(iv)] for each $i \in \cI$ and $i \in \cJ$, and for each
$t\ge 0$, the process  $\delta \Tilde{S}^{n}_{ij}(t, \cdot)$ agrees in law
with $S^{n}_{ij}(\mu_{ij}^{n}\,\cdot)$, and the process
$\delta \Tilde{R}^{n}_i(t, \cdot)$
agrees in law with 
$R^{n}_i (\gamma_i^{n} \cdot)$.  
\end{enumerate}
We denote the set of all admissible scheduling policies
$(Z^{n}, \bm{\mathscr{F}}^{n}, \bm{\mathscr{G}}^{n})$ by $\fZ^{n}$.
Abusing the notation we sometimes denote this as $Z^{n}\in\fZ^{n}$.
An admissible policy is called stationary Markov if
$Z^{n}(t)= z(X^n(t))$ for some function $z\colon \ZZ_+^I\to \ZZ^{\cG}_+$,
in which case we identify the policy with the function $z$.

Under an admissible scheduling policy, the state process $X^{n}$ can be represented as
\begin{align} \label{Xrep}
X^{n}_i(t)\;=\;X^{n}_i(0) + A^{n}_i(\lambda^{n}_i t) - \sum_{j \in \cJ(i)}
S^{n}_{ij} \left( \mu_{ij}^{n}\int_0^t Z_{ij}^{n}(s) \D{s} \right)
- R_i^{n} \left(\gamma_i^{n} \int_0^t Q^{n}_i(s) \D{s} \right)\,,
\end{align}
for $i \in \cI$ and $t\ge 0$.
Under a stationary Markov policy, $X^n$ is Markov with
generator
\begin{align}
\cL_n^z f(x) \;\df\; \sum_{i\in\cI} \lambda^{n}_i \bigl(f(x+e_i) &- f(x)\bigr)
+ \sum_{i\in\cI}\sum_{j\in\cJ(i)}\mu_{ij}^{n} z_{ij} \bigl(f(x-e_i) - f(x)\bigr)
\nonumber\\
&+ \sum_{i\in\cI} \gamma_i^{n} q_i(x,z)
\bigl(f(x-e_i)- f(x)\bigr)\,, \qquad   f\in\Cc(\RR^I)\,,\quad x \in \ZZ^{I}_{+}\,.
\label{E-cL}
\end{align}

\subsection{Joint work conservation} \label{S2.2}

\begin{definition}
We say that an action $z\in\cZn(x)$ is jointly work
conserving (JWC), if
\begin{equation}\label{E-JWC}
e\cdot q(x,z) \wedge e \cdot y^{n}(z)\;=\;0\,.
\end{equation}
We define
\begin{equation*}
\sX^{n}\;\df\;\bigl\{x\in\ZZ^I_+\,\colon\, ~e\cdot q(x,z) \wedge e \cdot y^n(x,z)
\;=\;0 ~\text{for some}~ z\in\cZn(x)\bigr\}\,, 
\end{equation*}
with $q$ and $y^n$ defined in \cref{E-qy}.
\end{definition}

Since \cref{baleq} implies that
\begin{equation*}
e\cdot (x- N^{n}) \;=\; 
e\cdot q(x,z) - e\cdot y^n(z)\,,
\end{equation*}
it is clear that \cref{E-JWC} is satisfied if and only if
\begin{equation*}
e\cdot q(x,z) \,=\, \bigl[e\cdot (x- N^{n})\bigr]^+\,,
\quad\text{and}\quad 
e\cdot y^n(z) \,=\, \bigl[e\cdot (x- N^{n})\bigr]^-\,.
\end{equation*}

Let
\begin{equation*}
\varTheta^n(x)\;\df\; \bigl\{(q,y)\in\ZZ_{+}^I\times\ZZ_+^J\;\colon
e\cdot q = [e\cdot (x- N^{n})]^+\,,~e\cdot y
= [e\cdot (x- N^{n})]^-\bigr\}\,,
\qquad x\in\ZZ_+^I\,.
\end{equation*}
It is evident that the JWC condition can be met at any point $x\in\ZZ_+^I$
at which the
image of $\cZn(x)$ under the
map $z\mapsto\bigl(q(x,z), y^n(z)\bigr)$ defined in \cref{E-qy}
intersects $\varTheta^n(x)$.

Let
\begin{equation*}
D_{\Psi} \;\df\; \bigl\{ (\alpha, \beta) \in \RR^{I}\times \RR^{J}\;\colon\,
e \cdot \alpha = e \cdot \beta\bigr\}\,.
\end{equation*}
As shown in Proposition~A.2 of \cite{Atar-05a},
provided that $\cG$ is a tree,
there exists a unique linear map $\Psi=[\Psi_{ij}]\colon D_\Psi\to \RR^{I\times J}$
solving
\begin{equation} \label{E-Psi}
\sum_j \Psi_{ij}(\alpha,\beta)\;=\;\alpha_i \quad \forall i \in \cI\,,
\quad\text{and}\quad
\sum_{i} \Psi_{ij}(\alpha,\beta)\;=\;\beta_j \quad \forall j \in \cJ\,,
\end{equation}
with $\Psi_{ij}(\alpha,\beta) \;=\;0$ for $i \nsim j$.

We quote a result from \cite{Atar-05b}, which is used later.
The proof of Lemma~3 in \cite{Atar-05b} assumes that the limits
in \cref{HWpara} exist, in particular, $\sqrt{n} (N_j^n - n \nu_j) \to 0$
as $n\to \infty$. 
Nevertheless, the proof goes through under the weaker assumption
that $N_j^n - n \nu_j = \sorder(n)$. 

\begin{lemma}[Lemma~3 in \cite{Atar-05b}]\label{L-JWC}
There exists a constant $M_{0}>0$ such that, the collection
of sets $\Breve{\sX}^{n}$ defined by
\begin{equation}\label{E-BsX}
\Breve{\sX}^{n}\;\df\;
\bigl\{x \in \ZZ^{I}_{+}\;\colon\,  \norm{x - n x^*}\le M_{0}\, n  \bigr\}\,,
\end{equation} 
satisfies $\Breve{\sX}^{n}\subset\sX^n$ for all $n\in\NN$.
Moreover,
\begin{equation*}
\Psi(x-q,N^n-y)\in\ZZ_+^{I\times J}\qquad\forall\,q,y\in\varTheta^n(x)\,,
\quad\forall\,x\in \Breve{\sX}^{n}\,.
\end{equation*}
\end{lemma}

\begin{remark}
\cref{L-JWC} implies that if $x\in\Breve{\sX}^{n}$, then 
for any $q\in\ZZ^I_+$ and $y\in\ZZ^J_+$ satisfying $e\cdot q\wedge e\cdot y=0$
and $e\cdot (x-q)= e\cdot (N^n-y)\ge0$, we have
$\Psi(x-q,N^n-y)\in\cZn(x)$.
\end{remark}

We need the following definition. 

\begin{definition}\label{DEJWC}
We fix some open ball $\Breve{B}$ centered at the origin,
such that $n (\Breve{B}+x^*)\subset \Breve{\sX}^n$ for all $n\in\NN$.
The \emph{jointly work conserving
action set $\bcZn(x)$ at $x$} is defined as the subset of $\cZn(x)$,
which satisfies
\begin{equation*}
\bcZn(x) \;\df\; \begin{cases}
\bigl\{ z\in\cZn(x)\,\colon\,~
e\cdot q(x,z) \wedge e \cdot y^n(z)\;=\;0\bigr\}
&\text{if\ \ } x\in n (\Breve{B}+x^*)\,,\\[3pt]
\cZn(x)&\text{otherwise,}
\end{cases}
\end{equation*}
with $q$ and $y^n$ as in \cref{E-qy}.
We also define the associated admissible policies by 
\begin{equation*}
\begin{split}
\Breve\fZ^n &\;\df\;
\bigl\{Z^{n}\in\fZ^n \,\colon\, Z^n(t)\in\bcZn\bigl(X^n(t)\bigr)
\;\;\;\forall\, t\ge0\bigr\}\,,\\[5pt]
\boldsymbol\fZ &\;\df\;\{Z^{n}\in\Breve\fZ^n\,,~n\in\NN\}\,.
\end{split}
\end{equation*}
We refer to the policies in $\boldsymbol\fZ$ as
\emph{eventually jointly work conserving} (EJWC).
\end{definition}

The ball $\Breve{B}$ is fixed in \cref{DEJWC} only for convenience.
We could instead adopt a more general definition of $\boldsymbol\fZ$,
as explained in Remark~2.1 in \cite{AP15}.  
The EJWC condition plays a crucial role in the derivation of the
controlled diffusion limit.
Therefore, the convergence of mean empirical measures of the
controlled diffusion-scaled state
process, and thus, also the lower and upper bounds for
asymptotic optimality are established for sequences
$\{Z^n,\,n\in\NN\}\subset\boldsymbol\fZ$. 

\subsection{The diffusion-scaled processes} \label{S2.3}

Let $x^*$ and $z^*$ be as in \cref{staticfluid}.
We define the diffusion-scaled processes
$\Hat{Z}^{n}$,
$\Hat{X}^{n}$,
$\Hat{Q}^{n}$, and $\Hat{Y}^{n}$, by
\begin{equation} \label{DiffDef}
\begin{aligned}
\Hat{X}^{n}_i(t) &\;\df\; \frac{1}{\sqrt{n}} (X_i^{n}(t) - n x^*_{i}) \,,  \\[5pt]
\Hat{Q}^{n}_i(t) &\;\df\; \frac{1}{\sqrt{n}} Q_i^{n}(t) \,,
\end{aligned}
\qquad
\begin{aligned}
\Hat{Z}^{n}_{ij}(t) &\;\df\; \frac{1}{\sqrt{n}} (Z_{ij}^{n}(t) - n z^*_{ij})\,,\\[5pt]
\Hat{Y}^{n}_j(t) &\;\df\; \frac{1}{\sqrt{n}} Y_j^{n}(t) \,.
\end{aligned}
\end{equation}

Let
\begin{equation*}
\begin{split}
\Hat{M}^{n}_{A, i}(t) &\;\df\;  \frac{1}{\sqrt{n}}(A^{n}_i(\lambda_i^{n} t)
- \lambda_i^{n} t), \\[5pt]
 \Hat{M}^{n}_{S, ij}(t) &\;\df\; \frac{1}{\sqrt{n}}\left( S^{n}_{ij}
 \left( \mu_{ij}^{n}\int_0^t Z_{ij}^{n}(s) \D{s} \right)
 - \mu_{ij}^{n}\int_0^t Z_{ij}^{n}(s) \D{s}\right) \,,\\[5pt]
\Hat{M}^{n}_{R, i}(t) &\;\df\;\frac{1}{\sqrt{n}}
\left(R_i^{n} \left(\gamma_i^{n} \int_0^t Q^{n}_i(s) \D{s} \right)
-\gamma_i^{n} \int_0^t Q^{n}_i(s) \D{s} \right)\,.
\end{split}
\end{equation*}
These are square integrable martingales w.r.t. the filtration $\bm{\mathscr{F}}^{n}$,
with quadratic variations
\begin{equation*}
\langle \Hat{M}^{n}_{A, i} \rangle(t) \;\df\; \frac{\lambda_i^{n}}{n} t\,, \quad 
\langle \Hat{M}^{n}_{S, ij} \rangle (t) \;\df\; \frac{\mu_{ij}^{n}}{n}
\int_0^t Z_{ij}^{n}(s)\D{s}\,,\quad 
\langle \Hat{M}^{n}_{R, i} \rangle (t)\;\df\;\frac{\gamma_i^{n}}{n}
\int_0^t Q^{n}_i(s)\D{s} \,. 
\end{equation*}
Let $\widehat{M}^n(t)\df\Hat{M}^{n}_{A, i} (t)
- \sum_{j \in \cJ(i)}\Hat{M}^{n}_{S, ij}(t) - \Hat{M}^{n}_{R, i}(t)$.
By \cref{Xrep}, we can write $\Hat{X}^{n}_i(t)$ as
\begin{equation} \label{hatXn-1}
\Hat{X}^{n}_i(t) \;=\; \Hat{X}^{n}_i(0) + \ell_i^{n} t
- \sum_{j \in \cJ(i)} \mu_{ij}^{n} \int_0^t \Hat{Z}^{n}_{ij}(s) \D{s}
- \gamma^{n}_i \int_0^t \Hat{Q}^{n}_i(s) \D{s}  \\[5pt]
+ \widehat{M}^n(t)\,, 
\end{equation}
where 
$\ell^{n} = (\ell^{n}_1,\dotsc,\ell^{n}_I)\transp$ is defined as
\begin{equation*}
\ell^{n}_i\;\df\; \frac{1}{\sqrt{n}} \bigg(\lambda^{n}_i
-  \sum_{j \in \cJ(i)} \mu_{ij}^{n} z_{ij}^* n \bigg) \,.
\end{equation*}
Under the assumptions on the parameters in
\cref{HWpara} and the first constraint in the LP,
it holds that
\begin{equation*}
\ell^{n}_i \;\xrightarrow[n\to\infty]{}\;
\ell_i\;\df\; \Hat{\lambda}_i - \sum_{j \in \cJ(i)} \Hat{\mu}_{ij} z_{ij}^* \,. 
\end{equation*}
We let $\ell\df(\ell_1,\dotsc,\ell_I)\transp$.

By \cref{staticfluid,baleq,DiffDef}, we obtain the balance equations
\begin{equation}\label{baleq-hat}
\begin{split}
& \Hat{X}^{n}_i(t) \;=\; \Hat{Q}_i^{n}(t) + \sum_{j \in \cJ(i)} \Hat{Z}^{n}_{ij}(t)
\qquad \forall\, i \in \cI\,, \\[5pt]
 & \Hat{Y}^{n}_j(t) + \sum_{i\in \cI(j)} \Hat{Z}^{n}_{ij}(t) \;=\; 0
 \qquad \forall\, j \in \cJ \,.
\end{split}
\end{equation}

\begin{definition}\label{D-hatx}
For each $x\in\ZZ^{I}_{+}$ and $z\in \cZn(x)$, we define 
\begin{equation}\label{Td-xn}
\begin{split}
\Tilde{x}^{n} \;=\;\Tilde{x}^{n}(x)\;\df\; x - n x^*\,, \qquad
 \Hat{x}^{n}\;=\;\Hat{x}^{n}(x)\;\df\; \frac{\Tilde{x}^{n}(x)}{\sqrt{n}}\,, \qquad
 \Hat{z}^n(z) \;\df\; \frac{ z - n z^*}{\sqrt{n}}\,,
 \\[5pt]
\Hat{q}^n(x,z)\;\df\; \frac{q(x,z)}{\sqrt{n}}\,,\qquad
\Hat{y}^n(z)\;\df\; \frac{y^n(z)}{\sqrt{n}}\,,\qquad
\Hat{\vartheta}^n(x,z)\;\df\; e\cdot \Hat{q}^n(x,z)\wedge e\cdot \Hat{y}^n(z)\,,
\end{split}
\end{equation}
with $q(x,z)$, $y^n(z)$ as in \cref{E-qy}.
We also let
\begin{equation*}
\sS^n \;\df\; \bigl\{\Hat{x}^{n}(x) \,\colon\,x\in\ZZ^I_{+}\bigr\}\,,
\qquad
\Breve\sS^n \;\df\; \bigl\{\Hat{x}^{n}(x) \,\colon\,x\in\Breve\sX^n\bigr\}
\end{equation*}
and
\begin{equation*}
\hcZn(\Hat{x}) \;\df\;
\{\Hat{z}^n(z)\,\colon\,z\in\cZn(\sqrt{n}\Hat{x}+n x^*)\}\, \qquad \Hat{x}\in\sS^n\,.
\end{equation*}
Abusing the notation, we also write
\begin{equation} \label{hatq-hatxz}
\Hat{q}^n_i(\Hat{x},\Hat{z})\;=\;
\Hat{q}^n_i(\Hat{x}^n,\Hat{z}^n) \;=\;
\Hat{x}_i^n - \sum_{j\in \cJ(i)} \Hat{z}^n_{ij} \quad \text{for} \quad i \in \cI\,,
\end{equation}
and
\begin{equation} \label{haty-hatz}
\Hat{y}^n_j(\Hat{z})\;=\; \Hat{y}^n_j(\Hat{z}^n) \;=\;
\frac{N^n_j- n \sum_{i\in \cI(j)} z_{ij}^*}{\sqrt{n}}
- \sum_{i \in \cI(j)} \Hat{z}^n_{ij} \quad \text{for} \quad j \in \cJ\,.
\end{equation}
\end{definition}

\begin{lemma}\label{L2.2}
There exists a constant $\Tilde{M}_0>0$ such that
for any $z\in\bcZn(x)$, $x\in\ZZ^I_+$, and $n\in\NN$, we have
\begin{equation*}
\max\;\biggl\{\max_{(i,j)\in\cE}\; \abs{\Hat{z}^n_{ij}(z)},\,
\norm{\Hat{q}^n(x,z)},\, \norm{\Hat{y}^n(z)},\,\Hat{\vartheta}^n(x,z)\biggr\}
\;\le\; \Tilde{M}_0\,\norm{\Hat{x}^n(x)}\,.
\end{equation*}
\end{lemma}

\begin{proof}
Note that
\begin{equation}\label{EL2.2A}
\norm{\Hat{q}^n(x,z)}=\Hat{\vartheta}^n(x,z)+ \bigl(e\cdot \Hat{x}^{n}(x)\bigr)^+\,,
\quad\text{and}\quad
\norm{\Hat{y}^n(z)}=\Hat{\vartheta}^n(x,z)+ \bigl(e\cdot \Hat{x}^{n}(x)\bigr)^-
\end{equation}
for all $x\in\ZZ^{I}_{+}$ and $z\in \cZn(x)$.
Therefore, there exist probability vectors $p^c\in[0,1]^I$ and $p^s\in[0,1]^J$
such that
$\Hat{q}^{n}= \bigl(\Hat{\vartheta}^{n}+(e\cdot\Hat{x}^{n})^{+}\bigr)p^{c}$ 
and $\Hat{y}^{n}= \bigl(\Hat{\vartheta}^{n}+(e\cdot\Hat{x}^{n})^{-}\bigr)p^{s}$.
By the linearity of the map $\Psi$ and \cref{L-JWC}, it easily follows
that
\begin{equation}\label{EL2.2B}
\Hat{z}^{n}\;=\; \Psi(\Hat{x}^{n}-\Hat{q}^{n},- \Hat{y}^{n})
\;=\; \Psi\bigl(\Hat{x}^{n}-(e\cdot\Hat{x}^{n})^{+}p^{c},
-(e\cdot\Hat{x}^{n})^{-}p^{s}\bigr)
-\Hat{\vartheta}^{n}\,\Psi(p^c,p^s)\,.
\end{equation}

If $x\notin\Breve\sX^n$, then $\norm{\Hat{x}^n}>M_0 \sqrt{n}$
by \cref{E-BsX}.
Since for some constant $C>0$, 
it holds that $\norm{\Hat{y}^n(z)}\le C \sqrt{n}$ for all $n\in\NN$,
the same bound also holds for $\Hat{\vartheta}^n(x,z)$.
Thus if $x\notin\Breve\sX^n$, we obtain the bound asserted in
the lemma by \cref{EL2.2A,EL2.2B}.

On the other hand, if $x\in\Breve\sX^n$ and $z\in\bcZn(x)$,
then $\Hat{\vartheta}^n(x,z)=0$,
and again the assertion of the lemma follows by \cref{EL2.2A,EL2.2B}.
This completes the proof.
\end{proof}

\begin{definition}\label{D-scrA}
We define the operator 
$\mathscr{A}^n\colon\Cc^{2}(\RR^I)\to\Cc(\RR^I,\RR^{I\times J})$ by 
\begin{equation*}
\mathscr{A}^n f 
\bigl(\Hat{x},\Hat{z}\bigr) \;\df \; \sum_{i\,\in\,\cI}
\Bigl(\mathscr{A}^{n}_{i,1}(\Hat{x}_i,\Hat{z})\,\partial_i f(\Hat{x})
+\mathscr{A}^{n}_{i,2}(\Hat{x}_i,\Hat{z})\,\partial_{ii} f(\Hat{x}) \Bigr)\,,
\quad f\in\Cc^{2}(\RR^I)\,,
\end{equation*}
where $\partial_{i}\df\tfrac{\partial~}{\partial{x}_{i}}$ and
$\partial_{ij}\df\tfrac{\partial^{2}~}{\partial{x}_{i}\partial{x}_{j}}$, and
\begin{align*}
\mathscr{A}^{n}_{i,1}\bigl(\Hat{x}_i,\Hat{z}\bigr)&\;\df\;
\ell^{n}_{i} - \sum_{j \in \cJ(i)}\mu^{n}_{ij} \Hat{z}_{ij}-\gamma^{n}_{i}
\biggl(\Hat{x}_i- \sum_{j \in \cJ(i)}\Hat{z}_{ij}\biggr)\,,\\[5pt]
\mathscr{A}^{n}_{i,2}\bigl(\Hat{x}_i,\Hat{z}\bigr)&\;\df\;
\frac{1}{2}\biggl[\frac{\lambda^{n}_{i}}{n}
+\sum_{j\in\cJ(i)} \mu_{ij}^{n} z_{ij}^*
+\frac{1}{\sqrt{n}} \sum_{j\in\cJ(i)} \mu_{ij}^{n} \Hat{z}_{ij} 
+ \frac{\gamma^{n}_i}{\sqrt{n}} \biggl(\Hat{x}_i - \sum_{j \in \cJ(i)}
\Hat{z}_{ij}\biggr)\biggr]\,.
\end{align*}
\end{definition}

By the Kunita--Watanabe formula for semi-martingales
(see, e.g.,
\cite[Theorem~26.7]{kallenberg}), we have
\begin{equation}\label{E-KW}
f(\Hat{X}^{n}(t)) \;=\; f(\Hat{X}^{n}(0)) +
\int_{0}^{t} \mathscr{A}^n f  \bigl(\Hat{X}^{n}(s), \Hat{Z}^{n}(s)\bigr)\,\D{s}  
+\sum_{s\le t} \mathscr{D} f(\Hat{X}^n,s)\qquad\forall\, f\in\Cc^{2}(\RR^I)\,,  
\end{equation}
for any admissible diffusion-scaled policy $\Hat{Z}^{n}$, where 
\begin{multline}\label{E-scrD}
\mathscr{D} f(\Hat{X}^n,s)\;\df\;
\Delta f(\Hat{X}^{n}(s))
- \sum_{i\in\cI} \partial_{i}f(\Hat{X}^{n}(s-)) \Delta \Hat{X}^{n}_i(s)\\
 - \frac{1}{2} \sum_{i, i'\in\cI} \partial_{ii'}f(\Hat{X}^{n}(s-))
\Delta \Hat{X}^{n}_i(s) \Delta \Hat{X}_{i'}^{n}(s)\,.
\end{multline}

\subsubsection{Control parameterization}

\begin{definition} \label{D-upara}
Let $\boldsymbol{\mathfrak{X}}^n\df\{(\Hat{x},\Hat{z})\colon
\Hat{x}\in\sS^n\,,\ \Hat{z}\in\hcZn(\Hat{x})\}$.
For each $(\Hat{x},\Hat{z})\in\boldsymbol{\mathfrak{X}}^n$, we define
\begin{equation*}
u^{c}_i(\Hat{x}, \Hat{z}) \;=\; u^{c,n}_i(\Hat{x},\Hat{z}) \;\df\; \begin{cases}
\frac{\Hat{q}_i^n(\Hat{x},\Hat{z})}{e\cdot \Hat{q}^n(\Hat{x},\Hat{z})}
& \text{if~} e\cdot \Hat{q}^n(\Hat{x},\Hat{z})>0\,,\\
e_I&\text{otherwise,}\end{cases}
\qquad i\in\cI\,,\quad t\ge0\,,
\end{equation*}
and
\begin{equation*}
u^{s}_j(\Hat{z}) \;=\; u^{s,n}_j(\Hat{z}) \;\df\; \begin{cases}
\frac{\Hat{y}^n_j(\Hat{z})}{e\cdot \Hat{y}^n(\Hat{z})}
& \text{if~} e\cdot \Hat{y}^n(\Hat{z})>0\,,\\
e_J&\text{otherwise,}\end{cases}
\qquad j\in\cJ\,,\quad t\ge0\,.
\end{equation*}
Let $u(\Hat{x},\Hat{z})\df (u^c(\Hat{x},\Hat{z}),u^s(\Hat{z}))$.
Then $u(\Hat{x},\Hat{z})$ belongs to the set 
\begin{equation}\label{E-Act}
\Act \;\df\; \bigl\{ u= (u^c, u^s) \in \RR^{I}_+ \times \RR^J_+
\,\colon\, e\cdot u^c = e\cdot u^s=1\bigr\}\,.
\end{equation}
We also define the processes
\begin{equation*}
U^{c,n}_i(t)\;\df\; u^{c}_i\bigl(\Hat{X}^n(t), \Hat{Z}^n(t)\bigr)\,,\qquad
U^{s,n}_i(t)\;\df\; u^{s}_i\bigl(\Hat{Z}^n(t)\bigr)\,,
\end{equation*}
and $U^{n}\df (U^{c,n}, U^{s,n})$, with
$U^{c,n}\df (U^{c,n}_1,\dotsc, U^{c,n}_I)\transp$,
and $U^{s,n}\df (U^{s,n}_1,\dotsc, U^{s,n}_J)\transp$.
\end{definition}

The process $U^{c,n}_i(t)$ represents the proportion of the total queue length
in the network at queue $i$ at time $t$, while $U^{s,n}_j(t)$ represents the
proportion of the total idle servers in the network at station $j$ at time $t$.
Given $Z^{n} \in \fZ^{n}$ the process $U^{n}$ is uniquely determined
and lives in the set $\Act$.

For $u\in\Act$, let $\widehat{\Psi}[u]\colon\RR^I\to\RR^{\cG}$ be defined by
\begin{equation}\label{E-HatPsi}
\widehat{\Psi}[u](x) \;\df\; \Psi(x- (e\cdot x)^{+} u^c, - (e\cdot x)^{-} u^s)\,,
\end{equation}
where $\Psi$ is as in \cref{E-Psi}.

We define the operator
$\Breve{\mathscr{A}}^n\colon\Cc^{2}(\RR^I)
\to\Cc(\RR^I,\Act)$
by
\begin{equation*}
\Breve{\mathscr{A}}^{n}f(\Hat{x},u) \;\df\; \sum_{i\in\cI}
\Bigl(\Breve{\mathscr{A}}_{i,1}^{n}(\Hat{x},u)\,\partial_{i}f(\Hat{x})
+  \Breve{\mathscr{A}}_{i,2}^{n}(\Hat{x},u)\,\partial_{ii}f(\Hat{x}) \Bigr)\,,
\end{equation*}
where 
\begin{subequations} 
\begin{align*}
\Breve{\mathscr{A}}_{i,1}^{n}(\Hat{x},u) &\;\df\;
\ell^{n}_i -\sum_{j \in \cJ(i)} \mu_{ij}^{n}
\widehat{\Psi}_{ij}[u](\Hat{x}) - \gamma^{n}_i (e\cdot \Hat{x})^{+} u^c_i\,,\\[5pt]
\Breve{\mathscr{A}}_{i,2}^{n}(\Hat{x},u) &\;\df\;
\frac{1}{2}\Biggl(\frac{\lambda^{n}_{i}}{n}
+\sum_{j\in\cJ(i)} \mu_{ij}^{n}z^{*}_{ij}
+\frac{1}{\sqrt{n}} \sum_{j \in \cJ(i)} \mu_{ij}^{n}
\widehat{\Psi}_{ij}[u](\Hat{x})
+ \frac{\gamma^{n}_i}{\sqrt{n}} \bigl(  (e\cdot \Hat{x})^{+} u^c_i  \bigr) \Biggr)\,.
\end{align*}
\end{subequations}

The following lemma is a result of a simple calculation based on
the definitions above.
Recall the definitions of  $\Breve{B}$, $\bcZn$, and $\Breve\sS^n$
from \cref{DEJWC,D-hatx}.

\begin{lemma} \label{L2.3}
Let $u=u(\Hat{x},\Hat{z})\colon\boldsymbol{\mathfrak{X}}^n \to\Act$
denote the map given in \cref{D-upara}.
Then for $f\in\Cc^{2}_c(\sqrt{n}\Breve{B})$, we have
\begin{equation*}
\Breve{\mathscr{A}}^n f\bigl(\Hat{x}, u(\Hat{x},\Hat{z})\bigr)
= \mathscr{A}^nf\bigl(\Hat{x},\Hat{z})\,,\qquad \forall\,
\Hat{x}\in\Breve\sS^n\cap\sqrt{n}\Breve{B}\,,
\ \forall\,\Hat{z}\in  \bcZn(\sqrt{n}\Hat{x}+n x^*)\,.
\end{equation*}
\end{lemma}

\subsection{The diffusion limit} \label{S2.4}

Consider the $I$-dimensional controlled diffusion given
by the It\^o equation 
\begin{equation} \label{E-diff}
d X_t\;=\;b(X_t, U_t)\,\D{t} + \Sigma \, \D W_t\,, 
\end{equation}
where $W$ is an $I$-dimensional standard Wiener process.
The drift $b\colon \RR^I \times \Act \to \RR^I$ takes the form
\begin{equation*}
b_i(x,u)\;=\;b_i(x, (u^c, u^s))\;\df\; \ell_i - \sum_{j \in \cJ(i)}
\mu_{ij} \widehat{\Psi}_{ij}[u](x)
- \gamma_i (e\cdot x)^{+} u^c_i  \qquad\forall\,i\in\cI\,,
\end{equation*}
where $\widehat{\Psi}_{ij}[u]$ is as in \cref{E-HatPsi}.
Also
$\Sigma\;\df\;\diag\bigl(\sqrt{2 \lambda_1}, \dotsc, \sqrt{2 \lambda_I}\bigr)$. 

The control process $U$ takes values in  $\Act$, defined in \cref{E-Act}, and
$U_{t}(\omega)$ is jointly measurable in
$(t,\omega)\in[0,\infty)\times\Omega$.
Moreover, it is \emph{non-anticipative}, i.e.,
for $s < t$, $W_{t} - W_{s}$ is independent of
\begin{equation*}
\sF_{s} \;\df\;\text{the completion of~} \sigma\{X_{0},U_{r},W_{r},\;r\le s\}
\text{~relative to~}(\sF,\Prob)\,.
\end{equation*}
Such a process $U$ is called an \emph{admissible control}.
Let $\Uadm$ denote the set of all admissible controls.
Recall that a control is called \emph{Markov} if
$U_{t} = v(t,X_{t})$ for a measurable map $v\colon\RR_{+}\times\RR^{I}\to \Act$,
and it is called \emph{stationary Markov} if $v$ does not depend on
$t$, i.e., $v\colon\RR^{I}\to \Act$.
Let $\Usm$ denote the set of stationary Markov controls.
Recall also that a control $v\in\Usm$ is called \emph{stable}
if the controlled process is positive recurrent.
We denote the set of such controls by $\Ussm$.
Let
\begin{equation}\label{E-generator}
\Lg^{u} f(x) \;\df\; \sum_{i\in\cI} \bigl[\lambda_i\,\partial_{ii} f(x)
+  b_{i}(x,u)\, \partial_{i} f(x)\bigr]\,,\quad u\in\Act\,,
\end{equation}
denote the extended controlled generator of the diffusion in \cref{E-diff}.

In \cite{AP15}, a leaf elimination algorithm was developed to obtain
an explicit expression for the drift $b(x,u)$.
This plays an important
role in understanding the recurrence properties of the controlled diffusion.
See also Remark~4.2 and Example~4.4 in \cite{AP15}.
We quote this result as follows.

\begin{lemma}[Lemma~4.3 in \cite{AP15}]
The drift $b(x, u)= b(x, (u^c, u^s))$ in the limiting diffusion $X$
in \cref{E-diff}
can be expressed as
\begin{equation} \label{E-drift2}
b(x, u) \;=\;\ell - B_1 (x -  (e\cdot x)^{+} u^c) + (e\cdot x)^{-} B_2 u^s
- (e \cdot x)^{+} \varGamma u^c\,,
\end{equation}
where $B_1$ is a lower-diagonal $I\times I$ matrix with positive diagonal
elements, $B_2$ is an $I \times J$ matrix and
$\varGamma= \diag\{\gamma_1,\dotsc, \gamma_I\}$.
\end{lemma}

The drift in \cref{E-drift2} takes the form
\begin{equation}\label{E-drift3}
b_{i}\bigl(x,u\bigr)\;=\; \ell_i- \mu_{ij_i} x_{i}
+ \Tilde{b}_{i}(x_{1},\dotsc,x_{i-1}) + \Tilde{F}_{i}\bigl((e\cdot x)^{+} u^c,(e\cdot x)^{-} u^s \bigr)
- \gamma_i\, (e\cdot x)^{+} u_i^c \,,
\end{equation}
where $j_i \in \cJ$, $i\sim j_i$,
is the unique server-pool node corresponding to $i$ when customer node $i$ 
is removed by the leaf elimination algorithm (see Section 4.1 in \cite{AP15}).   
Two things are important to note: (a) $\Tilde{F}_{i}$ is a linear function,
and (b) $\mu_{ij_i}>0$ (since $i\sim j_i$).

Under EJWC policies, convergence in distribution of the diffusion-scaled
processes $\Hat{X}^n$ to the
limiting diffusion $X$ in \cref{E-diff} follows by \cite[Proposition~3]{Atar-05b}
for certain classes of networks.
The fact that \cref{E-diff} can be viewed as a limit of the diffusion-scaled
process $\Hat{X}^n$ is also indicated
by the following lemma.

\begin{lemma} \label{L2.5}
We have
\begin{equation*}
\Breve{\mathscr{A}}_{i,1}^{n}\;\xrightarrow[n\to\infty]{}\;
b_i\,,\quad\text{and}\quad
\Breve{\mathscr{A}}_{i,2}^{n}\;\xrightarrow[n\to\infty]{}\;
\lambda_i
\end{equation*}
for $i\in\cI$, uniformly over compact sets of $\RR^I\times\Act$.
In particular, for any $f\in\Cc^2_c(\RR^I)$ it holds that
\begin{equation*}
\Breve{\mathscr{A}}^{n} f(x,u) \;\xrightarrow[n\to\infty]{}\;
\Lg^{u} f(x)\,.
\end{equation*}
\end{lemma}

Nevertheless, for the time being, we consider solutions
of \cref{E-diff} as the formal limit of \cref{hatXn-1}.
Precise links of the $n^{\text{th}}$ system model
and \cref{E-diff} are established in \cref{S6}.

\begin{definition}\label{Lyapk} 
Let
$\norm{x}_\beta \df
(\beta_1\,\abs{x_1}^2+\dotsb+\beta_I\,\abs{x_I}^2)^{\nicefrac{1}{2}}$,
with $\beta=(\beta_1,\dotsc,\beta_I)$ a positive vector. 
Throughout the paper, $\Lyap_{\kappa,\beta}$, $\kappa\ge1$, stands for
a $\Cc^2(\RR^I)$ function which agrees with $\norm{x}_\beta^\kappa$ on the complement
of the unit ball $B$ in $\RR^I$, i.e.,
$\Lyap_{\kappa,\beta} (x)\;=\; \norm{x}_\beta^\kappa$,
for $x\in B^c$.
Also, $\widetilde\Lyap_{\epsilon,\beta}$, $\epsilon>0$, is defined by
\begin{equation*}
\widetilde\Lyap_{\epsilon,\beta}(x) \;\df\;
\exp\Bigl( \epsilon \norm{x}_\beta^2
\bigl(1+\norm{x}_\beta^2\bigr)^{-\nicefrac{1}{2}}\Bigr)\,,\qquad x\in\RR^I\,.
\end{equation*}

In addition, for $\delta>0$, we define
\begin{equation*}
\cK_{\delta}\;\df\; \bigl\{ x\in\RI\,\colon \abs{e\cdot x} > \delta \abs{x}\bigr\}\,.
\end{equation*}
\end{definition}

As shown in Theorem~4.1 of \cite{AP15}, the drift $b$ in \cref{E-drift3}
has the following important structural property.
For any $\kappa\ge1$,
there exists a function $\Lyap_{\kappa,\beta}$ as in \cref{Lyapk},
and positive constants $c_i$, $i=0,1,2$,
such that
\begin{equation}\label{E-structural}
\begin{split}
b(x,u)\cdot\nabla \Lyap_{\kappa,\beta}(x)\;\le\;
c_0 - c_1 \Lyap_{\kappa,\beta}(x)\,\Ind_{\cK^c_{\delta}}(x)
+ c_2 \Lyap_{\kappa,\beta}(x)\,\Ind_{\cK_{\delta}}(x)
\qquad\forall\,(x,u)\in\RR^I\times\Act\,.
\end{split}
\end{equation}
Since the diffusion matrix is constant, it is evident that a similar
estimate holds for $\Lg^u \Lyap_{\kappa,\beta}$ uniformly over $u\in\Act$.
By a straightforward application of It\^o's formula, this implies  that for
any $\kappa\ge1$ there exists a constant $C$ depending only on $\kappa$
such that (see \cite[Lemma~3.1\,(c)]{AP15})
\begin{equation}\label{E-BQYBS}
\Exp^{U}_{x}\biggl[\int_{0}^{T}\abs{X_{s}}^{\kappa}\,\D{s}\biggr]
\;\le\; C\, \abs{x}^\kappa
+ C\,\Exp^{U}_{x}\biggl[\int_{0}^{T}
\bigl(1+\abs{e\cdot X_s}\bigr)^{\kappa}\,\D{s}\biggr]
\qquad\forall\,T>0\,,\quad\forall\,U\in\Uadm\,.
\end{equation}

Moreover, it is shown in \cite[Theorem~4.2]{AP15} that
there exists a stationary Markov control $\Bar{v}\in\Usm$ satisfying 
\begin{equation}\label{E-stablev}
\Lg^{\Bar{v}} \Lyap_{\kappa,\beta}(x)\;\le\; \Bar{c}_0 - \Bar{c}_1
\Lyap_{\kappa,\beta}(x) \qquad\forall x\in\RR^I\,,
\end{equation}
for any $\kappa\ge1$, and positive constants $\Bar{c}_0$ and $\Bar{c}_1$
depending only on $\kappa$. 
As a consequence of \cref{E-stablev}, the diffusion under the control
$\Bar{v}$ is exponentially ergodic.
A slight modification of that proof leads to the following theorem.

\begin{theorem}\label{T2.1}
Provided that $\gamma_i>0$ for some $i\in\cI$,
there exist $\epsilon>0$, a positive vector $\beta\in\RR^I$,
 and a stationary Markov control $\Bar{v}\in\Usm$ satisfying 
\begin{equation}\label{E-stablev-exp}
\Lg^{\Bar{v}} \widetilde{\Lyap}_{\epsilon,\beta}(x)\;\le\; \tilde{c}_0 - \tilde{c}_1
\widetilde{\Lyap}_{\epsilon,\beta}(x) \qquad\forall x\in\RR^I\,,
\end{equation}
for  some positive constants $\tilde{c}_0$ and $\tilde{c}_1$.
\end{theorem} 
The properties in \cref{E-structural,E-stablev,E-stablev-exp}
are instrumental in showing that the optimal control problems
defined in this paper are well posed.

\section{Ergodic Control Problems}
In this section, we consider two control objectives, which address the queueing (delay)
and/or idleness costs in the system:  (i) \emph{unconstrained problem},
minimizing the queueing and idleness cost and
(ii) \emph{constrained problem}, minimizing the queueing cost while
imposing a constraint on idleness.
We state both problems for the $n^{\rm th}$ system and the limiting diffusion.

\subsection{Ergodic control problems for the \texorpdfstring{$n^{\rm th}$}{nth} system}
\label{S3.1}

The running cost is a function of the diffusion-scaled processes,
which are related to the unscaled ones by \cref{DiffDef}.
For simplicity, in all three cost minimization problems, 
we assume that the initial condition $X^{n}(0)$ is deterministic
and $\Hat{X}^{n}(0) \to x \in \RR^I$ as $n \to \infty$.  
Let the running cost $\Hat{r}\colon \RR^I_{+}\times \RR^J_+ \to \RR_+$ be defined by
\begin{equation} \label{runcost-ex}
\Hat{r}(\Hat{q},\Hat{y})\;=\;
\sum_{i\in\cI} \xi_i \Hat{q}_i^m + \sum_{j\in\cJ} \zeta_j \Hat{y}_j^m\,,
\quad \Hat{q} \in \RR^{I}_{+}\,, \;\Hat{y} \in \RR^J_+\,,
\quad\text{for some~} m\ge 1\,, 
\end{equation}
where 
$\xi=(\xi_1,\dotsc,\xi_I)\transp$ is a positive vector
and $\zeta=(\zeta_1,\dotsc,\zeta_J)\transp$ is
a nonnegative vector. In the case $\zeta\equiv 0$,
only the queueing cost is minimized. 
We denote by $\Exp^{Z^n}$ the expectation operator under
an admissible policy $Z^n$.

\begin{itemize}
\item[\textbf{(P1)}] (\emph{unconstrained problem})
The running cost penalizes the queueing and idleness. 
Let $\Hat{r}(q,y)$ be the running cost function as defined in \cref{runcost-ex}.
Here $\zeta>0$.
Given an initial state $X^{n}(0)$, and an admissible scheduling policy
$Z^{n} \in \Breve\fZ^{n}$,
 we define the diffusion-scaled cost criterion by
\begin{equation} \label{cost-ds}
J\bigl(\Hat{X}^{n}(0), Z^{n}\bigr) \;\df\;
\limsup_{T \to \infty}\;\frac{1}{T}\;
\Exp^{Z^n} \left[\int_{0}^{T}
\Hat{r}\bigl(\Hat{Q}^{n}(s),\Hat{Y}^{n}(s)\bigr)\,\D{s}\right]\,.
\end{equation} 
The associated cost minimization problem becomes 
\begin{equation*} 
\Hat{V}^{n}(\Hat{X}^{n}(0)) \;\df\;
\inf_{Z^{n} \in \Breve\fZ^{n} } J \bigl(\Hat{X}^{n}(0), Z^{n}\bigr)\,.  
\end{equation*} 

\item[\textbf{(P2)}] (\emph{constrained problem})
The objective here is to minimize the queueing cost while
imposing idleness constraints on the server pools.
Let $\Hat{r}_{\mathsf{o}}(q)$ be the
running cost function corresponding to  $\Hat{r}$ in \cref{runcost-ex} with
$\zeta \equiv 0$.
The diffusion-scaled cost criterion
$J_{\mathsf{o}}\bigl(\Hat{X}^{n}(0), Z^{n}\bigr)$ is defined
analogously to \cref{cost-ds} with running cost
$\Hat{r}_{\mathsf{o}}(\Hat{Q}^{n}(s))$, that is, 
\begin{align*} 
J_{\mathsf{o}}\bigl(\Hat{X}^{n}(0), Z^{n}\bigr)
&\;\df\;\limsup_{T \to \infty}\;\frac{1}{T}\;
\Exp^{Z^n} \left[\int_{0}^{T}
\Hat{r}_{\mathsf{o}}\bigl(\Hat{Q}^{n}(s)\bigr)\,\D{s}\right]\,.\\
\intertext{Also define}
J_{\mathsf{c},j}\bigl(\Hat{X}^{n}(0), Z^{n}\bigr)
&\;\df\;\limsup_{T \to \infty}\;\frac{1}{T}\;
\Exp^{Z^n} \left[\int_{0}^{T}
\bigl(\Hat{Y}^{n}_j(s)\bigr)^{\Tilde{m}}\,\D{s}\right]\,,\qquad j\in\cJ\,,
\end{align*}
with $\Tilde{m}\ge1$.
The associated cost minimization problem becomes 
\begin{align} 
\Hat{V}^{n}_{\mathsf{c}}(\Hat{X}^{n}(0)) &\;\df\;
\inf_{Z^{n} \in \Breve\fZ^{n}} J_{\mathsf{o}}\bigl(\Hat{X}^{n}(0), Z^{n}\bigr)\,,
\nonumber\\[5pt]
&\text{subject to}\quad
J_{\mathsf{c},j}\bigl(\Hat{X}^{n}(0), Z^{n}\bigr)
\;\le\; \updelta_j \,, \quad j \in\cJ \,,\label{constraint-ds}
\end{align}
where $\updelta = (\updelta_1, \dotsc,\updelta_J)\transp$ is a positive vector.
\end{itemize}

We refer to $\Hat{V}^{n}(\Hat{X}^{n}(0))$ and
$\Hat{V}^{n}_{\mathsf{c}}(\Hat{X}^{n}(0))$ as the diffusion-scaled optimal values
for the $n^{\rm th}$ system given the initial state $X^{n}(0)$,
for (P1) and (P2), respectively.

\begin{remark}
We choose running costs of the form \cref{runcost-ex}
mainly to simplify the exposition.
However, all the results of this paper still hold
for more general classes of functions.
Let $h_{\mathsf{o}}\colon\RR^I\to\RR_+$ be a convex function satisfying
$h_{\mathsf{o}}(x)\ge c_{1}\abs{x}^m + c_{2}$ for some $m\ge1$ and constants
$c_1>0$ and $c_2\in\RR$, and
$h\colon\RR^I\to\RR_+$, $h_i\colon\RR\to\RR_+$, $i\in\cI$, be convex functions
that have at most polynomial growth.
Then we can choose
$\Hat{r}(q,y)=h_{\mathsf{o}}(q) + h(y)$ for the unconstrained problem,
and $h_i(y_i)$ as the functions in the constraints in \cref{constraint-ds}
(with $\Hat{r}_{\mathsf{o}}=h_{\mathsf{o}}$).
Analogous running costs can of course be used in the corresponding
control problems for the limiting diffusion, which are presented later
in \cref{S3.2}.
\end{remark}

\subsection{Ergodic control problems for the limiting diffusion}\label{S3.2}

We state the two problems which correspond to (P1)--(P2)
in \cref{S3.1}
for the controlled diffusion in \cref{E-diff}.
Let $r\colon \RR^I \times \Act \to \RR$ be defined by
\begin{equation*}
r(x, u)\;=\;r\bigl(x,(u^c, u^s)\bigr)\;\df\; \Hat{r}\bigl((e\cdot x)^{+}u^c,
(e\cdot x)^{-}u^s\bigr)\,,
\end{equation*}
with $\Hat{r}$ as in \cref{runcost-ex}, that is, 
\begin{equation}\label{E-cost}
r(x, u) \;=\; [(e\cdot x)^{+}]^m \sum_{i\in\cI} \xi_i (u^c_i)^m
+  [(e\cdot x)^{-}]^m \sum_{j\in\cJ} \zeta_j (u^s_j)^m, \quad m\ge 1\,,
\end{equation}
for the given $\xi=(\xi_1,\dotsc, \xi_I)\transp$ 
and $\zeta=(\zeta_1,\dotsc,\zeta_J)\transp$ in \cref{runcost-ex}. 
Let the ergodic cost associated with the controlled diffusion $X$ and the
running cost $r$ be defined as
\begin{equation*}
J_{x,U}[r] \;\df\; \limsup_{T \to \infty}\;\frac{1}{T}\;\Exp_x^U
\left[ \int_{0}^{T} r(X_t, U_t)\,\D{t} \right]\,, \quad U \in \Uadm\,. 
\end{equation*}
\begin{itemize}
\item[\textbf{(P1$\bm'$)}]
(\emph{unconstrained problem})
The running cost function $r(x,u)$ is as in
\cref{E-cost} with $\zeta>0$. The ergodic control problem is then defined as
\begin{equation*}
\varrho^*(x) \;\df\; \inf_{U \in \Uadm} \;J_{x,U}[r] \,.
\end{equation*}

\item[\textbf{(P2$\bm'$)}]
(\emph{constrained problem})
The running cost function $r_{\mathsf{o}}(x,u)$ is as in
\cref{E-cost} with $\zeta\equiv 0$.
 Also define
\begin{equation}\label{E-rj}
r_j(x,u)\;\df\; [(e\cdot x)^{-}u^s_j]^{\Tilde{m}}\,,\quad j\in\cJ\,,
\end{equation}
with $\Tilde{m}\ge1$,
and let $\updelta=(\updelta_1,\dotsc, \updelta_J)$ be a positive vector.
The ergodic control problem under idleness constraints is defined as 
\begin{equation*}
\varrho_{\mathsf{c}}^{*}(x) \;\df\; \inf_{U \in \Uadm} \;J_{x,U}[r_{\mathsf{o}}]\,,
\qquad
\text{subject to}\quad
J_{x,U}[r_{j}] \;\le\; \updelta_j\,,\quad j \in\cJ\,.
\end{equation*}
\end{itemize}

The quantities $\varrho^*(x)$ and $\varrho^*_{\mathsf{c}}(x)$  are called
the optimal values of
the ergodic control problems (P1$'$) and (P2$'$), respectively,
for the controlled diffusion process $X$ with initial state $x$.  
Note that as is shown in Section~3 of \cite{ABP14} and
Sections~3 and~5.4 of \cite{AP15},
the optimal values $\varrho^*(x)$ and  $\varrho^*_{\mathsf{c}}(x)$
do not depend on $x\in\RR^{I}$,
and thus we remove this dependence in the results stated in \cref{S3.3}.

Let
$\eom$ denote the set of ergodic occupation measures corresponding to controls
in $\Ussm$, that is, 
\begin{equation*}
\eom\;\df\;\biggl\{\uppi\in\cP(\RR^{I}\times\Act)\,\colon\,
\int_{\RR^{I}\times\Act}\Lg^{u} f(x)\,\uppi(\D{x},\D{u})=0\quad
\forall\,f\in\Cc^{\infty}_c(\RR^{I}) \biggr\}\,,
\end{equation*}
where $\Lg^{u}f(x)$ is the controlled extended generator of the diffusion $X$
given in \cref{E-generator}.  
The restriction of the ergodic control problem with running cost $r$ to
stable stationary Markov controls is equivalent
to minimizing
$\uppi(r)\;=\;\int_{\RR^{I}\times\Act} r(x,u)\,\uppi(\D{x},\D{u})$
over all $\uppi\in\eom$.
If the infimum is attained in $\eom$, then we say that the ergodic control
problem is \emph{well posed}, and we refer to
any $\Bar\uppi\in\eom$ that attains this infimum
as an \emph{optimal ergodic occupation measure}.

The characterization of the optimal solutions to the ergodic
control problems (P1$'$)--(P2$'$)
has been thoroughly studied in \cite{ABP14} and \cite{AP15}.
We refer the reader to  these papers for relevant results used
in the proof of
asymptotic optimality which follows in the next section.

\subsection{Asymptotic optimality results} \label{S3.3}

We summarize here the main results on asymptotic optimality, which assert
that the values of the two ergodic control
problems in the diffusion scale converge to the values of the corresponding
ergodic control problems for the limiting diffusion, respectively. 
The proofs of the asymptotic optimality are given in
\cref{S-LUB}.  

Recall the definitions of $J$, $J_{\mathsf{o}}$, $\Hat{V}^{n}$, and
$\Hat{V}^{n}_{\mathsf{c}}$ in (P1)--(P2),
and the definitions of $\varrho^*$ and $\varrho^{*}_{\mathsf{c}}$ in (P1$'$)--(P2$'$).

\begin{theorem}[unconstrained problem] \label{T3.1}
Suppose that $\gamma_i>0$ for some $i\in\cI$.
Then the following are true.
\begin{itemize}
\item[\upshape{(}i\upshape{)}]  $($lower bound$)$
For any sequence $\{Z^{n},\;n\in\NN\}\subset\boldsymbol\fZ$, 
the diffusion-scaled cost in \cref{cost-ds} satisfies
\begin{equation*}
\liminf_{n\to\infty}\;
J\bigl(\Hat{X}^{n}(0),\Hat{Z}^{n}\bigr) \;\ge\; \varrho^*\,.
\end{equation*}

\smallskip

\item[\upshape{(}ii\upshape{)}]  $($upper bound$)$
$\displaystyle\limsup_{n\to\infty}\;\Hat{V}^{n}(\Hat{X}^{n}(0))
\;\le\; \varrho^*$\,.
\end{itemize}
\end{theorem}

\begin{theorem}[constrained problem] \label{T3.2}
Under the assumptions of \cref{T3.1}, we have the following:
\begin{itemize}
\item[\upshape{(}i\upshape{)}]  $($lower bound$)$
Suppose that under a sequence
$\{Z^{n},\;n\in\NN\}\subset\boldsymbol\fZ$, the constraint
in \cref{constraint-ds} is satisfied
for all sufficiently large $n\in\NN$.
Then
\begin{equation*} 
\liminf_{n\to\infty}\; J_{\mathsf{o}}\bigl(\Hat{X}^{n}(0),\Hat{Z}^{n}\bigr)
\;\ge\;\varrho^{*}_{\mathsf{c}}\,,
\end{equation*}
and as a result we have that
$~\displaystyle
\liminf_{n\to\infty}\;\Hat{V}^{n}_{\mathsf{c}}(\Hat{X}^{n}(0))
\,\ge\, \varrho_{\mathsf{c}}^{*}\,.$
\smallskip

\item[\upshape{(}ii\upshape{)}]  $($upper bound$)$
For any $\epsilon>0$, there exists a sequence
$\{Z^{n},\;n\in\NN\}\subset\boldsymbol\fZ$ such that the constraint
in \cref{constraint-ds} is feasible for all sufficiently large $n$,
and
\begin{equation*}
\limsup_{n\to\infty}\; J_{\mathsf{o}}\bigl(\Hat{X}^{n}(0),\Hat{Z}^{n}\bigr)
\;\le\;\varrho^{*}_{\mathsf{c}}+\epsilon\,.
\end{equation*}
Consequently, we have that
$~\displaystyle
\limsup_{n\to\infty}\;\Hat{V}^{n}_{\mathsf{c}}(\Hat{X}^{n}(0))
\;\le\; \varrho_{\mathsf{c}}^{*}$\,.
\end{itemize}
\end{theorem}

\medskip
\section{BQBS Stability and Fairness} \label{S4}

\subsection{BQBS stable networks}

It follows by \cref{E-BQYBS} that the controlled
diffusion limit for multiclass multi-pool networks have the following
property.
If under some admissible control (admissible scheduling policy)
the mean empirical value of some power $\kappa\ge1$
of the queueing and
idleness processes is bounded, then the corresponding mean empirical value
of the state process also remains bounded.
This property also holds for the diffusion-scaled processes in the
$n^{\text{th}}$ system,
as shown later in \cref{P6.1}.

There is however a large class of networks that share a more specific property,
namely that
the average value of any moment of a state process, is controlled by
the average value of the corresponding moment of the queueing process alone.
More precisely, the limiting diffusion of this class of networks
satisfies
\begin{equation} \label{E4.1}
\Exp^{U}_{x}\biggl[\int_{0}^{T}\abs{X_{s}}^{\kappa}\,\D{s}\biggr]
\;\le\; C\, \abs{x}^\kappa
+ C\,\Exp^{U}_{x}\biggl[\int_{0}^{T}
\bigl[1+\bigl(e\cdot X_s\bigr)^+\bigr]^{\kappa}\,\D{s}\biggr]
\qquad\forall\,T>0\,,\quad\forall\,U\in\Uadm\,,
\end{equation}
for any $\kappa\ge1$, and for a constant $C$ which depends only on $\kappa$.
We refer to the class of networks which
satisfy \cref{E4.1} as \emph{bounded-queue, bounded-state} (BQBS) stable.

Define
\begin{equation}\label{E-cone+}
\cK_{\delta,+}\;\df\;
\bigl\{ x\in\RI\,\colon e\cdot x > \delta \abs{x}\bigr\}\,.
\end{equation}
It follows by the proof of \cite[Theorem~3.1]{ABP14} that
a sufficient condition for BQBS stability is that
\cref{E-structural} holds with $\cK_{\delta}$ replaced by $\cK_{\delta,+}$, i.e.,
\begin{equation}\label{E4.3}
b(x,u)\cdot\nabla \Lyap_{\kappa,\beta}(x)\;\le\;
c_0 + c_1 \Lyap_{\kappa,\beta}(x)\,\Ind_{\cK_{\delta,+}}(x)
- c_2 \Lyap_{\kappa,\beta}(x)\,\Ind_{\cK^c_{\delta,+}}(x)
\qquad\forall\, (x,u)\in\RR^I\times\Act\,.
\end{equation}
As shown later in \cref{P6.2}, the inequality
in \cref{E4.3} is sufficient for \cref{E4.1} to
hold for the $n^{\text{th}}$ system, uniformly in $n\in\NN$.
The class of networks which satisfy \cref{E4.3},
and are therefore BQBS stable, includes the following special classes: 
\begin{enumerate}
\item[(i)]
\emph{Networks with a single dominant class}:
there is only one class of jobs that can be served by more than
one server pools (see Corollary~4.2 in \cite{AP15}).
This includes the standard ``N" and  ``W" networks, the generalized ``N"
and ``W" networks,  and more general networks as depicted in \cref{fig-networks}.
\begin{figure}[ht]
\centering
    \includegraphics[width=0.47\textwidth]{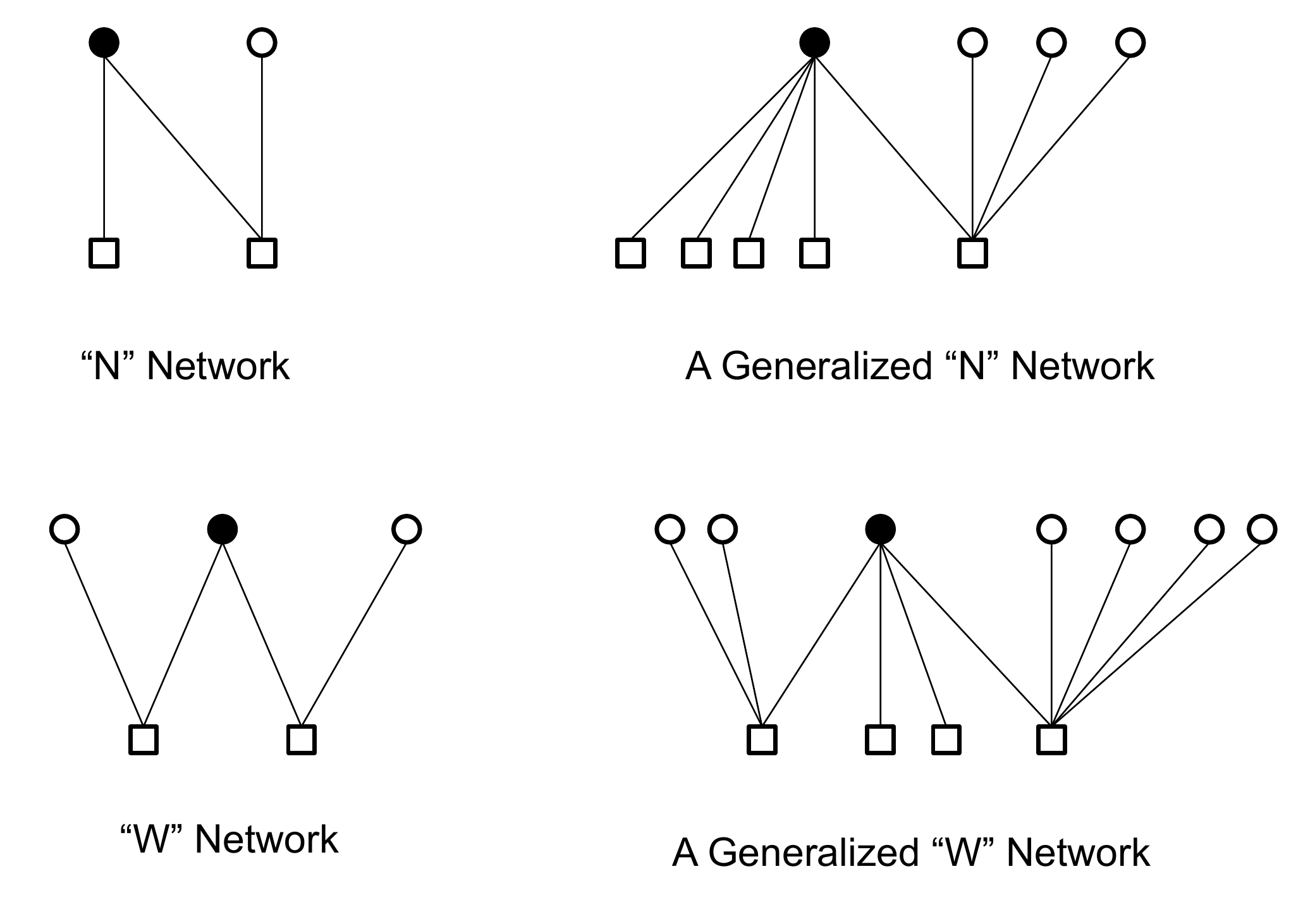}
 \ \ \ \   \includegraphics[width=0.438\textwidth]{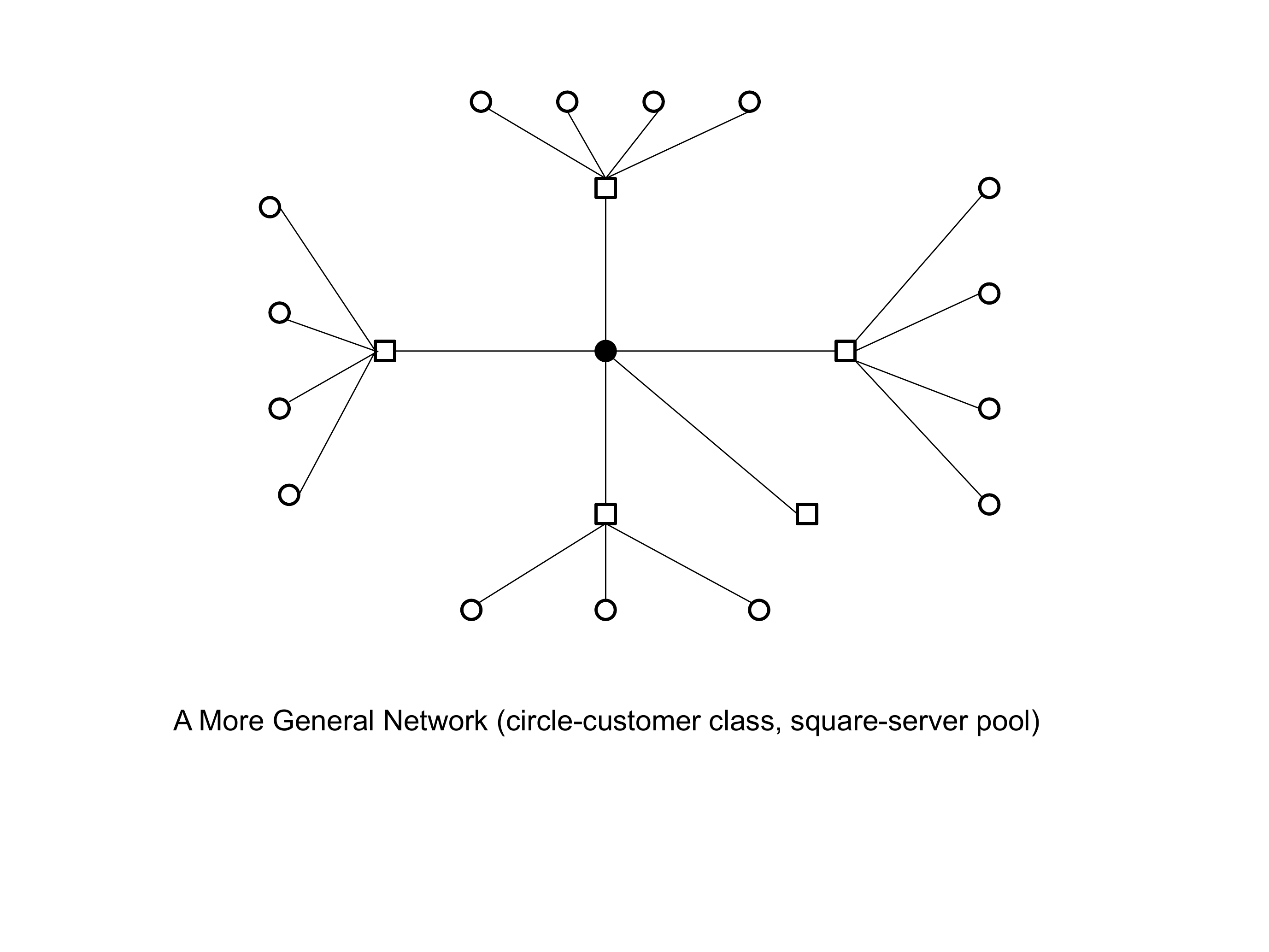}
 \caption{Examples of networks with a single dominant class} \label{fig-networks}
\end{figure}

\item[(ii)] \emph{Networks with the following parameter assumptions}:
\begin{equation*} 
\max_{i,i'\in\cI,\,j\in\cJ(i)}\;
\abs{\mu_{ij}-\mu_{i'j}}\;\le\; \Tilde\delta\,
\max_{i\in\cI,\,j\in\cJ}\;\{\mu_{ij}\}\,,
\end{equation*}
for sufficiently small $\Tilde{\delta}>0$.
This includes networks with pool-dependent service rates, i.e.,
$\mu_{ij} = \bar{\mu}_j$ for all $(i,j) \in \cE$,
 as a special class.
(See Corollary~4.1 in \cite{AP15}).
\end{enumerate}

\begin{remark}\label{R4.1}
For networks that satisfy \cref{E4.3}, ergodic control problems with
a running cost penalizing only the queue are well posed
(e.g., we may allow $\zeta=0$ in \cref{runcost-ex}).
This is because, in the diffusion scale,
the average value of the state process is controlled
by the average value of the queue, and also by the fact,
as shown in \cref{L2.2}, that idleness is upper bounded by some multiple of
the state.
\end{remark}

\subsection{The fairness problem} \label{S4.2}

In addition to ergodic control problems as in (P1)--(P2),
for BQBS stable networks we can also consider constrained problems which
aim at balancing idleness among the server pools, and result
in a fair allocation of idle servers.
Let
\begin{equation*}
\cS^J\;\df\;\{\uptheta\in(0,1)^{J}\;\colon\, e\cdot \uptheta=1\}\,.
\end{equation*}
For the $n^{\rm th}$ system, we formulate this type
of ergodic control problems as follows.
\smallskip

\textbf{(P3)} (\emph{fairness})
Here we minimize the queueing cost while keeping the average idleness
of the server pools balanced.
Let $\uptheta=(\uptheta_1,\dotsc,\uptheta_J)\transp\in\cS^{J}$ be a positive
vector and let $1\le\Tilde{m}< m$.
Let $\Bar{J}_{\mathsf{c}}\df\sum_{\jmath\in\cJ}J_{\mathsf{c},\jmath}$.
The associated cost minimization problem becomes 
\begin{align*} 
\Hat{V}^{n}_{\mathsf{f}}(\Hat{X}^{n}(0)) &\;\df\;
\inf_{Z^{n} \in \Breve\fZ^{n}} J_{\mathsf{o}}\bigl(\Hat{X}^{n}(0), Z^{n}\bigr)\,
\nonumber\\[5pt]
\text{subject to}\quad
J_{\mathsf{c},j}\bigl(\Hat{X}^{n}(0), Z^{n}\bigr)
&\;=\;\uptheta_j\,
\Bar{J}_{\mathsf{c}} \bigl(\Hat{X}^{n}(0), Z^{n}\bigr)\,,
\qquad j\in\cJ\,.
\end{align*}

\smallskip

For the corresponding diffusion, we have the following cost minimization problem. 

\textbf{(P3$\bm'$)} (\emph{fairness})
The running costs $r_{\mathsf{o}}$, and $r_j$, $j\in\cJ$, are as in (P2$'$).
Let $\uptheta=(\uptheta_1,\dotsc,\uptheta_J)\transp\in\cS^{J}$ be a positive
vector, and $1\le\Tilde{m}< m$.
The ergodic control problem under idleness fairness is defined as
\begin{equation} \label{E-P3'}
\begin{split} 
\varrho_{\mathsf{f}}^{*}(x) &\;=\; \inf_{U \in \Uadm} \;J_{x,U}[r_{\mathsf{o}}]\\[5pt]
\text{subject to}\quad
J_{x,U}[r_{j}] &\;=\; \uptheta_j\, \sum_{\jmath\in\cJ} J_{x,U}[r_{\jmath}]\,,
\qquad j\in\cJ\,. 
\end{split}
\end{equation}

We next state an optimality result for the fairness problem (P3$\bm'$). 
We first introduce some notation. 
Let 
\begin{equation} \label{E-H}
H_r(x,p)\;\df\;\min_{u\in\Act}\;\bigl[b(x,u)\cdot p + r(x,u)\bigr]\,.
\end{equation}
For $\uptheta=(\uptheta_1,\dotsc,\uptheta_J)\transp \in\RR_{+}^J$ and
$\uplambda=(\uplambda_{1},\dotsc, \uplambda_J)\transp\in\RR^{J}_{+}$,
define the running cost $h_{\uptheta,\uplambda}$ by
\begin{equation*}
h_{\uptheta,\uplambda} (x,u)\;\df\; r_{\mathsf{o}}(x,u) +
\sum_{j\in\cJ}
\uplambda_j\bigl(r_{j}(x,u)-\uptheta_j\, \Bar{r}(x,u)\bigr)\,,
\end{equation*}
where $\Bar{r} = r_1+\dotsb+r_J$.
We also let
\begin{equation*}
\sH_{\mathsf{f}}(\uptheta)\;\df\; \bigl\{\uppi\in\eom\;\colon\,
\uppi(r_{j}) = \uptheta_{j} \uppi(\Bar{r})\,,\; j\in\cJ\bigr\}\,,
\quad \uptheta  \in\RR_{+}^J\,. 
\end{equation*}

The following theorem characterizes of the optimal
solution of (P3$'$)---see Theorem~5.8 in \cite{AP15} and Theorem~4.3 in \cite{AP16}.
The existence of solutions to the HJB equation is proved for the diffusion
control problem of the ``N" network, but the argument used in the proof
is applicable to the general multiclass multi-pool model discussed here.
The uniqueness of the solutions $V_{\mathsf{f}}$ follows exactly as in the proof of
Theorem~3.2 in \cite{AP15}. 
\begin{theorem}\label{T4.1}
Suppose that the network is BQBS stable and $\gamma_i>0$ for some $i\in\cI$.
Then the constraint in \cref{E-P3'} is feasible
for any positive vector
$\uptheta=(\uptheta_1,\dotsc,\uptheta_J)\transp\in\cS^{J}$.
In addition, the following hold.
\begin{itemize}
\item[\upshape{(}a\upshape{)}]
There exists 
$\uplambda^{*}\in\RR^{J}_{+}$ such that
\begin{equation*}
\inf_{\uppi\,\in\,\sH(\uptheta)}\;\uppi(r_{\mathsf{o}}) \;=\;
\inf_{\uppi\,\in\,\eom}\;\uppi(h_{\uptheta,\uplambda^{*}})
\;=\;\varrho^*_{\mathsf{f}}\,.
\end{equation*}
\item[\upshape{(}b\upshape{)}]
If $\uppi^{*}\in\sH(\uptheta)$
attains the infimum of $\uppi\mapsto\uppi(r_{\mathsf{o}})$
in $\sH(\uptheta)$, then
$\uppi^{*}(r_{\mathsf{o}})\;=\;\uppi^{*}(h_{\uptheta,\uplambda^{*}})$,
and
\begin{equation*}
\uppi^{*}(h_{\uptheta,\uplambda})\;\le\;\uppi^{*}(h_{\uptheta,\uplambda^{*}})
\;\le\;
\uppi(h_{\uptheta,\uplambda^{*}})
\qquad\forall\,(\uppi,\uplambda)\in\eom\times\RR^{J}_{+}\,.
\end{equation*}
\item[\upshape{(}c\upshape{)}]
There exists  $V_{\mathsf{f}}\in\Cc^{2}(\RR^I)$ satisfying 
\begin{equation*}
\min_{u\in\Act}\;\bigl[\Lg^{u}V_{\mathsf{f}}(x)
+ h_{\uptheta,\uplambda^{*}}(x,u)\bigr]
\;=\; \uppi^{*}(h_{\uptheta,\uplambda^{*}}) = \varrho^*_{\mathsf{f}}\,,\quad
x\in\RR^{I}\,.
\end{equation*}
\item[\upshape{(}d\upshape{)}]
A stationary Markov control $v_{\mathsf{f}}\in\Ussm$
is optimal if and only if it satisfies 
\begin{equation*}
H_{h_{\uptheta,\uplambda^{*}}}\bigl(x,\nabla V_{\mathsf{f}}(x)\bigr) \;=\;
b\bigl(x, v_{\mathsf{f}}(x)\bigr)\cdot \nabla V_{\mathsf{f}}(x)
+ h_{\uptheta,\uplambda^{*}}\bigl(x,v_{\mathsf{f}}(x)\bigr)
\quad\text{a.e.~in~} \RR^{I}\,,
\end{equation*}
where $H_{h_{\uptheta,\uplambda^{*}}}$ is defined in \cref{E-H} with $r$
replaced by $h_{\uptheta,\uplambda^{*}}$.
\item[\upshape{(}e\upshape{)}]
The map
$\uptheta\mapsto\inf_{\uppi\,\in\,\sH(\uptheta)}\;\uppi(r_{\mathsf{o}})$
is continuous at any feasible point $\Hat\uptheta$.
\end{itemize} 
\end{theorem}

\smallskip

We next state the asymptotic optimality result for this class of networks. 

\begin{theorem} \label{T4.2}
For the class of  networks which satisfy \cref{E4.3}, and
$\gamma_i>0$ for some $i\in\cI$, the following hold.
\begin{itemize}
 \item [\upshape{(}i\upshape{)}] $($lower bound$)$
there exists a positive constant $C_{\mathsf{f}}$ such that  if a sequence
$\{Z^{n},\;n\in\NN\}\subset\boldsymbol\fZ$ satisfies
\begin{equation}\label{ET4.2A}
\max_{j\in\cJ}\;\babss{ \frac{
J_{\mathsf{c},j}\bigl(\Hat{X}^{n}(0), Z^{n}\bigr)}
{ \Bar{J}_{\mathsf{c}}\bigl(\Hat{X}^{n}(0), Z^{n}\bigr)}  - \uptheta_j } \;\le\;
\epsilon\,
\end{equation}
for some $\epsilon>0$ and for all sufficiently large $n\in\NN$, then
\begin{equation*}
\liminf_{n\to\infty}\;J_{\mathsf{o}}(\Hat{X}^n(0),Z^n) \;\ge\;
\varrho^{*}_{\mathsf{f}}  - C_{\mathsf{f}}\epsilon \,; 
\end{equation*}

\smallskip
\item [\upshape{(}ii\upshape{)}] $($upper bound$)$
for any $\epsilon>0$, there exists a sequence
$\{Z^{n},\;n\in\NN\}\subset\boldsymbol\fZ$ such that
\cref{ET4.2A} holds for all sufficiently large $n\in\NN$, and
\begin{equation*}
\limsup_{n\to\infty}\;J_{\mathsf{o}}(\Hat{X}^n(0),Z^n) \;\le\;
\varrho^{*}_{\mathsf{f}} + \epsilon\,.
\end{equation*}
\end{itemize}
\end{theorem}

\begin{remark}
The reader will certainly notice that whereas \cref{T4.1} holds for
BQBS stable networks in general, i.e., networks which satisfy
\cref{E4.1}, the Foster--Lyapunov condition \cref{E4.3} is assumed
in \cref{T4.2} which asserts asymptotic optimality.
The reason behind this, is that the corresponding BQBS stability
property should hold for the diffusion-scaled processes in order
to establish the lower bound, and \cref{E4.3}
needs to be invoked in order to assert this property
(see \cref{P6.2} in \cref{S6}).
However, \cref{E4.3} is quite natural for the models considered here.
\end{remark}

\section{A family of stabilizing policies} \label{S5}

We introduce a class of stationary Markov scheduling policies for the general
multiclass multi-pool networks which is stabilizing for the diffusion-scaled state
processes in the H--W regime. 
Let $\cI_\circ\df\{i\in\cI\colon\gamma_i=0\}$.
Throughout this section we fix a collection
 $\{N_{ij}^n\in\NN\,,\;(i,j)\in\cE\,,n\in\NN\}$ which satisfies
$$\lfloor \xi_{ij}^* N_j^n \rfloor \le N_{ij}^n
\le \lceil \xi_{ij}^* N_j^n \rceil\,,\qquad\text{and\ \ }
\sum_{i\in \cI(j)}N_{ij}^n=N_j^n\,.$$
We also define $\Bar{N}_i^n\df \sum_{j\in \cI(i)}N_{ij}^n$ for $i\in\cI$.

\begin{lemma} \label{L5.1}
Suppose that $\cI_\circ\ne\cI$.
Then, given $\Tilde{C}_0>0$, there exist a collection
\begin{equation*}
\{\Tilde{N}_{ij}^n\in\NN\,,\;(i,j)\in\cE\,,\;n\in\NN\}\,,
\end{equation*}
and a positive constant $\Hat{C}_0$ satisfying
\begin{align}
\sum_{j\in\cJ(i)}\mu_{ij}^n (\Tilde{N}_{ij}^n-N_{ij}^n) & \;\ge\; 
2\Tilde{C}_0\,\sqrt{n}  \qquad
\text{if\ } i\in\cI_\circ\,,\label{EL5.1A}\\[5pt]
\babs{N_{ij}^n-\Tilde{N}_{ij}^n} &\;\le\;
\Hat{C}_0\sqrt{n}\qquad\forall (i,j)\in\cE\,,\label{EL5.1B}
\end{align}
and
\begin{equation}\label{EL5.1C}
\sum_{i\in \cI(j)}\Tilde{N}_{ij}^n=N_j^n\qquad\forall\,j\in\cJ
\end{equation}
for all sufficiently large $n\in\NN$.
\end{lemma}

\begin{proof}
Suppose, without loss of generality (WLOG),  that $\cI_\circ=\{1,\dotsc, I-1\}$.
We claim that there exists a collection
$\{\psi_{ij}\in\RR\colon j\in\cJ(i)\,,\; i\in\cI\}$
of real numbers and a constant $C_0>0$, satisfying
$\sum_{j\in\cJ(i)}\mu_{ij}\psi_{ij}>C_0$ for all $i\in\cI_\circ$,
and
\begin{equation}\label{PL5.1A}
\sum_{i\in\cI(j)}\psi_{ij}\;=\;0\qquad\forall\,j\in\cJ\,.
\end{equation}
To prove the claim we use the argument of contradiction.
If not, then by \cite[Theorem~21.1]{Rock46} there exists a collection
of nonnegative real numbers $\varkappa_i$, $i=1,\dotsc,\cI$, such that
\begin{equation}\label{PL5.1B}
\sum_{i\in\cI_\circ} \varkappa_i\sum_{j\in\cJ(i)}\mu_{ij}\psi_{ij}\;\le\;0
\end{equation}
for all $\{\psi_{ij}\}$ satisfying \cref{PL5.1A}, and $\varkappa_{\Hat\imath}>0$
for some $\Hat\imath\in\cI_\circ$.
Since $\cG$ is a tree, there exists a pair of finite sequences
$i_1,\dotsc,i_\ell$ and
$j_1,\dotsc,j_{\ell-1}$ such that $\Hat\imath=i_1$, $i_\ell=I$,
and $i_k\sim j_k$, $i_{k+1}\sim j_k$ for $k=1,\dotsc,\ell-1$.
Choosing $\psi_{i_\ell j_{\ell-1}}=-1$, $\psi_{i_{\ell-1} j_{\ell-1}}=1$,
$\psi_{i j_{\ell-1}}=0$ if $i\notin\{i_{\ell-1},i_\ell\}$,
and $\psi_{i j}=0$ if $j\ne j_{\ell-1}$, it follows from \cref{PL5.1B} that
$\varkappa_{i_{\ell-1}}=0$. 
Thus, WLOG, we may suppose that
$\varkappa_{i_2}=0$.
But then replacing $\psi_{i_1,j_1}$ with $\psi_{i_1,j_1}+C$,
and $\psi_{i_2,j_1}$ with $\psi_{i_2,j_1}-C$, the new set of numbers
$\{\psi_{ij}\}$ satisfies \cref{PL5.1A}. Therefore, by \cref{PL5.1B}
we must have
\begin{equation*}
\sum_{i\in\cI_\circ} \varkappa_i\sum_{j\in\cJ(i)}\mu_{ij}\psi_{ij}
+ \varkappa_{i_1} \mu_{i_1 j_1} C\;\le\;0
\end{equation*}
for all $C\in\RR$, which is impossible since $\varkappa_{i_1} \mu_{i_1 j_1}>0$.
This proves the claim.

Scaling $\{\psi_{ij}\}$ by multiplying with a constant, we may assume that
\begin{equation}\label{PL5.1C}
\sum_{j\in\cJ(i)}\mu_{ij}\psi_{ij}\;>\;3\,\Tilde{C}_0\qquad\forall\,i\in\cI_\circ\,.
\end{equation}
For each $j\in\cJ$,
if $\cI(j)$  is a singleton, i.e., $\cI(j)=\{i_1\}$, then
we define $\Tilde{N}_{i_1j}^n \df N_{i_1j}^n$.
Otherwise, if $\cI(j)=\{i_1,\dotsc,i_\ell\}$, then we let
$\Tilde{N}_{i_kj}^n \df N_{i_kj}^n+\bigl\lfloor \psi_{i_kj}\sqrt{n}\bigr\rfloor$
for $k=1,\dotsc,\ell-1$, and
$\Tilde{N}_{i_\ell j}^n\df N_j^n -\sum_{k=1}^{\ell-1}\Tilde{N}_{i_kj}^n$.
It is clear then that \cref{EL5.1A} holds for all sufficiently large $n$
by \cref{PL5.1C}, while \cref{EL5.1B,EL5.1C} hold by construction.
This completes the proof.
\end{proof}

\begin{definition}\label{D-SDPf}
Let
$\{\Tilde{N}_{ij}^n\}$ be as in \cref{L5.1},
and $\Tilde{N}_i^n\df \sum_{j\in \cI(i)} \Tilde{N}_{ij}^n$ for $i\in\cI$.
Let $\sZ^n$ denote the class of Markov policies $z$ satisfying
\begin{align*}
z_{ij}(x) &\;\le\; \Tilde{N}_{ij}^n \quad\forall\, i\sim j\,,
\quad\text{and}~ \sum_{j\in\cJ(i)} z_{ij}(x)=x_i\,,
\qquad \text{if~} x_i \;\le\; \Tilde{N}_i^n\\[5pt]
z_{ij}(x) &\;\ge\; \Tilde{N}_{ij}^n \quad \forall\, i\sim j\,,
\qquad \text{if~} x_i \;>\; \Tilde{N}_i^n\,.
\end{align*}
We refer to this class of  Markov policies as
\emph{balanced saturation policies} (BSPs). 
\end{definition}

We remark that if all $\gamma_i>0$ for $i \in \cI$, then in \cref{D-SDPf},
we may replace $\Tilde{N}_{ij}^n$ and $\Tilde{N}^n_i$ by $N^n_{ij}$ and
$\Bar{N}^n_i$, respectively.
Note that by \cref{L5.1}, the quantities $\Tilde{N}^n_{ij}$ and $\Tilde{N}_i$
are within $\order(\sqrt{n})$ of the quantities $N^n_{ij}$ and $\Bar{N}_i$,
which can be regarded as the `steady-state' allocations for the $n^{\text{th}}$ system.
Thus, in the class of BSPs, if $\gamma_i>0$ for some $i$, then the scheduling policy
$z$ is determined using the `shifted' steady-state allocations
$\Tilde{N}^n_{ij}$ and $\Tilde{N}_i$. 

Note that the stabilizing policy for the `N' network in \cite{AP15} belongs
to the class of BSPs.
As another example, for the `M' network, if $\gamma_i>0$ for some $i=1,2$,
the scheduling policy
$z = z(x)$, $x\in \ZZ^2_+$, defined by
\begin{align*}
& z_{11} \;=\; x_1 \wedge N_1^n\,, \\
& z_{12} \;=\; \begin{cases} 
(x_1 - N_1^n)^+ \wedge \Tilde{N}_{12}^n &  \text{if} \  x_2 \ge \Tilde{N}^n_2 \\
(x_1 - N_1^n)^+ \wedge (x_2 - N^n_3)^+ & \text{otherwise} \,, 
\end{cases} \\
&z_{22} \;=\; \begin{cases} 
(x_2 - N_3^n)^+ \wedge \Tilde{N}_{22}^n & \text{if} \  x_1 \ge \Tilde{N}^n_1 \\
(x_2 - N_3^n)^+ \wedge (x_1 - N^n_1)^+ & \text{otherwise} \,,
\end{cases}  \\
& z_{23} \;=\; x_2 \wedge N^n_3\,,
\end{align*}
is a BSP. 
If $\gamma_i>0$  for $i=1,2$, then in the scheduling policy above,
we can replace $\Tilde{N}^n_{ij}$ and $\Tilde{N}_i^n$ by $N^n_{ij}$
and $\bar{N}_i^n$, respectively. 

Using the function $\Hat{x}^n$ in \cref{D-hatx},
we can write the generator $\widehat\cL_n^{z}$ of the diffusion-scaled
state process $\Hat{X}^n$ under the policy $z \in \sZ^n$  as 
\begin{equation} \label{E-gen2} 
\widehat\cL_n^{z} f(\Hat{x})
\;=\; \cL_n^{z} f\bigl(\Hat{x}^n(x)\bigr)\,,
\end{equation}
where $\cL_n^{z}$ is as defined in \cref{E-cL}.

Recall that a $\Rd$-valued Markov process $\{M_t\colon t\ge 0\}$ is called
\emph{exponentially ergodic} if it possesses an invariant probability 
measure $\overline\uppi(\D y)$ satisfying
\begin{equation*}
\lim_{t\to\infty}\;e^{\kappa t}\,\bnorm{P_t(x,\cdot)
-\overline\uppi(\cdot)}_{\mathrm{TV}} \;=\;0\qquad \forall\,x \in \Rd\,,
\end{equation*}
for some $\kappa>0$, 
where $P_t(x,\cdot)\df \mathbb{P}^x(M_t\in \cdot)$ denotes the transition probability
of $M_t$,  
and $\norm{\cdot}_{\mathrm{TV}}$ denotes the total variation norm. 

\begin{proposition}\label{P5.1}
Let $\widehat\cL_n^z$ denote the generator
of the diffusion-scaled state process $\Hat{X}^{n}$ under a BSP $z\in\sZ^n$.
Let $\widetilde\Lyap_{\epsilon,\beta}$ be as in \cref{Lyapk}, with $\beta\in\RR^I$
a positive vector. 
There exists $\epsilon>0$, and positive constants $C_0$ and $C_1$ such that
\begin{equation} \label{EP5.1A}
\widehat\cL_n^z\, \widetilde\Lyap_{\epsilon,\beta}(\Hat{x}) \;\le\;
C_0  - C_1\, \widetilde\Lyap_{\epsilon,\beta}(\Hat{x})
\qquad\forall\,\Hat{x}\in \sS^n \,,\quad\forall\,n\ge n_0\,.
\end{equation} 
The process $\Hat{X}^n$ is exponentially ergodic and admits a unique
invariant probability measure $\Hat\uppi^n$ satisfying
\begin{equation}\label{EP5.1B}
\lim_{t\to\infty}\, e^{\kappa t} \,\bnorm{P^n_t (x, \cdot)
- \Hat{\uppi}^n(\cdot)}_{\mathrm{TV}} \;=\;0\,, \quad x \in \RR^I\,,
\end{equation}
for any $\kappa<C_1$,
where $P^n_t (x, \cdot)$ denotes the transition probability of $\Hat{X}^n$.
\end{proposition}

\begin{proof}
Throughout the proof we use, without further mention, the fact that there exists
a constant $\Tilde{C}_0$ such that
\begin{equation*}
\begin{split}
\babss{\lambda_i^{n} - \sum_{j\in\cJ(i)}\mu_{ij}^{n} N_{ij}^n}
&\;\le\; \Tilde{C}_0\,\sqrt{n}\,,\\[5pt]
\babs{nx_i^* -\Bar{N}_i^n}\;=\;\babss{nx_i^* - \sum_{j\in\cJ(i)} N_{ij}^n}
&\;\le\; \Tilde{C}_0\,\sqrt{n}
\end{split}
\end{equation*}
for all $i\in\cI$ and all sufficiently large $n\in\NN$.
This follows by \cref{HWpara}.

Recall the collection $\{\Tilde{N}_{ij}^n\in\NN\,,\;(i,j)\in\cE\}$
constructed in \cref{L5.1} with respect to the constant $\Tilde{C}_0$ given above.
Define
$\Breve{x}_i=\Breve{x}_i^n(x)\df x_i - \Tilde{N}_i^n$,
$\doublehat{x}_i\df \frac{1}{\sqrt{n}}\Breve{x}_i^n$,
and let
\begin{equation*}
\widehat\Lyap_{\epsilon,\beta}(x)\;\df\;
\widetilde\Lyap_{\epsilon,\beta}(\doublehat{x})\,.
\end{equation*}
Using the identity
\begin{equation}\label{PP5.1Ba}
f(x\pm e_i) - f(x) \mp \partial_i f(x)
\;=\; \int_0^1 (1-t)\, \partial_{ii} f(x\pm t e_i) \,\D{t}\,,
\end{equation}
we obtain
\begin{equation}\label{PP5.1Bd}
\Babs{\widehat\Lyap_{\epsilon,\beta}(x\pm e_i)
- \widehat\Lyap_{\epsilon,\beta}(x) \mp
\epsilon\tfrac{\beta_i}{\sqrt{n}} \doublehat{x}_i \phi_\beta(\doublehat{x})\,
\widehat\Lyap_{\epsilon,\beta}(x)}
\;\le\; \tfrac{1}{n}\epsilon^2\,\Tilde\kappa_1\,\widehat\Lyap_{\epsilon,\beta}(x)
\end{equation}
for some constant $\Tilde\kappa_1>0$, and all $\epsilon\in(0,1)$, with
\begin{equation*}
\phi_\beta(x) \;\df\; \frac{2+\norm{x}_\beta^2}{
\bigl(1+\norm{x}_\beta^2\bigr)^{\nicefrac{3}{2}}}\,. 
\end{equation*}

Fix $n\in\NN$.
By \cref{E-cL}, with 
\begin{equation*}
q_i  \;=\; q_i(x_i) \;=\; x_i - \sum_{j\in\cJ(i)} z_{ij}
\end{equation*}
for $i \in \cI$ (see \cref{Td-xn}),
and using \cref{PP5.1Bd}, we obtain
\begin{align}\label{PP5.1Be} 
\cL_n^z\, \widehat\Lyap_{\epsilon,\beta}(x) &\;\le\; \epsilon
\sum_{i\in\cI} \biggl[\lambda^{n}_i
\Bigl(\tfrac{\beta_i}{\sqrt{n}} \doublehat{x}_i \phi_\beta(\doublehat{x})
+\tfrac{1}{n}\epsilon\,\Tilde\kappa_1\Bigr)
 + \sum_{j\in \cJ(i)}\mu_{ij}^{n}z_{ij}\,
\Bigl(-\tfrac{\beta_i}{\sqrt{n}} \doublehat{x}_i \phi_\beta(\doublehat{x})
+\tfrac{1}{n}\epsilon\,\Tilde\kappa_1\Bigr) \nonumber\\[5pt]
&\mspace{300mu} +  \gamma_i^{n} q_i\,
\Bigl(-\tfrac{\beta_i}{\sqrt{n}}\doublehat{x}_i \phi_\beta(\doublehat{x})
+\tfrac{1}{n}\epsilon\,\Tilde\kappa_1\Bigr)\biggr]
\widehat\Lyap_{\epsilon,\beta}(x)
\nonumber \\[5pt]
&\;=\; \epsilon\,\widehat\Lyap_{\epsilon,\beta}(x)\,
\sum_{i\in\cI} \Bigl(\tfrac{\beta_i}{\sqrt{n}}\,
\phi_\beta(\doublehat{x})\,F_{n,i}^{(1)}(x)
+\tfrac{1}{n}\epsilon\,\Tilde\kappa_1\,F_{n,i}^{(2)}(x)\Bigr)\,,
\end{align}
where
\begin{equation}\label{PP5.1C}
\begin{split}
F_{n,i}^{(1)}(x) &\;\df\; 
\doublehat{x}_i
\biggl(\lambda_i^{n} - \sum_{j \in \cJ(i)} \mu_{ij}^{n} z_{ij}
- \gamma^{n}_i q_i \biggr)\,,\\[5pt]
F_{n,i}^{(2)}(x) &\;\df\;  \lambda_i^{n}
+ \sum_{j \in \cJ(i)} \mu_{ij}^{n} z_{ij} + \gamma^{n}_i q_i\,.
\end{split}
\end{equation}
It always holds that $z_{ij} \le x_i$
and $q_i\le x_i$ for all $(i,j)\in\cE$.
By \cref{HWpara}, for some constant
$\Tilde\kappa_2$ we have
\begin{equation}\label{PP5.1D}
\lambda_i^n+ \sum_{j \in \cJ(i)} \bigl(\mu_{ij}^{n} + \gamma^{n}_i \bigr)
\,\Tilde{N}_i^n
\;\le\; n\Tilde\kappa_2\qquad\forall\,i\in\cI\,,
\end{equation}
and all $n\in\NN$.
Thus, by \cref{PP5.1C,PP5.1D}, we obtain
\begin{align} \label{PP5.1E}
F_{n,i}^{(2)}(x) &\;\le\;  \lambda_i^{n} + \sum_{j \in \cJ(i)}
\mu_{ij}^{n} x_i + \gamma^{n}_i  x_i  \nonumber\\[5pt]
&\;=\;   \lambda_i^{n}
+ \Biggl(\sum_{j \in \cJ(i)} \mu_{ij}^{n} + \gamma^{n}_i \Biggr)
(\Tilde{N}_i^n + \Breve{x}_i) \nonumber\\
&\;\le\;  \Tilde\kappa_2\, (n+ \Breve{x}_i)\,.
\end{align}

We next calculate an estimate for $F_{n,i}^{(1)}$ in \cref{PP5.1C}.
Consider any $i\in\cI$.  Define
\begin{equation*}
\Breve{z}_{ij}\;\df\; N_j^n - \sum_{i'\ne i} z_{i'j}\,.
\end{equation*}
We distinguish three cases.

\smallskip\noindent\emph{Case~A.}
Suppose that $x_i<\Tilde{N}_i^n$.
We write
\begin{equation*}
\sum_{j\in\cJ(i)}\mu_{ij}^n z_{ij}^n \;=\;
\sum_{j\in\cJ(i)}\mu_{ij}^n \Tilde{N}_{ij}^n +
\sum_{j\in\cJ(i)}\mu_{ij}^n \bigl(z_{ij}^n-\Tilde{N}_{ij}^n\bigr)\,.
\end{equation*}
Note that $z_{ij}^n-\Tilde{N}_{ij}^n\le0$ and
$\Breve{x}_i\le 0$.
Therefore, we have
\begin{equation*}
- \Breve{x}_i\sum_{j\in\cJ(i)}\mu_{ij}^n \bigl(z_{ij}^n-\Tilde{N}_{ij}^n\bigr)
\;\le\;
- \Breve{x}_i
\biggl(\min_{j\in\cJ(i)}\,\mu_{ij}^n\biggr)
\bigl(x_i-\Tilde{N}_i^n)
\;=\; -  \biggl(\min_{j\in\cJ(i)}\,\mu_{ij}^n\biggr)\,\abs{\Breve{x}_i}^2\,.
\end{equation*}
We also have that
\begin{equation} \label{PP5.1F}
\lambda_i^n - \sum_{j\in\cJ(i)}\mu_{ij}^n \Tilde{N}_{ij}^n
\;=\; \lambda_i^n - \sum_{j\in\cJ(i)}\mu_{ij}^n N_{ij}^n
- \sum_{j\in\cJ(i)}\mu_{ij}^n (\Tilde{N}_{ij}^n-N_{ij}^n)
\;\le\;  - \Tilde{C}_0\,\sqrt{n}\,.
\end{equation}
Thus we obtain
\begin{align*}
F_{n,i}^{(1)}(x) &\;\le\;
- \Tilde{C}_0\,\sqrt{n}\,\doublehat{x}_i
-\sqrt{n}\,\biggl(\min_{j\in\cJ(i)}\,\mu_{ij}^n\biggr)\,\abs{\doublehat{x}_i}^{2}\,.
\end{align*}

\smallskip\noindent\emph{Case~B.}
Suppose that
\begin{equation*}
x_i\;\ge\;\Tilde{N}_i^n\qquad\text{and}\qquad
 x_i \;\ge\; \sum_{j\in\cJ(i)} \Breve{z}_{ij}\,.
\end{equation*}
Define
\begin{equation*}
\Breve{\zeta}_{ij}\;\df\; \Breve{z}_{ij}- \Tilde{N}_{ij}^n\,,
\end{equation*}
and note that $\Breve{\zeta}_{ij}\ge0$.
Then
$z_{ij} = \Breve{z}_{ij}=\Tilde{N}_{ij}^n+\Breve{\zeta}_{ij}$. 

Suppose first that $\gamma_i=0$. 
Then by \cref{PP5.1F}, we have
\begin{align*}
F_{n,i}^{(1)}(x) &\;\le\;  - \Tilde{C}_0\,\sqrt{n}\,\doublehat{x}_i
+ \doublehat{x}_i
\Biggl[\lambda_i^n - \sum_{j\,\in\,\cJ(i)} \mu_{ij}^n\,\Breve{z}_{ij}\Biggr]
\nonumber\\[5pt]
&\;\le\;  - \Tilde{C}_0\,\sqrt{n}\,\doublehat{x}_i
+ \doublehat{x}_i
\Biggl[\lambda_i^n - \sum_{j\,\in\,\cJ(i)} \mu_{ij}^n\,(\Tilde{N}_{ij}^n
+\Breve{\zeta}_{ij})\Biggr]
\nonumber\\[5pt]
&\;\le\;  - 2 \Tilde{C}_0\,\sqrt{n}\,\doublehat{x}_i
- \doublehat{x}_i \Biggl[\sum_{j\,\in\,\cJ(i)} \mu_{ij}^n\,\Breve{\zeta}_{ij}\Biggr]
\nonumber\\[5pt]
&\;\le\; - 2\Tilde{C}_0\,\sqrt{n}\,\doublehat{x}_i\,.
\end{align*}
Suppose now that $\gamma_i>0$. Then by \cref{PP5.1F}, we have
\begin{align*}
F_{n,i}^{(1)}(x) &\;\le\;  - \Tilde{C}_0\,\sqrt{n}\,\doublehat{x}_i
+ \doublehat{x}_i \Biggl[\lambda_i^n - \sum_{j\,\in\,\cJ(i)} \mu_{ij}^n\,\Breve{z}_{ij}
- \gamma_i^n \Biggl(x_i - \sum_{j\,\in\,\cJ(i)}\Breve{z}_{ij}\Biggr)\Biggr]
\nonumber\\[5pt]
&\;\le\;  - \Tilde{C}_0\,\sqrt{n}\,\doublehat{x}_i
+ \doublehat{x}_i \Biggl[\lambda_i^n - \sum_{j\,\in\,\cJ(i)} \mu_{ij}^n\,
(\Tilde{N}_{ij}^n+\Breve{\zeta}_{ij})
-\gamma_i^n \Biggl(\Breve{x}_i - \sum_{j\,\in\,\cJ(i)}\Breve{\zeta}_{ij}\Biggr)\Biggr]
\nonumber\\[5pt]
&\;\le\;  - 2 \Tilde{C}_0\,\sqrt{n}\,\doublehat{x}_i
- \doublehat{x}_i \Biggl[\sum_{j\,\in\,\cJ(i)} \mu_{ij}^n\,\Breve{\zeta}_{ij}
+\gamma_i^n
\Biggl(\doublehat{x}_i - \sum_{j\,\in\,\cJ(i)}\Breve{\zeta}_{ij}\Biggr)\Biggr]
\nonumber\\[5pt]
&\;\le\; - 2\Tilde{C}_0\,\sqrt{n}\,\doublehat{x}_i
- \sqrt{n}\,
\Bigl(\gamma_i^n\wedge\min_{j\in\cJ(i)}\,\mu_{ij}^n\Bigr)\,\abs{\doublehat{x}_i}^{2}\,.
\end{align*}

\smallskip\noindent\emph{Case~C.}
Suppose that
\begin{equation*}
x_i\;\ge\; \Tilde{N}_i^n\qquad\text{and}\qquad
 x_i \;<\; \sum_{j\in\cJ(i)} \Breve{z}_{ij}\,.
\end{equation*}
Let $\Hat\jmath\in\cJ(i)$ be arbitrary.
We have
\begin{align*}
F_{n,i}^{(1)}(x) &\;\le\; -\Tilde{C}_0\,\sqrt{n}\,\doublehat{x}_i
+ \doublehat{x}_i
\Biggl[\lambda_i^n - \sum_{j\,\in\,\cJ(i)\setminus\{\Hat\jmath\}} \mu_{ij}^n\,\Breve{z}_{ij}
- \mu_{i\Hat\jmath}^n
\Biggl(x_i - \sum_{j\,\in\,\cJ(i)\setminus\{\Hat\jmath\}}\Breve{z}_{ij}\Biggr)\Biggr]
\nonumber\\[5pt]
&\;\le\; - \Tilde{C}_0\,\sqrt{n}\,\doublehat{x}_i
+ \doublehat{x}_i \Biggl[\lambda_i^n - \sum_{j\,\in\,\cJ(i)\setminus\{\Hat\jmath\}}
\mu_{ij}^n\,(\Tilde{N}_{ij}^n+\Breve{\zeta}_{ij})
-\mu_{i\Hat\jmath}^n \Biggl(\Tilde{N}_{i\Hat\jmath}^n+\Breve{x}_i
- \sum_{j\,\in\,\cJ(i)\setminus\{\Hat\jmath\}}\Breve{\zeta}_{ij}\Biggr)\Biggr]
\nonumber\\[5pt]
&\;\le\; - 2\Tilde{C}_0\,\sqrt{n}\,\doublehat{x}_i
- \doublehat{x}_i
\Biggl[\sum_{j\,\in\,\cJ(i)\setminus\{\Hat\jmath\}} \mu_{ij}^n\,\Breve{\zeta}_{ij} +
\mu_{i\Hat\jmath}^n \Biggl(\Breve{x}_i
- \sum_{j\,\in\,\cJ(i)\setminus\{\Hat\jmath\}}\Breve{\zeta}_{ij}\Biggr)\Biggr]
\nonumber\\[5pt]
&\;\le\; - 2\Tilde{C}_0\,\sqrt{n}\,\doublehat{x}_i
- \sqrt{n}\,\biggl(\min_{j\in\cJ(i)}\,\mu_{ij}^n\biggr)\,\abs{\doublehat{x}_i}^{2}\,.
\end{align*}

From cases A--C, we obtain 
\begin{equation*}
\begin{aligned}
F_{n,i}^{(1)}(x) & \; \le\;  - 2 \Tilde{C}_0\,\sqrt{n}\,
\doublehat{x}_i\,\Ind_{\{\Breve{x}_i>0\}}
 - \sqrt{n}\,\biggl(\Tilde{C}_0\,\doublehat{x}_i
+\biggl(\min_{j\in\cJ(i)}\,\mu_{ij}^n\biggr)\,
\abs{\doublehat{x}_i}^{2}\biggr)\,\Ind_{\{\Breve{x}_i\le0\}}
\quad &\text{if\, }  \gamma_i=0\,, \\[5pt]
F_{n,i}^{(1)}(x) & \; \le\;  - \sqrt{n}\,\biggl(2 \Tilde{C}_0\,\doublehat{x}_i
+ \Bigl(\gamma_i^n\wedge\min_{j\in\cJ(i)}\,\mu_{ij}^n\Bigr)
\abs{\doublehat{x}_i}^{2}\biggr)\,\Ind_{\{\Breve{x}_i>0\}}\\[3pt]
&\mspace{200mu}
- \sqrt{n}\,\biggl(\Tilde{C}_0\,\doublehat{x}_i
+\biggl(\min_{j\in\cJ(i)}\,\mu_{ij}^n\biggr)\,
\abs{\doublehat{x}_i}^{2}\biggr)\,\Ind_{\{\Breve{x}_i\le0\}}
\quad & \text{if\, }  \gamma_i>0\,.
\end{aligned}
\end{equation*}
It is clear from these estimates
together with \cref{PP5.1Be,PP5.1E}, that, for $\varepsilon>0$ small enough,
there exist positive constants 
$M_k$, $k=0,1$, satisfying
\begin{equation*}
\cL_n^z\,\widehat\Lyap_{\epsilon,\beta}(x) \; \le\;  M_0
-M_1\,\widehat\Lyap_{\epsilon,\beta}(x)\qquad\forall\,x\in\ZZ^I_+ \,.
\end{equation*}

Define 
$\varsigma^n_i\df \sqrt{n} x_i^* - \tfrac{1}{\sqrt{n}}\Tilde{N}_i^n$. Then 
$\widehat\Lyap_{\epsilon,\beta}(x) =
\widetilde\Lyap_{\epsilon,\beta}(\Hat{x} +\varsigma^n)$.
Thus, using the definition in \cref{E-gen2,D-hatx},
we obtain
\begin{equation}\label{PP5.1I}
\widehat\cL^z_n\,\widetilde\Lyap_{\epsilon,\beta}(\Hat{x} +\varsigma^n  )  \; \le\; 
M_0-M_1\,\widetilde\Lyap_{\epsilon,\beta}(\Hat{x} +\varsigma^n)\,.
\end{equation}
Since $\abs{\varsigma^n_i}\le \Hat{C}_0$ by \cref{L5.1},
it is clear that \cref{EP5.1A} follows by \cref{PP5.1I}.

It is standard to show that \cref{EP5.1A} implies \cref{EP5.1B}.
One can apply, for example, equation (3.5) in \cite[Theorem~3.2]{Douc-09},
using $\varPsi_1(x)=x$ and $\varPsi_2(x)=1$.
This completes the proof.
\end{proof}

The following is immediate from \cref{P5.1}.

\begin{corollary}
If $z\in\sZ^n$ is a BSP, then for some $\varepsilon>0$ we have
\begin{equation} \label{EC5.1B}
\sup_{n\ge n_0} \limsup_{T \to \infty}\;\frac{1}{T}\;
\Exp^z \biggl[ \int_{0}^{T} \E^{\varepsilon\,\abs{\Hat{X}^{n}(s)}} \,\D{s} \biggr]
\;<\; \infty\,,
\end{equation}
and the same holds if we replace $\Hat{X}^n$ with
$\Hat{Q}^n$ or $\Hat{Y}^n$ in \cref{EC5.1B}.
In particular, the invariant probability
measure of the diffusion-scaled process $\Hat{X}^n(t)$ under a BSP
has finite moments of any order.
\end{corollary}

\begin{remark}
A Foster--Lyapunov property analogous to \cref{E-stablev} can be obtained for the
diffusion-scaled state process $\Hat{X}^n$ under a BSP policy. 
Let $\Lyap_{\kappa,\beta}$ be as in \cref{Lyapk} and $\widehat\cL_n^z$
as in \cref{P5.1}.
One can show, by a slight modification of the proof of \cref{P5.1},
that for each $\kappa>1$, there exist positive constants $C_0$ and $C_1$
depending only on $\kappa$ and $n_0 \in \NN$, such that for all $z\in\sZ^n$, we have 
\begin{enumerate}
\item[(i)]
If $\gamma_i >0$ for all $i \in \cI$, then 
\begin{equation*}
\widehat\cL_n^z \Lyap_{\kappa, \beta}(\Hat{x}) \;\le\;
C_0  - C_1\, \Lyap_{\kappa,\beta} (\Hat{x})
\qquad\forall\,\Hat{x}\in \sS^n \,,\quad\forall\,n\ge n_0\,.
\end{equation*}  
\item[(ii)] If $\gamma_i >0$ for some $i \in \cI$ (but not all $i$), then 
\begin{equation} \label{EP5.1A1}
\widehat\cL_n^z \Lyap_{\kappa, \beta}(\Hat{x}) \;\le\;
C_0  - C_1\, \Lyap_{\kappa-1,\beta} (\Hat{x})
\qquad\forall\,\Hat{x}\in \sS^n \,,\quad\forall\,n\ge n_0\,.
\end{equation}  
\end{enumerate}
The Foster--Lyapunov property in \cref{EP5.1A1} can be equivalently written as
\begin{equation*}
\widehat\cL_n^z \Lyap_{\kappa, \beta}(\Hat{x}) \;\le\; C_0  - C_1'\,
\bigl(\Lyap_{\kappa,\beta} (\Hat{x}) \bigr)^{\frac{\kappa-1}{\kappa}}
\qquad\forall\,\Hat{x}\in \sS^n \,,\quad\forall\,n\ge n_0\,,
\end{equation*}  
for some constant $C_1'$. 
Such Foster--Lyapunov properties appear in the studies on subexponential rate of
convergence of Markov processes (see, e.g., \cite{Douc-09}, \cite{Hairer-16}
and references therein).
Thus \cref{EP5.1A1} provides an interesting example to that rich theory.  
On the other hand,  \cref{EP5.1A} with
the exponential function $\widetilde\Lyap_{\epsilon,\beta}$ is stronger,
and implies exponential ergodicity of the processes $\hat{X}^n(t)$ under a BSP. 
\end{remark}

\section{Ergodic Properties of the \texorpdfstring{$n^{\mathrm{th}}$}{th} System}
\label{S6}

\subsection{Moment bounds for general multiclass multi-pool networks}
Recall the moment bounds for the diffusion limit in
\cref{E-BQYBS}.
We prove the analogous property for the
diffusion-scaled state process $\Hat{X}^n$.

\begin{proposition}\label{P6.1}
For any $\kappa\ge1$,
there  exist constants
$\widetilde{C}_0$ and $\widetilde{C}_1$, depending only on
$\kappa$, such that
\begin{equation}\label{EP6.1A}
\Exp^{Z^{n}}\biggl[\int_{0}^{T}\abs{\Hat{X}^{n}(s)}^\kappa\,\D{s}\biggr]
\;\le\; \widetilde{C}_0\,\bigl(T+\abs{\Hat{X}^{n}(0)}^\kappa\bigr)
+\widetilde{C}_1
\Exp^{Z^{n}}\biggl[\int_{0}^{T}
\bigl(1+\abs{\Hat{Q}^{n}(s)}+\abs{\Hat{Y}^{n}(s)}\bigr)^{\kappa}\,\D{s}\biggr]
\end{equation}
for all $n\in\NN$,
and for any sequence $\{Z^{n}\in\fZ^{n},\;n\in\NN\}$.
\end{proposition}

\begin{proof}
Let 
$\mathscr{V}(x)\;\df\; \sum_{i\in\cI} \beta_{i} \mathscr{V}_{i}(x_i)$,
$x\in \RR^I$, where $\beta_{i}$, $i\in\cI$, are positive constants
to be determined later, and
$\mathscr{V}_{i}(x_i)=\abs{x_{i}}^{\kappa}$ when $\kappa>1$, whereas 
$\mathscr{V}_{i}(x_i)=\frac{\abs{x_{i}}^{2}}{\sqrt{\delta +\abs{x_{i}}^{2}}}$
for some $\delta>0$ when $\kappa=1$.
By applying It\^o's formula on $\mathscr{V}$, we obtain from
\cref{hatXn-1} that for $t\ge 0$,
\begin{equation}\label{EP6.1B}
\Exp\bigl[\mathscr{V}(\Hat{X}^{n}(t))\bigr] \;=\;
\Exp\bigl[\mathscr{V}(\Hat{X}^{n}(0)) \bigr] +
\Exp\biggl[\int_{0}^{t} \mathscr{A}^n \mathscr{V} 
\bigl(\Hat{X}^{n}(s), \Hat{Z}^{n}(s)\bigr)\,\D{s} \biggr] 
+\Exp \Biggl[\sum_{s\le t} \mathscr{D} \mathscr{V}(\Hat{X}^n,s) \Biggr]\,,  
\end{equation}
where $\mathscr{A}^n$ is given in \cref{D-scrA}, and
$\mathscr{D} \mathscr{V}(\Hat{X}^n,s)$ is as in \cref{E-scrD}.

Let $\Hat{\Theta}^{n}\df e\cdot\Hat{Q}^{n}\wedge e\cdot\Hat{Y}^{n}$. 
Then 
$\Hat{Q}^{n}= \bigl(\Hat{\Theta}^{n}+(e\cdot\Hat{X}^{n})^{+}\bigr)\Hat{u}^{c}$ 
and $\Hat{Y}^{n}= \bigl(\Hat{\Theta}^{n}+(e\cdot\Hat{X}^{n})^{-}\bigr)\Hat{u}^{s}$
for some $(\Hat{u}^c,\Hat{u}^s)\in\Act$ by \cref{baleq-hat}.
Also by the linearity of the map $\Psi$
in \cref{E-Psi}, we obtain
\begin{align}\label{EP6.1C}
\Hat{Z}^{n}&\;=\; \Psi(\Hat{X}^{n}-\Hat{Q}^{n},- \Hat{Y}^{n})\nonumber\\
&\;=\; \Psi\bigl(\Hat{X}^{n}-(e\cdot\Hat{X}^{n})^{+}\Hat{u}^{c},
-(e\cdot\Hat{X}^{n})^{-}\Hat{u}^{s}\bigr)
-\Hat{\Theta}^{n}\,\Psi(\Hat{u}^c,\Hat{u}^s)\,.
\end{align}

Define
\begin{equation*}
\Bar{\mathscr{A}}_{i,1}\bigl(x_i,\{z_{ij}\}\bigr)\;\df\;
\ell_{i} - \sum_{i \in \cJ(i)}\mu_{ij} z_{ij}-\gamma_{i}
\biggl(x_i- \sum_{j \in \cJ(i)} z_{ij}\biggr)\,, 
\end{equation*}
and
\begin{equation}\label{EP6.1D}
\Bar{\mathscr{A}} \mathscr{V} (x,z) \;\df \; \sum_{i\,\in\,\cI}
\Bigl(\Bar{\mathscr{A}}_{i,1}(x_i,\{z_{ij}\}) \partial_i \mathscr{V}(x)
+\lambda_i\,\partial_{ii} \mathscr{V}(x) \Bigr)\,.
\end{equation}
By the convergence of the parameters in \cref{HWpara}, we have
\begin{equation*}
\begin{split}
\babs{\Bar{\mathscr{A}}_{i,1}\bigl(x_i,\{z_{ij}\}\bigr)
-\mathscr{A}^{n}_{i,1}(x_i,\{z_{ij}\})}
&\;\le\; c_{1}(n)\bigl(1+\norm{x}\bigr)\,,\\[5pt]
\babs{\lambda_i
-\mathscr{A}^{n}_{i,2}(x_i,\{z_{ij}\})}
&\;\le\; c_{1}(n)\bigl(1+\norm{x}\bigr)\,,
\end{split}
\end{equation*}
for all $i\in\cI$,
for some constant $c_{1}(n)\searrow 0$ as $n\to\infty$.
Therefore
\begin{equation}\label{EP6.1E}
\babs{\Bar{\mathscr{A}} \mathscr{V} (x,z)-
\mathscr{A}^n\mathscr{V} (x,z)}\;\le\;
c'_{1}(n)\bigl(1+\norm{x}^\kappa\bigr)
\end{equation}
for some constant $c'_{1}(n)\searrow 0$ as $n\to\infty$.

Recall the drift representation $b(x,u)$ in \cref{E-drift2}.
By \cref{EP6.1C}, we obtain that for each $i$, 
\begin{equation}\label{EP6.1F}
\Bar{\mathscr{A}}_{i,1}\bigl(\Hat{X}^n_i,\{\Hat{Z}^n_{ij}\}\bigr)
\;=\;b_i\bigl(\Hat{X}^n,(\Hat{u}^c, \Hat{u}^s)\bigr)
+ \Hat{\Theta}^{n}\sum_{j \in \cJ(i)} \mu_{ij} {\Psi}_{ij}(\Hat{u}^c,\Hat{u}^s)
-\gamma_{i}\,\Hat{\Theta}^{n}\Hat{u}^c_{i} \,. 
\end{equation}
Since $- B_{1}$ in \cref{E-drift2} is lower diagonal and Hurwitz,
there exist positive constants $\beta_{i}$, $i\in\cI$, such
that 
\begin{equation*}
\nabla\mathscr{V}(x)\cdot B_{1}x \;\ge\; c_{2} \mathscr{V}(x)\,,
\end{equation*}
for some positive constant $c_2$. 
Thus, applying Young's inequality, after some simple calculations,
we obtain
\begin{equation}\label{EP6.1G}
\Bar{\mathscr{A}}\mathscr{V} 
\bigl(\Hat{X}^{n}, \Hat{Z}^{n}\bigr)
\;\le\; -c_{3} \mathscr{V}(\Hat{X}^n)
+ c_{4} \bigl(1+\abs{e\cdot\Hat{X}^{n}}^{\kappa} + (\Hat{\Theta}^{n})^{\kappa})
\end{equation}
for some positive constants $c_{3}$ and $c_{4}$.

Concerning the last term in \cref{EP6.1B}, we first note that by the definition of
$\mathscr{V}_{i}$,
since the jump size is of order $\frac{1}{\sqrt{n}}$, there exists a
positive constant $c_5$ such that 
$\sup_{|x'_i-x_i| \le 1} \abs{\mathscr{V}_{i}''(x'_i)} \le c_5 (1+|x_i|^{\kappa-2})$
for each $x_i \in \RR$. 
Then by the Taylor remainder theorem, we obtain that for each $i \in\cI$,
\begin{equation*}
\Delta\mathscr{V}_{i}\bigl(\Hat{X}^{n}_i(s)\bigr)
-\mathscr{V}_{i}'\bigl(\Hat{X}^{n}_i(s-)\bigr)\cdot\Delta\Hat{X}^{n}_{i}(s)
\;\le\; \frac{1}{2} \sup_{|x'_i - \Hat{X}^{n}_i(s-)| \le 1} \;
|\mathscr{V}_{i}''(x'_i)| (\Delta\Hat{X}^{n}_{i}(s))^2\,.
\end{equation*}
Thus, we have
\begin{align}\label{EP6.1H}
\Exp \Biggl[\sum_{s\le t} \mathscr{D} \mathscr{V}_{i}(\Hat{X}^n,s) \Biggr]
&\le\; \Exp \Biggl[ \sum_{s\le t} c_5 \biggl(1+
\abs{\Hat{X}^{n}_{i}(s-)}^{\kappa-2} \biggr) (\Delta\Hat{X}^{n}_{i}(s))^2  \Biggr]
\nonumber\\[5pt]
&=\; c_5 \Exp \biggl[ \int_0^t \biggl(1+ \abs{\Hat{X}^{n}_{i}(s-)}^{\kappa-2} \biggr)
\biggl(\frac{\lambda^{n}_i}{n} + \frac{1}{n}\sum_{j \in \cJ(i)}
\mu^{n}_{ij} Z^{n}_{ij}(s)
+ \frac{1}{n} \gamma^{n}_i Q^{n}_i(s) \biggr)\,\D{s} \biggr]
\nonumber\\[5pt]
&\le\; c_5 \Exp \biggl[ \int_0^t \biggl(1+ \abs{\Hat{X}^{n}_{i}(s-)}^{\kappa-2} \biggr)
\nonumber\\
&\mspace{130mu}
\biggl(\frac{\lambda^{n}_i}{n} + \frac{1}{n}\sum_{j \in \cJ(i)}
\mu^{n}_{ij} N^{n}_{j}
+ \frac{1}{n} \gamma^{n}_i \Bigl(\Hat{\Theta}^{n}(s)+
\bigl(e\cdot\Hat{X}^{n}(s)\bigr)^{+}\Bigr)
\biggr)\,\D{s} \biggr]
\nonumber\\[5pt]
&\le\; \Exp \biggl[ \int_0^t \biggl( \frac{c_3}{4}
\babs{\Hat{X}^{n}_{i}(s)}^\kappa
+ c_6 \Bigl(1+\abs{e\cdot\Hat{X}^{n}(s)}^\kappa +
\bigl(\Hat{\Theta}^{n}(s)\bigr)^{\kappa}\Bigr)
\biggr)\,\D{s} \bigg]\,, 
\end{align}
for some positive constant $c_6$, independent of $n$,
where in the first equality we use the fact that the optional martingale 
$[\Hat{X}^{n}_{i}]$
is the sum of the squares of the jumps, and that
$[\Hat{X}^{n}_{i}]-\langle\Hat{X}^{n}_{i}\rangle$ is a martingale,
and in the last inequality we use Young's inequality.

Therefore, the assertion of the proposition follows by combining
\cref{EP6.1B,EP6.1E,EP6.1G,EP6.1H}, and the
inequality
\begin{equation*}
1+\abs{e\cdot\Hat{X}^{n}}^{\kappa} + (\Hat{\Theta}^{n})^{\kappa}
\;\le\;
\bigl(1+\abs{\Hat{Q}^{n}}+\abs{\Hat{Y}^{n}}\bigr)^{\kappa}\,.\qedhere
\end{equation*}
\end{proof}

\subsection{Moment bounds for BQBS stable networks}

For the class of BQBS stable networks, the moment bound in \cref{E4.1}
holds for the limiting diffusion $X$. The following proposition shows that
the analogous moment bound also holds for the diffusion-scaled process
$\Hat{X}^n$ of this class of networks.

\begin{proposition}\label{P6.2}
Suppose that \cref{E4.3} holds.
Then \cref{P6.1} holds with \cref{EP6.1A} replaced by
\begin{equation}\label{EP6.2A}
\Exp^{Z^{n}}\biggl[\int_{0}^{T}\abs{\Hat{X}^{n}(s)}^\kappa\,\D{s}\biggr]
\;\le\; \widetilde{C}_0\,\bigl(T+\abs{\Hat{X}^{n}(0)}^\kappa\bigr)+\widetilde{C}_1
\Exp^{Z^{n}}\biggl[\int_{0}^{T}
\bigl(1+\abs{\Hat{Q}^{n}(s)}\bigr)^{\kappa}\,\D{s}\biggr]
\end{equation}
for all $n\in\NN$.
\end{proposition}

\begin{proof}
Recall the definition of the cone $\cK_{\delta,+}$ in \cref{E-cone+}.
By \cref{E4.3,EP6.1D,EP6.1F}, we obtain
\begin{align}\label{EP6.2B}
\Bar{\mathscr{A}} \mathscr{V} \bigl(\Hat{X}^{n},Z^n \bigr)&\;\le\;
c_0 \Bigl(1+  \mathscr{V}(\Hat{X}^{n}) \Ind_{\cK_{\delta,+}}(\Hat{X}^{n}) +
\abs{\grad\mathscr{V}(\Hat{X}^{n})}\,\Hat{\Theta}^{n}\Bigr)
- c_1 \mathscr{V}(\Hat{X}^{n})\,\Ind_{\cK^c_{\delta,+}}(\Hat{X}^{n})\nonumber\\[5pt]
&\;\le\; (c_0\vee c_1) \Bigl(1+  \mathscr{V}(\Hat{X}^{n})
\Ind_{\cK_{\delta,+}}(\Hat{X}^{n}) +
\abs{\grad\mathscr{V}(\Hat{X}^{n})}\,\Hat{\Theta}^{n}\Bigr)
- c_1 \mathscr{V}(\Hat{X}^{n})
\end{align}
for some positive constants $c_0$ and $c_1$.
Since
\begin{align*}
\norm{\Hat{Q}^{n}}
 &\;=\; \Hat{\Theta}^{n}+(e\cdot\Hat{X}^{n})^{+}\nonumber\\
&\;\ge\; \Hat{\Theta}^{n}+\delta \abs{\Hat{X}^{n}}
 \,\Ind_{\cK_{\delta,+}}(\Hat{X}^{n})\,,
\end{align*}
we obtain by \cref{EP6.2B} that
\begin{align}\label{EP6.2C}
\Bar{\mathscr{A}} \mathscr{V} \bigl(\Hat{X}^{n},Z^n \bigr)&\;\le\;
(c_0\vee c_1) \Bigl(1+ \norm{\Hat{Q}^{n}} +
\abs{\grad\mathscr{V}(\Hat{X}^{n})}\,\norm{\Hat{Q}^{n}}\Bigr)
- c_1 \mathscr{V}(\Hat{X}^{n})\nonumber\\[5pt]
&\;\le\;c_0'\, \bigl(1+  \mathscr{V}(\Hat{Q}^{n})\bigr)- c'_1 \mathscr{V}(\Hat{X}^{n})
\end{align}
for some positive constants $c_i'$, $i=0,1$, where the second inequality
in \cref{EP6.2C} follows by applying Young's inequality.
The rest follows as in the proof of \cref{P6.1}.
\end{proof}

As a consequence of \cref{P6.2}, we also obtain the following moment
bound for the idleness process.  
\begin{corollary}\label{C6.1}
If \cref{E4.3} holds,
then there exist some constants $\widetilde{C}_0'>0$
and $\widetilde{C}_1'>0$ such that
\begin{equation*}
\Exp^{Z^{n}}\biggl[\int_{0}^{T}\abs{\Hat{Y}^{n}(s)}^\kappa\,\D{s}\biggr]
\;\le\; \widetilde{C}_0'\,\bigl(T+\abs{\Hat{Y}^{n}(0)}^\kappa\bigr)+\widetilde{C}_1'
\Exp^{Z^{n}}\biggl[\int_{0}^{T}
\bigl(1+\abs{\Hat{Q}^{n}(s)}\bigr)^{\kappa}\,\D{s}\biggr]\,,
\end{equation*}
for all $n \in \NN$, and 
for any sequence $\{Z^{n}\in\fZ^{n},\;n\in\NN\}$. 
\end{corollary}

\begin{proof}
The claim follows from \cref{EP6.2A} and the fundamental identities
$\norm{\Hat{Q}^{n}}= \Hat{\Theta}^{n}+(e\cdot\Hat{X}^{n})^{+}$ 
and $\norm{\Hat{Y}^{n}}= \Hat{\Theta}^{n}+(e\cdot\Hat{X}^{n})^{-}$. 
\end{proof}

\subsection{Convergence of mean empirical measures}

For the process $\Hat{X}^n$ under a scheduling policy
$Z^n$, and with $U^n$ as in \cref{D-upara},
we define the mean empirical measures 
\begin{equation} \label{E-emp}
\Phi^{Z^n}_{T}(A\times B)\;\df\;
\frac{1}{T}\,\Exp^{Z^n}\biggl[\int_{0}^{T}\Ind_{A\times B}
\bigl(\Hat{X}^{n}(t),U^{n}(t)\bigr)\,\D{t}\biggr]
\end{equation}
for Borel sets $A\subset\RR^{I}$ and $B\subset \Act$. Recall \cref{DEJWC}.
The lemma which follows
shows that if the long-run average first-order moment of the diffusion-scaled
state process under an EJWC scheduling policy is finite, then the mean empirical
measures $\Phi^{Z^n}_{T}$ are tight and converge
to an ergodic occupation measure corresponding to some stationary stable Markov
control for the limiting diffusion control problem.
This property is used in the proof of the lower bounds in \cref{T3.1,T3.2}.

\begin{lemma}\label{L6.1}
If under some sequence of scheduling policies 
$\{Z^{n},\;n\in\NN\}\subset\boldsymbol\fZ$, we have
\begin{equation}\label{L6.1A}
\sup_{n}\;\limsup_{T\to\infty}\;\frac{1}{T}\,
\Exp^{Z^{n}}\biggl[\int_{0}^{T}\babs{\Hat{X}^{n}(s)}\,\D{s}\biggr]\;<\,\infty\,,
\end{equation}
then $\bigl\{\Phi^{Z^n}_{T}\colon n\in \NN, \ T>0\bigr\}$ is tight,
and
any limit point $\uppi\in\cP(\RR^{I}\times\Act)$ of $\bigl\{\Phi^{Z^n}_{T}\bigr\}$
over a sequence $(n_k,T_k)$, with $n_k\to\infty$ and $T_k\to\infty$,
lies in $\eom$. 
\end{lemma}

\begin{proof}
It is clear that \cref{L6.1A} implies that
$\bigl\{\Phi^{Z^n}_{T}\colon n\in \NN, \ T>0\bigr\}$ is tight. 
For $f \in \Cc^{\infty}_c(\RR^I)$, by \cref{E-KW}, the definition of
$\Phi^{Z^n}_{T}$ in \cref{E-emp}, and \cref{L2.3}, we obtain
\begin{equation}\label{L6.1B}
\frac{\Exp\bigl[f(\Hat{X}^{n}(T))\bigr] -
\Exp\bigl[f(\Hat{X}^{n}(0))\bigr]}{T} \;=\;
\int_{\RR^I \times \Act} \Breve{\mathscr{A}}^{n} f(x,u)\, \Phi^{Z^n}_{T} (\D{x}, \D{u})
+ \frac{1}{T}\;\Exp\Biggl[ \sum_{s\le T}\mathscr{D} f(\Hat{X}^n,s)\Biggr]\,.
\end{equation}

Let
\begin{equation*}
\norm{f}_{\Cc^3} \df \sup_{x\in\RR^I}
\Bigl( \abs{f(x)} + \sum_{i\,\in\,\cI} \abs{\partial_i f(x)}
+ \sum_{i\,,i'\,\in\,\cI} \abs{\partial_{ii'} f(x)}
+ \sum_{i\,,i'\,,i''\,\in\,\cI} \abs{\partial_{ii'i''} f(x)}\Bigr)\,.
\end{equation*}
By Taylor's formula, using also the fact that the jump size is 
$\frac{1}{\sqrt{n}}$, we obtain
\begin{equation*}
\babs{\mathscr{D} f(\Hat{X}^n,s)}
\;\le\;\frac{\kappa \norm{f}_{\Cc^{3}}}{\sqrt{n}}\sum_{i,i'=1}^{I}
\babs{\Delta\Hat{X}^{n}_{i}(s)\Delta \Hat{X}^{n}_{i'}(s)}\,,
\end{equation*}
for some constant $\kappa$ that does not depend on $n\in\NN$.
On the other hand, since independent Poisson processes have no simultaneous
jumps w.p.1., we have
\begin{equation*}
\frac{1}{T}\;\Exp\biggl[\int_{0}^{T}\sum_{i,i'=1}^{I}
\babs{\Delta\Hat{X}^{n}_{i}(s)\Delta \Hat{X}^{n}_{i'}(s)}\,\D{s}\biggr] 
\;=\;
\frac{1}{T}\;\Exp\babss{\int_{0}^{T}\sum_{i\in\cI}
\biggl(\frac{\lambda^{n}_i}{n} + \frac{1}{n}\sum_{j \in \cJ(i)}
\mu^{n}_{ij} Z^{n}_{ij}(s)
+ \frac{1}{n} \gamma^{n}_i Q^{n}_i(s) \biggr)\,\D{s}}\,,
\end{equation*}
and the right hand side is uniformly bounded over $n\in\NN$
and $T>0$
by \cref{L6.1A}.

Therefore, taking limits in \cref{L6.1B}, we obtain
\begin{equation}\label{L6.1C}
\limsup_{n\,\to\,\infty,\;T\,\to\,\infty}\; 
\int_{\RR^I \times \Act} \Breve{\mathscr{A}}^{n} f(x,u)\, \Phi^{Z^n}_{T} (\D{x}, \D{u})
\;=\;0\,.
\end{equation}

Let $(n_{k},T_{k})$, with $n_k\to\infty$ and $T_k\to\infty$, 
be any sequence along which
$\Phi^{Z^n}_{T} (\D{x}, \D{u})$ converges to some $\uppi\in\cP(\RR^{I}\times\Act)$.
By \cref{L6.1C},  \cref{L2.5}, and a standard triangle inequality, we obtain
\begin{equation*}
\int_{\RR^I \times \Act} \Lg^{u} f(x)\,\uppi(\D{x}, \D{u})
\;=\;0 \,.
\end{equation*}
This implies that $\uppi \in \eom$.
\end{proof}

We introduce a canonical construction of scheduling policies which is used in the
proofs of the upper bounds for asymptotic optimality. 
Recall \cref{DEJWC,D-hatx}, and $\Breve{\sX}^{n}$
defined in \cref{E-BsX}.

\begin{definition}\label{D6.1}
Let
$\varpi\colon \{ x \in \RR^{I}_{+}\colon e\cdot x \in \mathbb{Z}\}
\to \mathbb{Z}^{I}_{+}$ be
a measurable map defined by
\begin{equation*}
\varpi(x) \;\df\; \Bigl(\lfloor x_1 \rfloor, \dotsc,\lfloor x_{I-1} \rfloor,\,
e\cdot x - \textstyle\sum_{i=1}^{I-1} \lfloor x_i \rfloor\Bigr)\,,\qquad
x \in \RR^I\,.
\end{equation*}
By abuse of notation, we denote by $\varpi$ the similarly defined map
$\varpi\colon \{ x \in \RR^{J}_{+}\colon e\cdot x \in \mathbb{Z}\}
\to \mathbb{Z}^{J}_{+}$.
For a precise  control $v \in \Ussm$, define the maps
$q^{n}[v]\colon\RR^{I}\to \mathbb{Z}^{I}_{+}$
and $y^{n}[v]\colon\RR^{I}\to \mathbb{Z}^{J}_{+}$ by
\begin{equation*}
q^{n}[v](\Hat{x}) \;\df\;
\varpi \bigl((e\cdot (\sqrt{n}\Hat{x}+ n x^*) )^{+} v^c(\Hat{x})\bigr) \,,
\qquad
y^{n}[v](\Hat{x}) \;\df\;
\varpi\bigl((e \cdot   (\sqrt{n}\Hat{x}+ n x^*) )^{-} v^s(\Hat{x}) \bigr)\,,
\end{equation*}
where $\Hat{x} \in \sS^n$. 
Recall the definition of the linear map $\Psi$ in \cref{E-Psi}.
Define the Markov scheduling policy $z^n[v]$ on $\Breve{\sS}^{n}$ by 
\begin{equation*}
z^{n}[v](\Hat{x}) \;\df\;
\Psi\bigl(x-q^{n}[v](\Hat{x}), N^n - y^{n}[v](\Hat{x})\bigr)\,. 
\end{equation*}
\end{definition}

\begin{corollary}\label{L-Xn}
For any precise control $v \in \Ussm$, 
we have   
\begin{equation*}
e\cdot q^n[v](\Hat{x}^n(x)) \wedge e \cdot y^n[v](\Hat{x}^n(x))=0\,,
\quad \text{and} \quad z^n[v](\Hat{x}^n(x))\in \cZn(x)\,,
\end{equation*}
for all $x\in\Breve{\sX}^{n}$, i.e., the JWC condition is
satisfied in $\Breve{\sX}^{n}$.
\end{corollary}

\begin{proof}
This follows from \cref{L-JWC} and the definition of the
maps $q^n[v]$,  $y^n[v]$ and  $z^n[v]$. 
\end{proof}

The next lemma is used  in the proof of upper bounds in \cref{T3.1,T3.2}.
It shows that for any given continuous precise  stationary stable Markov control,
if we construct a sequence of EJWC scheduling policies as in \cref{D6.1},
then the corresponding mean empirical measures of the diffusion-scaled processes
converge and the limit agrees with the ergodic occupation measure of the
limiting diffusion corresponding to that control. 

\begin{lemma}\label{L6.2}
Let $v\in\Ussm$ be a  continuous precise control,
and $\bigl\{Z^{n}\,\colon\, n\in\NN\bigr\}$ be any sequence of
admissible scheduling policies such that each $Z^n$ agrees with
the Markov scheduling policy $z^n[v]$ given in \cref{D6.1}
on $\sqrt{n}\Breve{B}$, i.e.,
$Z^n(t) = z^n[v]\bigl(\Hat{X}^n(t)\bigr)$ whenever $\Hat{X}^n(t)\in\sqrt{n}\Breve{B}$.
For $\Hat{x}\in \sqrt{n}\Breve{B}\cap\sS^n$, we define
\begin{align*}
u^{c,n}[v] (\Hat{x}) &\;\df\; \begin{cases}
\frac{q^{n}[v](\Hat{x})}{e\cdot q^{n}[v](\Hat{x})}
& \text{if~} e\cdot q^{n}[v](\Hat{x})>0\,,\\[5pt]
v^c(\Hat{x}) &\text{otherwise,}\end{cases}
\intertext{and} 
u^{s,n}[v](\Hat{x}) &\;\df\; \begin{cases}
\frac{y^{n}[v](\Hat{x})}{e\cdot y^{n}[v](\Hat{x})}
& \text{if~} e\cdot y^{n}[v](\Hat{x})>0\,,\\[5pt]
v^s(\Hat{x}) &\text{otherwise.}\end{cases}
\end{align*}
For the process $X^n$ under the scheduling policy $Z^n$,
define the mean empirical measures 
\begin{equation} \label{EL6.2A}
\Tilde\Phi^{Z^n}_{T}(A\times B)\;\df\;
\frac{1}{T}\,\Exp^{Z^n}\biggl[\int_{0}^{T}\Ind_{A\times B}
\bigl(\Hat{X}^{n}(t),u^{n}[v]\bigl(\Hat{X}^{n}(t)\bigr)\bigr)\,\D{t}\biggr]
\end{equation}
for Borel sets $A\subset\sqrt{n}\Breve{B}$ and $B\subset \Act$.
Suppose that \cref{L6.1A} holds under this sequence  $\{Z^n\}$.
Then the ergodic occupation measure $\uppi_{v}$ of the controlled
diffusion in \cref{E-diff} corresponding to $v$
is the unique limit point in $\cP(\Rd\times\Act)$
of $\Tilde\Phi^{Z^n}_{T}$ as $n$ and $T$ tend to $\infty$.
\end{lemma}

\begin{proof}
The proof follows exactly as that of Lemma~7.2 in \cite{AP16}. 
\end{proof}

\subsection{A stability preserving property in the JWC region}

If there exists a stationary Markov control under which the controlled
diffusion is exponentially ergodic, then it can be shown that under the
corresponding scheduling policy as constructed in \cref{D6.1},
the diffusion-scaled state process satisfies a Foster--Lyapunov condition
of the exponential ergodicity type in the JWC region.
We refer to this as the \emph{stability preserving property in the JWC region}.
This property is important to prove the upper bounds for the asymptotic optimality,
and is also the reason why the spatial truncation technique works. 
We present this in the following Proposition. 

\begin{proposition}\label{P6.3}
Let $\widetilde{\Lyap}_{\epsilon,\beta}$ be as in \cref{Lyapk}.
Suppose $v\in\Ussm$ is such that for some
positive constants $c_0$, $c_1$, and $\epsilon>0$, 
and a positive vector $\beta\in\RR^I$, it holds that
\begin{equation} \label{EP6.3A}
\Lg^{v} \widetilde{\Lyap}_{\epsilon,\beta}(x) \;\le\;
c_0  - c_1\, \widetilde{\Lyap}_{\epsilon,\beta}(x)\qquad
\forall x\in\RR^I\,.
\end{equation} 
Let $\Hat{X}^{n}$ denote the diffusion-scaled state process under the
scheduling policy $z^{n}[v]$ in \cref{D6.1}, and
$\widehat\cL_n$ denote its generator.
Then, there exists $n_{0}\in\NN$ such that
\begin{equation*}
\widehat\cL_n \widetilde{\Lyap}_{\epsilon,\beta}(\Hat{x}) \;\le\;
\Hat{c}_0  - \Hat{c}_1 \,\widetilde{\Lyap}_{\epsilon,\beta}(\Hat{x})\qquad
\forall \Hat{x}\in\Breve\sS^n\cap\sqrt{n}\Breve{B}\,,
\end{equation*} 
for some positive constants $\Hat{c}_0$ and $\Hat{c}_1$, and for all $n\ge n_0$.
\end{proposition}

\begin{proof}
Recall the notation $\Hat{x}= \Hat{x}^n(x)$ in \cref{D-hatx}. Under the Markov
scheduling policy $z^n[v]$ in \cref{D6.1},  for each given $x\in \RR^I$, 
we define the  associated diffusion-scaled quantities 
\begin{equation*}
\Hat{q}^n \;=\; \Hat{q}^n[v]\;\df\; \frac{q^n[v]}{\sqrt{n}}\,, \quad
\Hat{y}^n \;=\; \Hat{y}^n[v]\;\df\; \frac{y^n[v]}{\sqrt{n}}\,, \quad
\Hat{z}^n \;=\; \Hat{z}^n[v]\;\df\; \frac{z^n[v] -  nz^*}{\sqrt{n}}\,.
\end{equation*}

Recall that $\widehat{\cL}_n=\widehat{\cL}_n^{z^n[v]}$ denotes the generator
of $\Hat{X}^n$ under the
scheduling policy $z^n[v]$ (see \cref{E-cL,E-gen2}).
Let $\widehat\Lyap_{\epsilon,\beta}(x)\;\df\;\widetilde\Lyap_{\epsilon,\beta}(\Hat{x})$,
with $\widetilde\Lyap_{\epsilon,\beta}$ as in \cref{Lyapk}.
Using the identity in \cref{PP5.1Ba}, we obtain 
\begin{equation}\label{PP6.3A}
\Babs{\widehat\Lyap_{\epsilon,\beta}(x\pm e_i)
- \widehat\Lyap_{\epsilon,\beta}(x) \mp
\epsilon\tfrac{1}{\sqrt{n}}\,\beta_i \Hat{x}_i \phi_\beta(\Hat{x})\,
\widehat\Lyap_{\epsilon,\beta}(x)}
\;\le\; \tfrac{1}{n}\epsilon^2\,\Hat\kappa_1\,\widehat\Lyap_{\epsilon,\beta}(x)
\end{equation}
for some constant $\Hat\kappa_1>0$, and all $\epsilon\in(0,1)$.
Thus by \cref{E-cL,E-gen2,PP6.3A}, we obtain 
\begin{equation}\label{PP6.3B} 
\cL_n^z\, \widehat\Lyap_{\epsilon,\beta}(x) \;\le\;
\epsilon\,\widehat\Lyap_{\epsilon,\beta}(x)\,
\sum_{i\in\cI} \Bigl(\beta_i\,\Hat{x}_i\,\phi_\beta(\Hat{x})\,G_{n,i}^{(1)}(\Hat{x})
+\epsilon\,\Hat\kappa_1\,G_{n,i}^{(2)}(\Hat{x})\Bigr)\,,
\end{equation}
in direct analogy to \cref{PP5.1Be}, where 
\begin{align*} 
G_{n,i}^{(1)}(\Hat{x}) &\;\df\;
\ell_i^{n} - \sum_{j \in \cJ(i)} \mu_{ij}^{n} \Hat{z}^n_{ij}
- \gamma^{n}_i \Hat{q}^n_i \,,\\[5pt]
G_{n,i}^{(2)}(\Hat{x}) &\;\df\;  \frac{\lambda^{n}_{i}}{n}
+\sum_{j\in\cJ(i)} \mu_{ij}^{n} z_{ij}^*
+\frac{1}{\sqrt{n}} \sum_{j\in\cJ(i)} \mu_{ij}^{n} \Hat{z}_{ij}^n
+ \frac{\gamma^{n}_i}{\sqrt{n}}  \Hat{q}^n_i\,.
\end{align*}
The dependence of $G_{n,i}^{(1)}$ and $G_{n,i}^{(2)}$ on $\Hat{x}$
is implicit through $z^n[v]$.
By \cref{D6.1}, it always holds that $z^n_{ij} \le x_i$
and $q^n_i\le x_i$ for all $(i,j)\in\cE$.
Since $\frac{x_i}{\sqrt{n}} =\Hat{x}_{i}+\sqrt{n}x^*_i$,
and  $\Hat{z}_{ij} = \frac{x_i}{\sqrt{n}}+ z_{ij}^*$, we obtain 
\begin{align}\label{PP6.3C}
n\,G_{n,i}^{(2)}(x) 
& \;\le\;  \lambda_i^{n}
+  \sqrt{n} \sum_{j \in \cJ(i)} \mu_{ij}^{n} z^*_{ij}
+  \sqrt{n}\sum_{j \in \cJ(i)} \mu_{ij}^{n} \bigl(\Hat{x}_{i}
+ \sqrt{n}(x^*_i -  z^*_{ij}) \bigr)  + \sqrt{n} \gamma^{n}_i
(\Hat{x}_i + \sqrt{n} x^*_i) \, \nonumber \\
& \;=\;  \lambda_i^{n}
+  n \biggl( \sum_{j \in \cJ(i)} \mu_{ij}^{n} + \gamma^{n}_i \biggr) x_i^* 
+  \sqrt{n} \biggl( \sum_{j \in \cJ(i)} \mu_{ij}^{n} + \gamma^{n}_i \biggr)
\Hat{x}_{i}  \, \nonumber\\
&\;\le\;  \Hat{\kappa}_2 \big(n+  \sqrt{n} \abs{\Hat{x}_{i}} \big) \,,
\end{align}
for some constant $\Hat\kappa_2>0$,
where the last inequality follows from the assumption on the parameters
in \cref{HWpara}.

Since the control $v$ satisfies \cref{EP6.3A}, we must have,
for some positive constants $c_0'$ and $c_1'$ that
\begin{equation}\label{PP6.3D}
\sum_{i\in\cI} \epsilon\beta_i b_i\bigl(x,v(x)\bigr) x_i \phi_\beta(x)
\widetilde\Lyap_{\epsilon,\beta}(x) \;\le\;
c_0' - c_1'\widetilde\Lyap_{\epsilon,\beta}(x)\qquad\forall x\in\RR^I\,.
\end{equation}
By \cref{L2.3,L2.5}, we have
$G_{n,i}^{(1)}(\Hat{x})\to b_i\bigl(\Hat{x},v(\Hat{x})\bigr)$,
uniformly over compact sets of $\RR^I$ as $n\to\infty$.
Therefore, the result follows by combining \cref{PP6.3B,PP6.3C,PP6.3D}.
\end{proof}

\medskip
\section{Proofs of Asymptotic Optimality}\label{S-LUB}

We need the following lemma, which is used in the proof of the upper bound.

\begin{lemma}\label{L7.1}
For any $\varepsilon>0$, there exists a
continuous precise
control $v_\epsilon\in\Ussm$ with the following properties:
\begin{enumerate}
\item[(a)]
For some positive vector $\beta\in\RR^{I}$ which does not depend
on $\varepsilon$, and any $\kappa>1$, we have
\begin{equation}\label{EL7.1A}
\Lg^{v_\epsilon} \Lyap_{\kappa,\beta}(x) \;\le\;
c_0 - c_1 \,\Lyap_{\kappa,\beta}(x)\qquad \forall x\in \RR^I
\end{equation}
for some constants $c_0$ and $c_1$ depending only on $\kappa$.
\item[(b)]
With
$\uppi_{v_{\varepsilon}}$ denoting the ergodic
occupation measure corresponding to $v_{\varepsilon}$, it holds that 
\begin{equation*}
\uppi_{v_{\varepsilon}} (r)
\;=\; \int_{\RR^{I} \times\Act} r(x,u)\, \uppi_{v_{\varepsilon}}(\D{x},\D{u}) \;<\;
\varrho^* + \varepsilon\,,
\end{equation*}
where $\varrho^*$ is the optimal value of problem {\upshape(P1$^\prime$)}.
\end{enumerate}
\end{lemma}

\begin{proof}
By \cite[Theorem~4.2]{AP15} there exists a constant Markov control
$\Bar{u}$ and a positive vector $\beta\in\RR^{I}$ satisfying
\begin{equation}\label{EL7.1B}
\Lg^{\Bar{u}} \Lyap_{\kappa,\beta}(x) \;\le\;
\Bar{c}_0 - \Bar{c}_1 \,\Lyap_{\kappa,\beta}(x)\qquad \forall x\in \RR^I
\end{equation}
for all $\kappa>1$ and some constants $\Bar{c}_0$ and $\Bar{c}_1$.
Even though not stated in that theorem, it follows from its proof that
the constants $\Bar{c}_0$ and $\Bar{c}_1$ depend only on $\kappa$.
We perturb $r$ by adding a positive strictly
convex function $f\colon\Act\to\RR_+$, such that
the optimal value of the problem {\upshape(P1$^\prime$)} with $r$ replaced
by $r+f$ is smaller than $\varrho^* + \frac{\varepsilon}{3}$.
Following the proof of Theorems~4.1 and 4.2 in \cite{ABP14}, there exists
$R>0$ large enough and a stationary Markov control $\Bar{v}_R$, which
agrees with $\Bar{u}$ on $B_R^c$ and satisfies
$\uppi_{\Bar{v}_R} (r+f)<\varrho^* + \frac{2\varepsilon}{3}$.
This control satisfies, for some $V_R\in\Cc^2(B_R)$,
\begin{equation}\label{EL7.1C}
\min_{u\in\Act}\;\bigl[b(x,u)\cdot\grad V_R + r(x,u)+f(u)\bigr]\;=\;
b\bigl(x,\Bar{v}_R(x)\bigr)\cdot\grad V_R + r\bigl(x,\Bar{v}_R(x)\bigr)
+f\bigl(\Bar{v}_R(x)\bigr)
\end{equation}
for all $x\in B_R$.
Since $u\mapsto \{b(x,u)\cdot p + r(x,u)+f(u)\}$
is strictly convex whenever it is not constant, it follows by \cref{EL7.1C} that
$\Bar{v}_R$ is continuous on $B_R$.
Consider the concatenated Markov control which agrees with
$\Bar{v}_R$ in $B_R$ and with $\Bar{u}$ in $B_R^c$.
As in the proof of \cite[Theorem~2.2]{ABP14}, we can employ a cut-off function
to smoothen the discontinuity of this control at the concatenation boundary,
and thus obtain a Markov control $v_\varepsilon$ satisfying
$\uppi_{v_\varepsilon} (r+f)<\varrho^* + \varepsilon$.
Clearly then part (b) holds since $f$ is nonnegative,
while part (a) holds by \cref{EL7.1B} and the fact that
$v_\varepsilon$ agrees with $\Bar{u}$ outside a compact set.
This completes the proof. 
\end{proof}

Concerning the constrained problem (P2$^\prime$) and the fairness problem
(P3$'$) we have the following analogous result.

\begin{corollary}\label{C7.1}
For any $\varepsilon>0$, there exists a
continuous precise
control $v_\epsilon\in\Ussm$ satisfying \cref{L7.1}\,\textup{(}a\textup{)},
and constants
$\updelta^{\epsilon}_j<\updelta_j$, $j\in\cJ$
such that:
\begin{itemize}
\item[(i)] In the case of problem {\upshape(P2$^\prime$)} we have
\begin{equation*}
\uppi_{v_{\varepsilon}}(r_{\mathsf{o}})\;<\;
\varrho_{\mathsf{c}}^{*} +\varepsilon\,,\quad\text{and}\quad
\uppi_{v_{\varepsilon}}(r_{j}) \;\le\; \updelta^{\epsilon}_j\,,\quad j \in\cJ\,.
\end{equation*}
\item[(ii)] In the case of problem {\upshape(P3$^\prime$)}  we have
\begin{equation*}
\uppi_{v_{\varepsilon}}(r_{\mathsf{o}})\;<\;
\varrho_{\mathsf{f}}^{*} +\varepsilon\,,\quad\text{and}\quad
\uppi_{v_{\varepsilon}}(r_{j}) \;=\; \uptheta_j\, \sum_{\jmath\in\cJ}
\uppi_{v_{\varepsilon}}(r_{\jmath})\,,\qquad j\in\cJ\,.
\end{equation*}
\end{itemize}
\end{corollary}

\begin{proof}
Part (i) follows as in Theorem~5.7 in \cite{AP15}.
The proof of part (ii) is completely analogous.
\end{proof}

\subsection{Proof of Theorem~\ref{T3.1}}

\begin{proof}[Proof of the lower bound]
Without loss of generality, suppose that $Z^{n_k}\in\fZ^{n_k}$,
for some increasing sequence $\{n_k\}\subset\NN$, is a collection of
scheduling policies in $\boldsymbol\fZ$ such that 
$J\bigl(\Hat{X}^{n_{k}}(0),Z^{n_{k}}\bigr)$ converges to a finite value
as $k\to\infty$.
Denote by $\Phi^{n_k}_T$, the mean empirical measure $\Phi^{Z^{n_k}}_T$ defined 
in \cref{E-emp}.
Then by \cref{P6.1} and the definitions of the running cost
$\Hat{r}$ in \cref{runcost-ex} and $J\bigl(\Hat{X}^{n}(0),Z^{n}\bigr)$
in \cref{cost-ds}
we obtain
\begin{equation}\label{PT3.1A}
\sup_{k\in\NN}\;\limsup_{T\to\infty}\;
\Exp^{Z^{n_k}}\biggl[\int_{0}^{T}\abs{\Hat{X}^{n_k}(s)}^m\,\D{s}\biggr]
\;<\; \infty\,.
\end{equation}
It is also clear by the definition of $\Phi^{n}_T$ and $J$ that
we can select a sequence $\{T_k\}\subset\RR_+$, with $T_{k}\to\infty$,
such that 
\begin{equation}\label{PT3.1B}
\int_{\RR^{I}\times\Act} r(x,u)\,\Phi^{n_{k}}_{T_{k}} (\D{x}, \D{u})\;\le\;
J\bigl(\Hat{X}^{n_{k}}(0),Z^{n_{k}}\bigr)+ \frac{1}{k} \qquad\forall\,k\in\NN\,.
\end{equation}
By \cref{L6.1} and \cref{PT3.1A},
$\{\Phi^{n_{k}}_{T_k}\,\colon\, k \in\NN \}$ is tight and the limit of
any converging subsequence $\{\Phi^{n_{k}'}_{T_k'}\}$ is in $\eom$.
Therefore it follows by \cref{PT3.1B} that
\begin{equation*}
\lim_{k\to0}\;J\bigl(\Hat{X}^{n_{k}}(0),Z^{n_{k}}\bigr)
\;\ge\; \inf_{\uppi\in\eom}\; \uppi (r)\;=\;\varrho^{*}\,.
\end{equation*} 
and this completes the proof. 
\end{proof}

\begin{proof}[Proof of the upper bound]

Recall $\Lyap_{\kappa, \beta}$ in \cref{Lyapk}.
Let $\kappa=m+2$.
By \cref{L7.1}, there exists a continuous precise control
$v_\varepsilon$ such that the corresponding ergodic occupation measure satisfies
$\uppi_{v_{\varepsilon}} (r) \;<\;
\varrho^* + \varepsilon$,  and \cref{EL7.1A} holds. 

For the $n^\text{th}$ system, we construct a concatenated Markov
scheduling policy
$\mathring{z}^n$ as follows. 
Recall \cref{DEJWC}.
Inside $\Breve{\sX}^{n}$, we apply the stationary policy $z^n[v_\epsilon]$
as in \cref{D6.1}, and outside $\Breve{\sX}^{n}$, we apply some Markov 
scheduling policy $z\in\sZ^n$ in
\cref{D-SDPf} that is exponentially stable. 
By \cref{P5.1,P6.3}
there exist positive constants $\Hat{c}_0$, $\Hat{c}_1$,
a positive vector $\beta\in\RR^I$, and $n_0\in\NN$,
such that
\begin{equation} \label{PT3.1C}
\widehat\cL_n^{\mathring{z}^n} \, \widetilde{\Lyap}_{\epsilon,\beta}(\Hat{x}) \;\le\;
\Hat{c}_0  - \Hat{c}_1\, \widetilde{\Lyap}_{\epsilon,\beta}(\Hat{x})\qquad
\forall \Hat{x}\in\sS^n\,,\quad\forall\, n\ge n_0\,.
\end{equation} 
This immediately implies that 
$\sup_{n\ge n_0}\,J(\Hat{X}^n(0),Z^n)<\infty$. 
Let  $\Tilde\Phi^n_T\equiv \Tilde\Phi^{\mathring{z}^n}_{T}$ be the corresponding mean
empirical measures as defined in \cref{EL6.2A}.
Then the Foster--Lyapunov condition
 in \cref{PT3.1C} implies that we can choose a sequence
$\{T_n\}$ such that 
\begin{equation}\label{PT3.1D}
\sup_{n\ge n_0}\;\sup_{T\ge T_n}\; \int_{\RR^I\times\Act}
\widetilde{\Lyap}_{\epsilon,\beta}(\Hat{x})\,
\Tilde\Phi^n_T(\D{\Hat{x}},\D{u})\;<\;\infty\,.
\end{equation}
WLOG, we assume that $T_n\to\infty$. 

It is clear that $\mathring{z}^n$ can be viewed as a function of
$\Hat{x}\in\sS^n$.
We let $\Hat{\mathring z}^n_{ij}(\Hat{x})
\;\df\; \frac{(\mathring{z}^n_{ij}(\Hat{x})- nz^*)}{\sqrt{n}}$
as in \cref{D-hatx}.
In analogy to
\cref{hatq-hatxz,haty-hatz} we define
\begin{align*}
\Hat{\mathring q}^n_i(\Hat{x}) &\;\df\; \Hat{x}_i - \sum_{j\in \cJ(i)}
\Hat{\mathring z}^n_{ij}(\Hat{x})\,, \quad \forall\, i \in\cI\,, \\
 \Hat{\mathring y}^n_j(\Hat{x}) &\;\df\;
 \frac{N^n_j- n \sum_{i\in \cI(j)} z_{ij}^*}{\sqrt{n}}
 - \sum_{i\in \cJ(j)} \Hat{\mathring z}^n_{ij}(\Hat{x})\,,
 \quad \forall j \in \cJ\,. 
\end{align*}

The running cost $\Hat{r}$ is uniformly integrable with
respect to the collection $\{\Tilde\Phi^n_T\,,\ n\in\NN\,,\,T\ge0\}$
by \cref{PT3.1D}.
Thus by  Birkhoff's ergodic theorem, for any $\eta>0$, we can choose a
ball $B(\eta)$, and a sequence $T_n$ such that
\begin{equation}\label{PT3.1E}
\babss{\int_{B(\eta)\times\Act}
\Hat{r}\bigl( (e\cdot\Hat{\mathring q}^n(\Hat{x})\bigr)^+\,u^c,\,
(e\cdot\Hat{\mathring y}^n(\Hat{x})\bigr)^+\,u^s\bigr)\,
\Tilde\Phi^n_T(\D{\Hat{x}},\D{u})-J(\Hat{X}^n(0),\mathring{z}^n)}\;\le\;
\frac{1}{n}+\eta \,,
\end{equation}
for all $ T\ge T_{n}$.

By  the JWC condition on  $\{\Hat{x}\in\Breve\sS^n\}$ and \cref{L-Xn}, we have
$(e\cdot\Hat{\mathring q}^n(\Hat{x})\bigr)^+ = (e\cdot\Hat{x})^+$ and
$(e\cdot\Hat{\mathring y}^n(\Hat{x})\bigr)^+ = (e\cdot\Hat{x})^-$
for all $\Hat{x}\in B(\eta)$, and for all large
enough $n$.
On the other hand we have
\begin{equation}\label{PT3.1F}
\sup_{(\Hat{x},u)\in B(\eta)\times\Act}\;
\babs{\Hat{r}\bigl( (e\cdot\Hat{\mathring q}^n(\Hat{x})\bigr)^+\,u^c,\,
(e\cdot\Hat{y}^n(\Hat{\mathring x})\bigr)^+\,u^s\bigr)-r(\Hat{x},u)}
\;\xrightarrow[n\to\infty]{}\;0\,.
\end{equation}
Since $v_\varepsilon$ is a continuous precise
control then $\Tilde\Phi^n_T$ converges to 
$\uppi_{v_\epsilon}$ in $\cP(\RR^I\times\Act)$ as $n$ and $T$ tend to $\infty$
by \cref{L6.2}.
Thus, using \cref{PT3.1F} and a triangle inequality, we obtain
\begin{equation}\label{PT3.1G}
\int_{B(\eta)\times\Act}
\Hat{r}\bigl( (e\cdot\Hat{\mathring q}^n(\Hat{x})\bigr)^+\,u^c,\,
(e\cdot\Hat{\mathring y}^n(\Hat{x})\bigr)^+\,u^s\bigr)\,
\Tilde\Phi^n_{T_n}(\D{\Hat{x}},\D{u})
\;\xrightarrow[n\to\infty]{}\;
\int_{B(\eta)\times\Act} r(x,u)\,\uppi_{v_\epsilon}(\D{x},\D{u})\,.
\end{equation}
By \cref{PT3.1E,PT3.1G} we obtain
\begin{equation*}
\limsup_{n\to\infty}\;J(\Hat{X}^n(0),\mathring{z}^n)
\;\le\; \varrho^{*} + \epsilon +\eta\,.
\end{equation*}
Since $\eta$ and $\epsilon$ are arbitrary, this completes the proof
of the upper bound.
\end{proof}

\begin{remark}
It is clear that if the network satisfies \cref{E4.3} and $\zeta=0$ in
\cref{E-cost}, then the same conclusion for the lower bound
can be drawn by invoking \cref{P6.2} in the preceding proof.
\end{remark}

\subsection{Proof of Theorem~\ref{T3.2}}

\begin{proof}[Proof of the lower bound]
The proof follows by a similar argument as in the proof of the lower bound
for \cref{T3.1}. 
Let $\{Z^{n_k}\in\fZ^{n_k}\} \subset \boldsymbol\fZ$, with
$\{n_{k}\}\subset\NN$ an increasing sequence,
such that 
$J_{\mathsf{o}}\bigl(\Hat{X}^{n_{k}}(0),Z^{n_{k}}\bigr)$ converges to
a finite value.
Select an increasing sequence $\{T_k\}\subset\RR_+$ such that
\eqref{PT3.1B} holds with $J$ replaced by $J_\circ$ and $r$ by $r_\circ$.
Following the proof of \cref{T3.1},
let $\Hat\uppi\in\cP(\RR^I \times \Act)$ be the limit of
$\Phi^{n_k'}_{T_k'}$ along some subsequence $\{n_k',T_k'\}\subset\{n_k,T_k\}$. 
Recall the definition of $r_{j}$ in \cref{E-rj}.
Since $r_{j}$ is bounded below, taking limits, we obtain
$\Hat\uppi(r_{j})\;\le\; \updelta_j$, $j\in\cJ$.
Thus,  by Lemmas~3.3--3.5 and Theorems~3.1--3.2 in \cite{AP15},
optimality implies that
$\Hat\uppi(r_{\mathsf{o}})\ge\varrho^{*}_{\mathsf{c}}$. 
As the proof of \cref{T3.1}, we obtain,
\begin{equation*}
\liminf_{k\to\infty}\;J_{\mathsf{o}}\bigl(\Hat{X}^{n_{k}'}(0),Z^{n_{k}'}\bigr)
\;\ge\;\Hat\uppi(r_{\mathsf{o}})\;\ge\;\varrho^{*}_{\mathsf{c}} \,.
\end{equation*}
This proves the lower bound. 
\end{proof}

\begin{proof}[Proof of the upper bound]
Let $\epsilon>0$ be given.
By \cref{C7.1}, there exists a continuous
precise control $v_{\epsilon} \in \Ussm$ and constants
$\updelta^{\epsilon}_j<\updelta_j$, $j\in\cJ$, satisfying \cref{EL7.1A}, and
\begin{equation*}
\uppi_{v_{\epsilon}}(r_{\mathsf{o}})\le \varrho^*_{\mathsf{c}}+\epsilon\,,
\quad \text{and} \quad
\uppi_{v_{\epsilon}}(r_j)\le \updelta^{\epsilon}_j\,, \quad \forall j\in\cJ \,.
\end{equation*}

For the $n^{\rm th}$ system, we construct a  Markov scheduling policy $Z^n$
as in the proof of the upper bound of \cref{T3.1}, by concatenating
$z^n[v_\epsilon]$ and $z\in\sZ^n$ in \cref{D-SDPf}.

Following the proof of part (i) and choosing $\eta$ small enough,
i.e.,
$\eta< \epsilon\wedge\frac{1}{2}\min\{\updelta_j-\updelta^{\epsilon}_j,\;j\in\cJ\}$,
we obtain
\begin{align*}
\limsup_{n\to\infty}\;J_{\mathsf{o}}(\Hat{X}^n(0),Z^n) &\;\le\;
\varrho^{*}_{\mathsf{c}} + 2\epsilon\,,\\[5pt]
\limsup_{n\to\infty}\;J_{\mathsf{c},j}\bigl(\Hat{X}^{n}(0), Z^{n}\bigr)
&\;\le\; \frac{1}{2}(\updelta_j+\updelta^{\epsilon}_j) \,, \quad j \in\cJ\,.
\end{align*}
This completes the proof of the upper bound.
\end{proof}

\subsection{Proof of Theorem~\ref{T4.2}}

\begin{proof}[Proof of the lower bound]
The proof follows along the same lines as that of \cref{T3.2},
with the only difference that we use \cref{P6.2} instead of
\cref{P6.1} to assert tightness of the ergodic occupation measures.
With $\Hat{\uppi}$ as given in that proof, we have
\begin{equation}\label{PT4.2A}
\liminf_{k\to\infty}\;J_{\mathsf{o}}\bigl(\Hat{X}^{n_{k}}(0),Z^{n_{k}}\bigr)
\;\ge\;\Hat\uppi(r_{\mathsf{o}})\,.
\end{equation}
We then obtain
\begin{equation}\label{PT4.2B}
(\uptheta_j-\epsilon) \Hat\uppi(\Bar{r})\;\le\; \Hat\uppi(r_j)
\;\le\; (\uptheta_j+\epsilon) \Hat\uppi(\Bar{r})\qquad\forall\,j\in\cJ\,,
\end{equation}
by \cref{ET4.2A} and the uniform integrability of
\begin{equation*}
\frac{1}{T}\;
\Exp^{Z^{n_k}} \left[\int_{0}^{T}
\bigl(\Hat{Y}^{n}_j(s)\bigr)^{\Tilde{m}}\,\D{s}\right]\,,\qquad j\in\cJ\,,
\end{equation*}
which is asserted by \cref{C6.1}.
The proof is then completed using \cref{PT4.2A,PT4.2B} and
\cref{T4.1}\,(e).
\end{proof}

\begin{proof}[Proof of the upper bound]
This also follows along the lines as that of \cref{T3.2}. 
For the limiting diffusion control problem, by Theorem~5.7 and Remark~5.1
in \cite{AP15},
for any $\epsilon>0$,
there exists a continuous
precise control $v_{\epsilon} \in \Ussm$ for (P3$'$)
satisfying \cref{EL7.1A} and
\begin{equation}\label{PT4.2C}
\uppi_{v_{\epsilon}}(r_{\mathsf{o}})\;\le\; \varrho^*_{\mathsf{f}} +\epsilon\,,
\qquad\text{and}\quad
\uppi_{v_{\epsilon}}(r_{j})\;=\;\uptheta_j\,\uppi_{v_{\epsilon}}(\Bar{r})\,,
\quad j\in\cJ\,. 
\end{equation}
In addition, we have 
\begin{equation*}
\inf_{\epsilon\in(0,1)}\; \uppi_{v_{\epsilon}}(\Bar{r})>0\,.
\end{equation*}
This follows from observing that
$\{\uppi_{v_{\epsilon}}\,,\;\epsilon\in(0,1)\}$ is tight,
and $(e\cdot x)^-$ is strictly positive on an open subset of $B_1$, and from applying
 the Harnack inequality for the density of the invariant probability
measure of the limiting diffusion. 

For the $n^{\rm th}$ system, we construct a  Markov scheduling policy $\mathring{z}^n$
as in the proof of the upper bound of \cref{T3.2}, and obtain
\begin{equation}\label{PT4.2D}
\begin{split}
\limsup_{n\to\infty}\;J_{\mathsf{o}}(\Hat{X}^n(0),\mathring{z}^n) &\;\le\;
\varrho^{*}_{\mathsf{f}} + \epsilon\,,\\[5pt]
\lim_{n\to\infty}\;
J_{\mathsf{c},j}\bigl(\Hat{X}^{n}(0),\mathring{z}^n\bigr)
&\;=\; \uppi_{v_{\epsilon}}(r_{j})\,,\qquad j \in\cJ\,.
\end{split}
\end{equation}
The result then follows by \cref{PT4.2C,PT4.2D}, thus
completing the proof.
\end{proof}

\section{Conclusion} \label{S8}

In this work as well as in \cite{ABP14, AP15, AP16}, we have studied ergodic
control problems for multiclass multi-pool networks in the H--W regime under the
hypothesis that at least one abandonment parameter is positive. 
The key technical contributions include (i) the development of a new framework of
ergodic control (unconstrained and constrained) of a broad class of diffusions,
(ii) the stabilization of the limiting diffusion and the diffusion-scaled
state processes, and (iii) the technique to prove asymptotic optimality involving
a spatial truncation and concatenation of scheduling policies that are stabilizing. 
The methodology and theory can be potentially used to study ergodic control
of other classes of stochastic systems. 

There are several open problems that remain to be solved. 
First, in this work, we have identified a class of BQBS stable networks as discussed
in \cref{S4}. It will be interesting to find some examples of network models in which
the boundedness of the queueing process would not imply the boundedness of the
state process. 
Second, we have studied the networks with at least one positive abandonment parameters.
It remains to study the networks with no abandonment.
The challenges lie in understanding the stability properties of both the
limiting diffusions and the diffusion-scaled state processes. 
It is worth noting that the existence of a stabilizing control asserted
in \cref{T2.1}, which is established via the
leaf elimination algorithm in \cite{AP15},
depends critically on the assumption that
at least one abandonment parameter is positive.
Although the proof of exponential ergodicity of the BSPs also relies
on that assumption,
this property is expected to hold
with certain positive safety staffing for at least one server pool 
when all abandonment rates are zero.

\section*{Acknowledgements}
This research was supported in part by the Army Research Office through
grant W911NF-17-1-0019, and in part by the National Science Foundation through
grants  DMS-1715210 and DMS-1715875.
In addition, the work of Ari Arapostathis was supported in part by the Office of Naval
Research through grant N00014-14-1-0196.



\begin{thebibliography}{10}

\bibitem{AAM}
Z.~Aksin, M.~Armony, and V.~Mehrotra.
\newblock The modern call center: a multi-disciplinary perspective on
  operations management research.
\newblock {\em Prod. Oper. Manag.}, 16:665--688, 2007.

\bibitem{AIMMTYT}
M.~Armony, S.~Israelit, A.~Mandelbaum, Y.~N. Marmor, Y.~Tseytlin, and
  G.~Yom-Tov.
\newblock Patient flow in hospitals: a data-based queueing-science perspective.
\newblock {\em Stoch. Syst.}, 5(1):146--194, 2015.

\bibitem{CY01}
H.~Chen and D.~D. Yao.
\newblock {\em Fundamentals of queueing networks. Performance, asymptotics, and
  optimization}, volume~46 of {\em Applications of Mathematics}.
\newblock Springer-Verlag, New York, 2001.

\bibitem{GKM03}
N.~Gans, G.~Koole, and A.~Mandelbaum.
\newblock Telephone call centers: Tutorial, review and research prospects.
\newblock {\em Manuf. Serv. Oper. Manag.}, 5:79--141, 2003.

\bibitem{Hall}
R.~W. Hall.
\newblock {\em Patient Flow: Reducing Delay in Healthcare}.
\newblock Springer, New York, 2010.

\bibitem{KY14}
F.~Kelly and E.~Yudovina.
\newblock {\em Stochastic Networks}.
\newblock Cambridge University Press, 2014.

\bibitem{MMR}
A.~Mandelbaum, W.~A. Massey, and M.~I. Reiman.
\newblock Strong approximations for {M}arkovian service networks.
\newblock {\em Queueing Systems Theory Appl.}, 30(1-2):149--201, 1998.

\bibitem{PY11}
G.~Pang and D.~D. Yao.
\newblock Heavy-traffic limits for a many-server queueing network with
  switchover.
\newblock {\em Adv. in Appl. Probab.}, 45(3):645--672, 2013.

\bibitem{Shietal}
P.~Shi, M.~Chou, J.G. Dai, D.~Ding, and J.~Sim.
\newblock Models and insights for hospital inpatient operations: Time-dependent
  ed boarding time.
\newblock {\em Management Science}, 62(1):1--28, 2016.

\bibitem{TX13}
J.~N. Tsitsikilis and K.~Xu.
\newblock Queueing system topologies with limited flexibility.
\newblock {\em Proceedings of ACM SIGMETRICS}, 41(1):167--178, 2013.

\bibitem{WW02a}
W.~Whitt.
\newblock Stochastic models for the design and management of customer contact
  centers: some research directions.
\newblock Technical report, Columbia University, 2002.

\bibitem{WW02}
W.~Whitt.
\newblock {\em Stochastic-process limits. An introduction to stochastic-process
  limits and their application to queues}.
\newblock Springer Series in Operations Research. Springer-Verlag, New York,
  2002.

\bibitem{Williams16}
R.~J. Williams.
\newblock Stochastic processing networks.
\newblock {\em Annual Review of Statistics and Its Application}, 3:323--345,
  2016.

\bibitem{ABP14}
A.~Arapostathis, A.~Biswas, and G.~Pang.
\newblock Ergodic control of multi-class ${M/M/N+M}$ queues in the
  {H}alfin-{W}hitt regime.
\newblock {\em Ann. Appl. Probab.}, 25(6):3511--3570, 2015.

\bibitem{AP15}
A.~Arapostathis and G.~Pang.
\newblock Ergodic diffusion control of multiclass multi-pool networks in the
  {H}alfin-{W}hitt regime.
\newblock {\em Ann. Appl. Probab.}, 26(5):3110--3153, 2016.

\bibitem{AP16}
A.~Arapostathis and G.~Pang.
\newblock Infinite horizon average optimality of the {N}-network queueing model
  in the {H}alfin-{W}hitt regime.
\newblock {\em ArXiv e-prints}, 1602.03275, 2016.

\bibitem{stolyar-yudovina-12a}
A.~L. Stolyar and E.~Yudovina.
\newblock Tightness of invariant distributions of a large-scale flexible
  service system under a priority discipline.
\newblock {\em Stoch. Syst.}, 2(2):381--408, 2012.

\bibitem{stolyar-yudovina-12b}
A.~L. Stolyar and E.~Yudovina.
\newblock Systems with large flexible server pools: instability of ``natural"
  load balancing.
\newblock {\em Ann. Appl. Probab.}, 23(5):2099--2183, 2012.

\bibitem{stolyar-15}
A.~L. Stolyar.
\newblock Diffusion-scale tightness of invariant distributions of a large-scale
  flexible service system.
\newblock {\em Adv. in Appl. Probab.}, 47(1):251--269, 2015.

\bibitem{Atar-05a}
R.~Atar.
\newblock A diffusion model of scheduling control in queueing systems with many
  servers.
\newblock {\em Ann. Appl. Probab.}, 15(1B):820--852, 2005.

\bibitem{Atar-05b}
R.~Atar.
\newblock Scheduling control for queueing systems with many servers: asymptotic
  optimality in heavy traffic.
\newblock {\em Ann. Appl. Probab.}, 15(4):2606--2650, 2005.

\bibitem{Atar-09}
R.~Atar, A.~Mandelbaum, and G.~Shaikhet.
\newblock Simplified control problems for multiclass many-server queueing
  systems.
\newblock {\em Math. Oper. Res.}, 34(4):795--812, 2009.

\bibitem{GW09}
I.~Gurvich and W.~Whitt.
\newblock Queue-and-idleness-ratio controls in many-server service systems.
\newblock {\em Math. Oper. Res.}, 34(2):363--396, 2009.

\bibitem{GW09b}
I.~Gurvich and W.~Whitt.
\newblock Scheduling flexible servers with convex delay costs in many-server
  service systems.
\newblock {\em Manuf. Serv. Oper. Manag.}, 11(2):237--253, 2009.

\bibitem{GW10}
I.~Gurvich and W.~Whitt.
\newblock Service-level differentiation in many-server service system via
  queue-ratio routing.
\newblock {\em Oper. Res.}, 58(2):316--328, 2010.

\bibitem{dai-tezcan-08}
J.~G. Dai and T.~Tezcan.
\newblock Optimal control of parallel server systems with many servers in heavy
  traffic.
\newblock {\em Queueing Syst.}, 59(2):95--134, 2008.

\bibitem{dai-tezcan-11}
J.~G. Dai and T.~Tezcan.
\newblock State space collapse in many-server diffusion limits of parallel
  server systems.
\newblock {\em Math. Oper. Res.}, 36(2):271--320, 2011.

\bibitem{Armony-05}
M.~Armony.
\newblock Dynamic routing in large-scale service systems with heterogeneous
  servers.
\newblock {\em Queueing Syst.}, 51(3-4):287--329, 2005.

\bibitem{AW-10}
M.~Armony and A.~R. Ward.
\newblock Fair dynamic routing in large-scale heterogeneous-server systems.
\newblock {\em Oper. Res.}, 58(3):624--637, 2010.

\bibitem{WA-13}
A.~R. Ward and M.~Armony.
\newblock Blind fair routing in large-scale service systems with heterogeneous
  customers and servers.
\newblock {\em Oper. Res.}, 61(1):228--243, 2013.

\bibitem{Biswas-15}
A.~Biswas.
\newblock An ergodic control problem for many-sever multi-class queueing
  systems with help.
\newblock {\em ArXiv e-prints}, 1502.02779v2, 2015.

\bibitem{gamarnik-stolyar}
D.~Gamarnik and A.~L. Stolyar.
\newblock Multiclass multiserver queueing system in the {H}alfin-{W}hitt heavy
  traffic regime: asymptotics of the stationary distribution.
\newblock {\em Queueing Syst.}, 71(1-2):25--51, 2012.

\bibitem{dieker-gao}
A.~B. Dieker and X.~Gao.
\newblock Positive recurrence of piecewise {O}rnstein-{U}hlenbeck processes and
  common quadratic {L}yapunov functions.
\newblock {\em Ann. Appl. Probab.}, 23(4):1291--1317, 2013.

\bibitem{APS17}
A.~Arapostathis, G.~Pang, and N.~Sandri{\'c}.
\newblock Ergodicity of {L}{\'e}vy-driven {SDE}s arising from multiclass
  many-server queues.
\newblock {\em ArXiv e-prints}, 1707.09674, 2017.

\bibitem{stolyar-14}
A.~L. Stolyar.
\newblock Tightness of stationary distributions of a flexible-server system in
  the {H}alfin--{W}hitt asymptotic regime.
\newblock {\em Stoch. Syst.}, 5(2):239--267, 2015.

\bibitem{williams-2000}
R.~J. Williams.
\newblock On dynamic scheduling of a parallel server system with complete
  resource pooling.
\newblock {\em Analysis of Communication Networks: Call Centres, Traffic and
  Performance. Fields Inst. Commun. Amer. Math. Soc., Providence, RI.},
  28:49--71, 2000.

\bibitem{kallenberg}
O.~Kallenberg.
\newblock {\em Foundations of Modern Probability}.
\newblock Probability and its Applications. Springer-Verlag, New York, second
  edition, 2002.

\bibitem{Rock46}
R.~T. Rockafellar.
\newblock {\em Convex Analysis}, volume~28 of {\em Princeton Mathematical
  Series}.
\newblock Princeton University Press, Princeton, New Jersey, 1946.

\bibitem{Douc-09}
R.~Douc, G.~Fort, and A.~Guillin.
\newblock Subgeometric rates of convergence of {$f$}-ergodic strong {M}arkov
  processes.
\newblock {\em Stochastic Process. Appl.}, 119(3):897--923, 2009.

\bibitem{Hairer-16}
M.~Hairer.
\newblock Convergence of {M}arkov {P}rocesses.
\newblock {L}ecture {N}otes, University of Warwick, 2016.
\newblock Available at http://www.hairer.org/notes/Convergence.pdf.

\end{thebibliography}
\end{document}